\documentclass[twoside,11pt]{article}

\usepackage{graphicx, subfigure}
\usepackage[top=1in,bottom=1in,left=1in,right=1in]{geometry}
\usepackage[numbers,sort&compress]{natbib} 
\usepackage{amsfonts, amsmath, amssymb, amsthm, bbm}
\usepackage[normalem]{ulem}
\usepackage{comment}
\usepackage{eufrak}
\usepackage{mathrsfs} 
\usepackage[active]{srcltx}

\usepackage{algorithm}
\usepackage{algorithmic}

\usepackage{color}
\definecolor{darkred}{RGB}{100,0,0}
\definecolor{darkgreen}{RGB}{0,100,0}
\definecolor{darkblue}{RGB}{0,0,150}

\usepackage{hyperref}
\hypersetup{colorlinks=true, linkcolor=darkred, citecolor=darkgreen, urlcolor=darkblue}
\usepackage{url}

\newtheorem{prp}{Proposition}
\newtheorem{lem}{Lemma}
\newtheorem{cor}{Corollary}

\def\beq{\begin{equation}}
\def\eeq{\end{equation}}
\def\beqn{\begin{eqnarray*}}
\def\eeqn{\end{eqnarray*}}
\def\bitem{\begin{itemize}}
\def\eitem{\end{itemize}}
\def\benum{\begin{enumerate}}
\def\eenum{\end{enumerate}}
\def\bmult{\begin{multline*}}
\def\emult{\end{multline*}}
\def\bcenter{\begin{center}}
\def\ecenter{\end{center}}

\newcommand{\ol}[1]{\overline{#1}}

\DeclareMathOperator*{\argmax}{arg\, max}
\DeclareMathOperator*{\argmin}{arg\, min}

\def\cA{\mathcal{A}}
\def\cB{\mathcal{B}}
\def\bC{\mathcal{C}}

\def\cE{\mathcal{E}}
\def\cF{\mathcal{F}}

\def\cI{\mathcal{I}}

\def\cL{\mathcal{L}}

\def\bN{\mathcal{N}}

\def\cP{\mathcal{P}}

\def\cS{\mathcal{S}}
\def\cT{\mathcal{T}}
\def\cU{\mathcal{U}}

\def\bB{\mathbf{B}}
\def\bC{\mathbf{C}}

\def\bE{\mathbf{E}}

\def\bI{\mathbf{I}}

\def\bN{\mathbf{N}}

\def\bP{\P}

\def\bU{\mathbf{U}}

\def\bY{\mathbf{Y}}

\def\ba{\mathbf{a}}

\def\bt{\mathbf{t}}
\def\bu{\mathbf{u}}

\def\1{{\mathbf 1}}

\newcommand{\btau}{{\boldsymbol\tau}}

\newcommand{\bmu}{{\boldsymbol\mu}}

\newcommand{\btheta}{{\boldsymbol\theta}}

\newcommand\bepsilon{{\boldsymbol\epsilon}}

\newcommand\bPi{{\boldsymbol\Pi}}
\newcommand\wh{\widehat}

\def\bbE{\mathbb{E}}

\def\bbL{\mathbb{L}}

\def\bbP{\mathbb{P}}

\def\bbR{\mathbb{R}}

\newcommand{\ac}[1]{\left\{ \left. #1 \right. \right\}}
\newcommand{\abs}[1]{\left\lvert #1 \right\rvert} 

\newcommand{\E}{\operatorname{\mathbb{E}}}
\renewcommand{\P}{\operatorname{\mathbb{P}}}

\def\pen{\mathrm{pen}}

\def\Cr{\mathrm{Cr}}

\DeclareMathOperator{\sign}{sign}

\newcommand{\nv}[1]{\textcolor{magenta}{[Nicolas: #1]}}

\newcommand{\mf}[1]{\textcolor{blue}{ #1}}
\title{Optimal Change-Point Detection and Localization}
\author{Nicolas Verzelen, Magalie Fromont, Matthieu Lerasle, and Patricia Reynaud-Bouret}
\begin{document}

\maketitle

\begin{abstract}
Given a times series ${\bf Y}$ in $\mathbb{R}^n$, with a piece-wise contant mean and independent components, the twin problems of change-point detection and change-point localization respectively amount to detecting the existence of times where the mean varies and estimating the positions of those change-points.  In this work, we tightly  characterize optimal rates for both problems and uncover the phase transition phenomenon from a global testing problem to a local estimation problem. Introducing a suitable definition of the energy of a change-point, we first establish in the single change-point setting  that the optimal detection threshold is $\sqrt{2\log\log(n)}$. When the energy is just above the detection threshold, then the problem of localizing the change-point becomes purely parametric: it only depends  on the difference in means and not on the position of  the change-point anymore. Interestingly, for most change-point positions, including all those  away from the endpoints of the time series, it is possible to detect and localize them at a much smaller energy level. In the multiple change-point setting, we establish the energy detection threshold and show similarly that the optimal localization error of a specific change-point becomes purely parametric. Along the way, tight optimal rates for Hausdorff and $l_1$ estimation losses of the vector of all change-points positions are also established. 
Two procedures achieving these optimal rates are introduced. The first one is a least-squares estimator with a new multiscale penalty that favours well spread change-points. The second one is a two-step multiscale post-processing procedure whose computational complexity can be as low as $O(n\log(n))$. Notably, these two procedures accommodate with the presence of possibly many low-energy and therefore undetectable change-points and are still able to detect and localize high-energy change-points even with the presence of those nuisance parameters.
\end{abstract}

\section{Introduction}\label{sec:intro}

Following a long historical line of work, starting with Wald \cite{Wald1945}, Girshick and Rubin \cite{Girshick1952}, Page \cite{Page1954}, and Fisher \cite{Fisher1958} in the 1940-1950's, leading to a huge bibliography including several prominent monographs such as \cite{Basseville1993,Carlstein1994,Csorgo1997,Brodsky2013,Brodsky2013bis,Tartakovsky2014}, and which is still vivid (e.g.~\cite{fryzlewicz2020detecting,cho2019localised,wang2020}), we consider the prototypical problem of univariate change-point analysis.  Let  $\bY=(Y_1,\ldots, Y_n)$ be a time series with values in $\mathbb{R}^n$, with unknown mean vector $\btheta=(\theta_1,\ldots,\theta_n)$ in $\mathbb{R}^n$. Change-point analysis amounts to studying possible variations in the mean vector $\btheta$. 
In this work, we focus our attention on the independent observation setting 
\beq\label{eq:univariate_model}
Y_i= \theta_i + \epsilon_i\enspace ,  \quad i=1,\ldots, n\enspace,
\eeq
where the \emph{noise} random vector $\bepsilon=(\epsilon_1,\ldots,\epsilon_n)$ is made of independent mean-zero random variables satisfying a sub-Gaussian condition 
\[
 \E[e^{t\epsilon_i}]\leq e^{\sigma^2 t^2/2} \textrm{ for all } t \text{ in }\mathbb{R} ,\quad \text{ for }i=1,\ldots,n\enspace ,
\]
where $\sigma$ is known. By homogeneity and standardization, we assume henceworth that $\sigma=1$. Our general objective is twofold: first, we carefully analyze the intrinsic difficulty of several detection and localization problems, thereby closing long-standing  gaps between the early asymptotic results and recent non-asymptotic ones. This allows us to define desirable specifications for a change-point estimation method. Second, we introduce two procedures achieving these specifications.

\subsection{Model and change-point procedures}

Since we view the mean vector $\btheta$ as piece-wise constant vector, we can define it through the change points.  Given $\btheta$ in $\mathbb{R}^n$, there exists an integer $0\leq K\leq n-1$, a  vector of integers $\btau^*=(\tau^*_1,\ldots,\tau^*_K)$ satisfying  
$1=\tau^*_0< \tau^*_1<\ldots<\tau^*_K<\tau_{K+1}^*=n+1$, a vector $\bmu=(\mu_1,\ldots,\mu_{K+1})$ in $\bbR^{K+1}$ satisfying  $\mu_k\ne\mu_{k+1}$ for all $k$ in $\{1,\ldots,K\}$ such that 
\beq\label{eq:definition_theta}
\theta_i = \sum_{k=1}^{K+1}\mu_k  \1_{\tau^*_{k-1}\leq i<  \tau^*_{k}}
\enspace.
\eeq
Then, $\tau^*_k$ is called the \emph{position} of the $k$-th change-point (or sometimes for the sake of simplicity the $k$-th change-point) and $\Delta_k= \mu_{k+1}-\mu_k$ is called the \emph{height} of the $k$-th change-point in $\btheta$. 
It follows from~\eqref{eq:definition_theta} that $\btheta$ is uniquely defined by $\btau^*$ and $\bmu$. As a consequence, one may easily deduce an estimator of $K$, $\btau^*$, and $\bmu$ from an estimator of $\btheta$. Conversely, any estimator $\widehat{\btau}$ with length $\widehat{K}$ of the change-points positions leads to an estimator $\widehat{\btheta}$ by simply plugging the empirical mean on the corresponding partition, that is $\widehat{\theta}_i= \sum_{k=1}^{\widehat{K}+1}\1_{\widehat{\tau}_{k-1}\leq i\leq \widehat{\tau}_{k}-1}(\widehat{\tau}_{k}-\widehat{\tau}_{k-1})^{-1}  [\sum_{j= \widehat{\tau}_{k-1}}^{\widehat{\tau}_{k}-1} Y_i ] $ for $i=1,\ldots, n$. Hence, any change-point estimation method may be indifferently interpreted as an estimator of $\btheta$ and $\btau^*$.

\medskip 
In general, the number $K$ of change-points is supposed to be unknown. 
Still, for historical and mathematical reasons to be discussed below, the literature usually divides into two settings: the {\it single change-point} setting\footnote{It is sometimes referred as at most one change-point setting in the literature.} where one assumes that  $K \leq 1$ and the  {\it multiple change-point} setting where $K$ is possibly arbitrarily large. In the next paragraphs, we provide a short account of classical  approaches for multiple change-point estimation, and then we explain the connections between those approaches in the single change-point setting. In the sequel, we write $\Theta_K\subset \bbR^n$ for the collection  of vectors $\btheta$ with $K$ change-points exactly.

\subsection{Main approaches for change-point analysis}
It is beyond the scope of this paper to give an exhaustive survey of existing methods and we refer the interested reader to~\cite{Niu2016,truong2020selective}. Still, we can roughly divide most procedures into two general categories: the ones based on minimization of (penalized) least-squares criteria and the ones based on test statistics, mostly related to the CUSUM statistic.

 \paragraph{Penalized least-squares minimization}
 If the number $K$ of change-points is known and if the noise vector follows an homoscedastic Gaussian distribution, then the maximum likelihood estimator of $\btheta$ is equal to the least-squares estimator  $\widehat{\btheta}_{LS,K}= \arg\min_{\btheta'\in \Theta_K}\sum_{i=1}^n (Y_i-\theta'_i)^2$. While $\Theta_K$ is not convex, it has been observed in the pioneering work of Bellman~\cite{bellman1961approximation} that $\widehat{\btheta}_{LS,K}$ can be efficiently computed by a dynamic programming algorithm, whose worst-case complexity is $O(n^2)$ operations. In practice, $K$ is unknown and it is natural to select it by adding a penalty term, such as BIC~\cite{YaoAu1989} or more involved forms of penalties~\cite{Lebarbier2005} arising from the general theory of model selection~\cite{birge2001gaussian}. On the computational aspect, the quadratic complexity turns out to  be prohibitive for some large-scale  problems arising for instance in genomics. To address this issue, Killick et al.~\cite{killick2012optimal} (see also~\cite{rigaill2010pruned}) have developed pruning techniques that accelerate the dynamic programming algorithm and lead to a quasi-linear time complexity in favorable situations.
 
 More generally, penalized least-squares criteria of the form $\arg\min_{\btheta'\in \Theta_K}\sum_{i=1}^n (Y_i-\theta'_i)^2+ \pen(\btheta')$ may involve penalty terms $\pen(\btheta')$ that do not only depend on the number $K$ of change-points of $\btheta'$. 
 For instance, Zhang and Siegmund~\cite{zhang2007modified} argue for a penalty $\pen(\btheta')$  that also depends on the spacing between the change-points. Alternatively, choosing     
 $\pen(\btheta')$ proportional to the total variation norm  $\|\btheta'\|_{TV}= \sum_{i=2}^n |\theta'_i-\theta'_i|$ corresponds to the Fused Lasso estimator~\cite{tibshirani2005sparsity}.  Efficient solvers compute the Fused Lasso estimator in quasi-linear time~\cite{hoefling2010path}.

\paragraph{Binary segmentation, CUSUM, and multiscale methods.}

Beside penalty-based approaches, the other broad class of methods is based on the CUSUM statistics.  In the sequel, $\cT_3$ refers to the collection of triads, that is the set of all triplets of integers $\bt=(t_1,t_2,t_3)$ such that $1\leq t_1< t_2< t_3\leq  n+1$. Given $\bt=(t_1,t_2,t_3)$ in $\cT_3$, the CUSUM statistic at $\bt$ is defined as the weighted difference of empirical means on $[t_1,t_2)$ and $[t_2,t_3)$
\beq\label{eq:definition_cusum}
\bC(\bY,\bt)=  \left[\frac{\sum_{i=t_2}^{t_3-1}Y_i}{t_3-t_2}- \frac{\sum_{i=t_1}^{t_2-1}Y_i}{t_2-t_1}\right]\sqrt{ \frac{(t_2-t_1)(t_3-t_2)}{t_3-t_1}}\enspace .
\eeq
For homoscedastic Gaussian noise, this  statistic also corresponds to likelihood ratio test statistic of the null hypothesis {\it \{$\btheta$ is constant over $[t_1,t_3)$\}} versus {\it \{$\btheta$ has one change-point at $t_2$\}}. Under the null hypothesis, $\bC(\bY,\bt)$ follows a standard normal distribution. In principle, one then could adopt a multiple-testing perspective and apply all tests based on all CUSUM statistics. However, there are two main caveats with this naive approach. First, $|\cT_3|$ is of the order of $n^3/6$ which can lead to a prohibitive $O(n^3)$ computational complexity. Second, there is no straightforward way of transforming a collection of $|\cT_3|$ $p$-values into a single change-point estimator $\widehat{\btau}$. This is why most earlier CUSUM-based change-point procedures follow greedy approaches. 
CUSUM statistics have a long history in the single change-point literature. Originally, Hinkley~\cite{hinkley1970} maximizes the   CUSUM $\bC({\bf Y}, (1,t,n+1))$ over all possible change-point positions $t=2,\ldots, n+1$. Binary Segmentation (BS)~\cite{scott1974cluster} algorithm for multiple change-points detection amounts to recursively cut the time series into two parts by applying Hinkley's method at each step. However, BS does not consistently estimate the change-points which lead 
to the introduction of many variants of BS including Wild Binary Segmentation~\cite{fryzlewicz2014wild,fryzlewicz2018tail,Wang2018,wang2020,kovacs2020seeded,fryzlewicz2020narrowest}. Some of these procedure exhibit a quasi-linear time complexity~\cite{fryzlewicz2018tail}. 
 While the connection is less clear, other methods based on moving sums such as MOSUM~\cite{eichinger2018mosum} may also be interpreted as an aggregation procedure of the CUSUM statistics with symmetric windows $\bt=(\tau-h,\tau,\tau+h)$ where $1\leq \tau-h< \tau+h\leq n+1$. An important aspect of many of these procedures is that they consider CUSUM~\eqref{eq:definition_cusum} statistics at many different scales $(t_3-t_1)$ thereby being able to detect close change-points with large heights and change-points whose heights are small but which are distant from any other change-point.
 \medskip 
 
 In the single change-point setting where the statistician knows that $K=1$, the estimator $\widehat{\tau}$ maximizing the 
CUSUM equals the change-point of  the least-squares estimator $\widehat{\btheta}_{LS,1}$, so that both procedures are equivalent. Hence, the distinction between  CUSUM-based and least-squares-based procedures is not clear in that setting.

\paragraph{Post-processing methods.} The change-points estimator $\widetilde{\btau}$ produced by one of the previous methods is sometimes refined  in a second step by removing spurious estimated change-points~\cite{fryzlewicz2018tail} or/and improving the precision of the estimated change-points positions by a local CUSUM maximization (see e.g. Sect.3.2 in~\cite{fryzlewicz2014wild}). See also Lin et al.~\cite{lin2016approximate}  for another recent post-processing method.

\subsection{Statistical Problems}

There are several ways of assessing the quality of a change-point procedure, whose choice mainly depends on the question of interest. As mentioned earlier, one can easily deduce an estimator of $\btau^*$ from an estimator $\btheta$ (and conversely). However, $\btau^*$ is not a continuous function of $\btheta$ and  a near perfect estimator of $\btheta$ does not necessarily lead to a precise estimation of the number of change-points. We can broadly summarize the statistical objectives into three classes, that are detailed below.
\begin{enumerate}
  \item[(a)] {\bf Signal Denoising}. Here, one mainly aims at estimating the mean $\btheta$ in $\mathbb{R}^n$ from ${\bf Y}$ taking into account the side information that $\btheta$ is  a piece-wise constant vector. 

 \item[(b)] {\bf Change-points Detection}. The objective  is now to detect the existence of change-points in $\btheta$. It is easier to state this problem in the single change-point setting where $K\leq 1$. Indeed, this boils down to testing the hypothesis $\{\btheta\in \Theta_0\}$ (there is no change-point) versus $\{\btheta\in \Theta_1\}$ (there is exactly one change-point). Obviously, detecting one change-point is feasible only if its height $\Delta_1=\mu_2-\mu_1$ is not too small (in absolute value). As a consequence, a suitable detection procedure should have a small type I error probability as well as be able to detect with high probability a change-point whose height is not too small. 
 We shall further formalize this problem in the next subsection. For multiple change-points problems, one can possibly interpret the problem of change-points detection as that of estimating the number $K$ of change-points~(e.g. \cite{frick2014multiscale,fryzlewicz2014wild,fryzlewicz2020detecting}). Unfortunately, stating this as an estimation problem of the functional $K$ may hide the fact that, in a vector $\btheta$, some change-points may be easier to detect than some others. One of our contribution in this work is to propose an alternative formalization of this problem.   
 
 \item[(c)] {\bf Change-point Localization}. Again, we start with the single change-point setting where $\btheta$ is in $\Theta_1$. One aims at building an estimator $\widehat{\tau}$ of $\tau_1$  such that $|\widehat{\tau}-\tau_1|$ is the smallest possible. Intuitively, localization is only feasible if the change-point has been detected so that some minimal assumption on the height has to be done. In the multiple change-points setting, there are several ways of measuring the localization error that depend whether one is interested into estimating a specific change-point, a subset of significant change-points or the whole set of change-points.  Given a vector $\btau$ with length $|\btau|$ and with coordinates $(\tau_1,\ldots,\tau_{|\btau|})$ and one change-point $\tau'$, let  the distance $d_{H,1}(\btau,\tau')=\min_{i=1,\ldots, |\btau|}|\tau_i-\tau'|$ between $\tau'$ and its closest element in $\btau$. Hence, for any $1\leq k\leq K$,  $d_{H,1}(\widehat{\btau},\tau^*_k)$ quantifies to what extent $\widehat{\btau}$ estimates well the true change-point $\tau^*_k$. We refer to this as a \emph{point-wise} loss. Besides, we define 
 \beq\label{eq:definition_Haussdorf}
 d_{H,1}(\btau,\btau^*)=\max_{j=1,\ldots, |\btau^*|}d_{H,1}(\btau,\tau^*_j)\enspace ; \quad\quad d_H(\btau,\btau^*)=\max \left\{d_{H,1}(\btau,\btau^*), d_{H,1}(\btau^*,\btau)\right\}  \enspace,
\eeq
which respectively correspond to  the screening distance and Hausdorff distance between the sets $\{\tau^*_i,\ i=1,\ldots, |\tau^*|\}$  and $\{\tau_i,\ i=1,\ldots, |\tau|\}$ (see \cite{lin2016approximate} for instance). 
For simplicity, we will call them the screening distance from $\btau^*$ to $\btau$, and the Hausdorff distance between $\btau^*$ and $\btau$, thereby  confusing  vectors and sets of their values. In some sense, the Hausdorff distance $d_H(\btau,\btau^*)$ should be understood as an {\it uniform} error measure between $\btau^*$ and $\btau$.
Finally, when the change-points vectors $\btau$ and $\btau^*$ have the same length ($|\btau|=|\btau^*|$), one considers
\beq\label{eq:definition_wasserstein}
d_W(\btau, \btau^*)= \sum_{j=1}^{|\btau|} |\tau_j-\tau^*_j| \enspace,
\eeq
that corresponds, up to the re-normalization by $|\btau|$, to the $\bbL_1$-Wasserstein distance between the empirical probability measures associated with the vectors. This can be interpreted an {\it $l_1$ -loss} and is referred henceforth to as the Wasserstein distance between $\btau^*$ and $\btau$.

\end{enumerate}

In this manuscript, we mainly focus on the two latter problems where the statistician is more interested on change-points in themselves than on the signal $\btheta$. Before explaining our contributions, we summarize how these testing and estimation problems are considered in the literature as well as the best known bounds.

\subsection{State of the art}

Although we are not specifically interested in signal denoising, we briefly discuss the literature as this viewpoint falls into the well established field of nonparametric statistic and is well understood. In particular, Gao et al.~\cite{Gao2020} have studied the optimal (in the minimax sense) risk  $\E[\|\widehat{\btheta}-\btheta\|^2]$ achievable by any estimator when $\btheta$ belongs to $\Theta_K$. For $K\geq 2$, the optimal risk is (up to a multiplicative numerical constant) of the order of $K\log(2n/K)$ (see also \cite{Lebarbier2005,Arlot2019}) whereas, for $K=1$, the optimal risk is qualitatively different and is of the order of $\log\log(16n)$. All these bounds are achieved by least-squares estimators $\widehat{\btheta}_{LS,K}$~\cite{Gao2020}.  The penalized least-squares estimator of Lebarbier~\cite{Lebarbier2005} also achieves such bounds for all $K\geq 2$ without requiring the knowledge of $K$. By constrast,  the Fused LASSO achieves similar near-optimal risk bounds for $K\geq 2$  if we further restrict our attention to evenly spaced change-points~\cite{guntuboyina2020adaptive}. Unfortunately, for some other $\btheta$ in $\Theta_K$ with non-even spaced change-points, the risk of the Fused LASSO is much large; see Theorem 4.1 in~\cite{fan2018approximate}. 

In fact, most works dedicated to such penalized least-squares procedures for signal estimation provide, as in \cite{Gao2020}, theoretical risk bounds for the targeted mean vector estimation and then, empirically evaluate their performances in terms of change-points detection and localization error (see e.g.~\cite{Lebarbier2005,Harchaoui2010,lin2016approximate,Arlot2019}). Intuitively, near optimal signal estimation risk bounds should suggest that the procedures detect most change-points and localize them well. While this heuristic argument is appealing, tentative formalizations of it lead to quite pessimistic results~\cite{lin2016approximate,Arlot2019} both in terms of detection and localization.

\medskip 

Since the difficulty of signal denoising qualitatively changes between the single change-point  and the multiple change-points settings, the remainder of the literature review is organized according to these two settings.

\paragraph{Detection of a single change-point.}
As mentioned in the previous section, we need to formalize the significance of a single change-point $\tau_1^*$. Given $\btheta
$ in $\Theta_1$, we define its energy 
\beq\label{def:energy1}
\bE_1(\btheta)=|\Delta_1| \sqrt{\frac{(\tau_1^*-1)(n+1-\tau_1^*)}{n}}\enspace .
\eeq
The  energy $\bE_1(\btheta)$ is equal to the  $l_2$ distance between $\btheta$ and $\Theta_0$ and thereby quantifies the difficulty of assessing the existence of $\tau_1^*$. As an alternative to the energy $\bE_1(\btheta)$, some authors~(e.g.\cite{Gao2020}) consider the quantity $\bE^2_{\min}(\btheta)= \Delta_1^2 [(\tau_1^*-1)\wedge (n+1-\tau_1^*)]$, where $x\wedge y$ stands for the minimum of $x$ and $y$.  Both quantities are equivalent ($\bE^2_1(\btheta)\leq   \bE^2_{\min}(\btheta)\leq 2\bE_1^2(\btheta)$), but we use $\bE_1(\btheta)$ as it is more intrinsic to the change-point problem. Testing the null hypothesis  \{$\btheta\in \Theta_0$\} versus  \{$\btheta \in \Theta_1$\} has been extensively studied since the seminal work of~\cite{Page1955, Page1957, Hawkins1977}. Cs\"org\"o and Horv\'ath~\cite[][eq.(3.5.22)]{Csorgo1997} have proved that the test rejecting the null for large values of  scan CUSUM statistic $\max_{2\leq t \leq n}|\bC({\bf Y};(1,t,n+1))|$ is asymptotically powerful when $\bE_1(\btheta)/\sqrt{\log\log(n)}\rightarrow \infty$. Recently,  Gao et al.~\cite{Gao2020} established a non-asymptotic counterpart of this result as well as a minimax lower bound, stating that no test is able to reject the null with high probability simultaneously for all $\btheta$ in $\Theta_1$ such that $\bE_1(\btheta)\geq c_1 \sqrt{\log\log(16n)}$ where $c_1$ is a small numerical constant. 

\paragraph{Localization of a single change-point.} Recall that, in this setting, both the least-squares estimator and the max CUSUM estimator are equal and are referred to  $\widehat{\tau}$ in this paragraph. Early work from~\cite{hinkley1970, hinkleyhinkley1970} considered an asymptotic setting where $\Delta_1$ is constant and established that $|\widehat{\tau}- \tau_1^*|= O_P(1)$. Later, D\"umbgen~\cite{dumbgen1991} and Cs\"org\"o and Horv\'ath~\cite{Csorgo1997} have worked out the asymptotic distribution of $|\widehat{\tau}- \tau_1^*|$ under the assumption that the change-point energy is high-enough so that  $\bE_1(\btheta)/\sqrt{\log\log(n)}\rightarrow \infty$ and under the restriction that $\tau_1^*$ is proportional to $n$. In particular, one deduces from this asymptotic distribution that $|\widehat{\tau}- \tau_1^*|= O_P(1/\Delta_1^2)$. When $\Delta_1^2=O(1)$, this bound is minimax optimal (see Proposition 10 in ~\cite{Wang2018}). Up to our knowledge, there are few non-asymptotic localization results. Still, we can easily deduce from Gao et al~\cite{Gao2020} that, with positive probability, $|\widehat{\tau}- \tau_1^*|\leq c \log\log(n)/\Delta_1^2$ as long as $\bE_1(\btheta)\geq c' \sqrt{\log\log(n)}$, where $c$ and $c'$ are positive constants.

\paragraph{Detection for multiple change-points.} Early works typically considered an asymptotic setting where $\theta_i= g(i/n)$ and $g$ is a fixed $(K+1)$-step function defined on $[0,1]$. Notably, Yao and Au~\cite{YaoAu1989} proved that the least-squares estimator with a BIC penalty selects a number $\widehat{K}$ converging in probability to $K$. To formalize the detection problem in a non-asymptotic setting, and define the significance of each change-point, we extend the notion of change-point energy  $\bE_1(\btheta)$ for $\btheta$ in $\Theta_1$.  Given a vector $\btheta$, with $K$ change-points and an integer $1\leq k\leq K$, we define the energy $\bE_k(\btheta)$ of the $k$-th change-point by 
\beq\label{eq:defintion_Ek}
\bE_k(\btheta)= |\Delta_k| \sqrt{\frac{(\tau_{k+1}^*-\tau_k*)(\tau_k^* -\tau^*_{k-1}) }{\tau^*_{k+1}- \tau_{k-1}^*}}\ ,
\eeq
so that this matches \eqref{def:energy1} when $k=K=1$. We show later that $\bE^2_k(\btheta)$ equals the squared $l_2$ distance between $\btheta$ and its best approximation by a vector $\btheta'$ with $(K-1)$ change-points $(\tau^*_1,\ldots, \tau^*_{k-1}, \tau^*_{k+1},\ldots, \tau^*_K)$. As in the single change-point case, one easily checks that $\bE_k^2(\btheta)$ is equivalent (up to a factor 2) to $\Delta_k^2 [(\tau_{k+1}^*-\tau^*_k)\wedge (\tau_k^* -\tau^*_{k-1})]$. The non-asymptotic counterpart of the detection problem would then correspond to establishing minimal condition on $\bE^2_k(\btheta)$ so that a change-point procedure detects $\widehat{K}=K$ change-points with high probability. As an alternative to $\min_{1\leq k\leq K}\bE^2_k(\btheta)$, most recent papers in the literature (see e.g. \cite{frick2014multiscale,fryzlewicz2018tail,Wang2018,wang2020}) use the following quantity
\beq\label{eq:definition_energie_min}
\bE_{min}(\btheta)=\left[\min_{1\leq k\leq K} |\Delta_k|\right] \left[\min_{0\leq k\leq K} |\tau_{k+1}^*-\tau_{k}^*|^{1/2}\right]\enspace.
\eeq
When all $|\Delta_k|$ are equal and all change-points $\tau^*_k$ are equi-spaced, then $\bE^2_{min}(\btheta)\geq \min_{k}\bE^2_k(\btheta)$ but $\bE^2_{min}(\btheta)$ is possibly much smaller than $\min_{k}\bE^2_k(\btheta)$ for heterogeneous jumps. In most of the modern analyses of change-point procedures~\cite{frick2014multiscale,fryzlewicz2018tail,wang2020}, authors usually aim at establishing that $\widehat{K}=K$ with high probability as long as $\bE^2_{\min}(\btheta)\geq c\log(n)$ (for some $c>0$). Such a detection property is achieved by some penalized least-squares procedures~\cite{wang2020} and CUSUM-based procedures (e.g.~\cite{baranowski2019narrowest,wang2020}).
Notably, Frick et al.~\cite{frick2014multiscale} prove that their SMUCE procedure consistently estimates $K$ under  the tighter condition $\bE^2_{\min}(\btheta)\geq c\log(n/\min_{k}(\tau^*_{k+1}-\tau_k^*))$ which is of the order of $\log(K+1)$ when there are few equi-spaced change-points.  Conversely, Chan and Walther \cite{Chan2013} (see~\cite{arias2005near} for earlier results) have established in the particular setting of change-point detection that no procedure is able to consistently estimate $K$ unless $\min_{k}\bE_k^2(\theta) \geq 2\log[n/\min_{k}(\tau^*_{k+1}-\tau_k^*)\wedge (\tau^*_{k}-\tau_{k-1}^*)]$.

\paragraph{Localization of several change-points.} In their asymptotic setting where $\theta_i=g(i/n)$ for a fixed function $g$, Yao and Au~\cite{YaoAu1989} established the asymptotic distribution of $(\widehat{\tau}_k-\tau^*_k)$, $k=1,\ldots, K$ where the vector $\widehat{\btau}$ is that of the least-squares estimator $\widehat{\btheta}_{LS,K}$. The differences $\widehat{\tau}_k-\tau^*_k$
are thus proved to be asymptotically independent and to have limiting sub-exponential  distributions with parameter $1/\Delta_k^2$. This result was later extended by Bai and Perron~\cite{Bai1998} and Lavielle and Moulines~\cite{lavielle2000least} to asymptotic settings where $K$ is still fixed but the $|\Delta_k|$'s are allowed to converge to zero. In the non-asymptotic setting, Frick et al.~\cite{frick2014multiscale}, Wang et al.~\cite{wang2020}, Baranowski et al.~\cite{baranowski2019narrowest} established that, as soon as $\bE^2_{\min}(\btheta)\geq c \log(n)$, the estimator $\widehat{\btau}$ satisfies $d_H(\widehat{\btau}, \btau^*)\leq c'\log(n)/[\min_{k}\Delta_k^2]$ with high probability. It follows from the analysis of the single change-point case that this upper bound is optimal up to a (possible) $\log(n)$ term. In an almost concomitant but independent work, Cho and Kirch~\cite{cho2019localised} have recently proved that their multiscale MOSUM procedure satisfies the tighter bound $d_H(\widehat{\btau}, \btau^*)\leq c'\log(K+1)/[\min_{k}\Delta_k^2]$.

\subsection{Our contribution}

We now describe our main results.
\paragraph{Pinpointing minimal conditions for change-point detection.} In the single change-point setting, we establish that the uniform detection threshold for the energy $\bE_1(\btheta)$  is at $\sqrt{2\log\log(n)}$. Importantly, change-points $\tau_1^*$ that are away from the endpoints of the time series can be detected at energy level $\bE_1(\btheta)$ of the order $\sqrt{2\log\log({16n}/(\tau_1^*\wedge (n+1-\tau_1^*)))}$, which can be as small as a constant for $\tau^*_1$ proportionnal to $n$. Regarding the multiple change-points setting, we introduce a new way of assessing the detection of change-points. An estimator $\widehat{\btau}$ is said to detect $\tau_k^*$ if some estimated change-point $\widehat{\tau}_l$ belongs to $[(\widehat{\tau}_k^*+\widehat{\tau}_{k-1}^*)/2,(\widehat{\tau}_k^*+\widehat{\tau}_{k+1}^*)/2)$. Conversely, $\widehat{\btau}$ does not detect any spurious change-point if the interval $[(\widehat{\tau}_k^*+\widehat{\tau}_{k-1}^*)/2,(\widehat{\tau}_k^*+\widehat{\tau}_{k+1}^*)/2)$ contains at most one estimated change-point $\widehat{\tau}_l$. Any change-point $\tau^*_k$ is detectable  by a procedure $\widehat{\btau}$ whose probability of estimating spurious change-point is small  as soon as 
 \beq\label{eq:high_ener}
 \bE^2_k(\theta)\geq  c\log\left( \tfrac{c'n}{(\tau^*_{k+1}-\tau_k^*)\wedge (\tau^*_{k}-\tau_{k-1}^*)}\right)\ , 
 \eeq
thereby matching (up to constants) the minimax rate for the simpler problem of segment detection~\cite{Chan2013}. What matters here is that we prove such detectability results in settings where we allow an arbitrarily large number of change-points to have a low energy. This implies that the presence of many arbitrarily small change-points does not make high-energy change-points much harder to detect. 

\paragraph{Transition from a global to a local estimation problem.} As soon as a change-point $\tau^*_k$ is detectable (as its energy is above the appropriate threshold, then the change-point $\tau^*_k$ can be estimated at a sub-exponential rate with scale $\Delta_k^2$. As a consequence, the error rate for estimating a specific change-point is purely local and neither depends on its energy (as long as it is high enough) nor on $n$. Our non-asymptotic analysis bridges the gap between the asymptotic expansions of~\cite{Csorgo1997} and \cite{lavielle2000least} and known non-asymptotic bounds. Besides, we recover that the respective positions estimation errors of two high-energy change-points behave like nearly independent variables~\cite{YaoAu1989, lavielle2000least}. In turn, this allows us to establish tight risk optimality results with respect to both the Hausdorff~\eqref{eq:definition_Haussdorf}  and  Wasserstein~\eqref{eq:definition_wasserstein} distances. Finally, we note that this {\it global to local transition phenomenon}  also occurs in the presence of multiple low-energy change-points.

\paragraph{Penalized least-squares estimation with a multiscale penalty.} We introduce two multiple change-points procedures achieving all these optimality properties. The first one is a penalized least-squares type estimator with a multiscale penalty that promote equi-spaced change-points positions. As the corresponding penalty is additive, this estimator is easily computed by (pruned) dynamic programming~\cite{killick2012optimal}. In contrast to the BIC-type penalty studied recently in~\cite{wang2020}, this allows us to recover the optimal logarithmic terms as well as to handle settings where low-energy change-points are present. 
 
 \paragraph{Optimal post-processing procedure based on aggregation of CUSUM tests.} As an alternative to the penalized least-squares estimator, we promote a two-step method based on the aggregation of many CUSUM tests. This method can either serve as a self-standing procedure or as a post-processing procedure to improve the detection and localization properties of a preliminary estimator. It is shown to satisfy the same optimality property as the previous multiscale procedure, whereas its computational complexity can be taken as low as $n\log(n)$ operations. 
 
 \bigskip 
 
 From a technical perspective, our main results are based on a novel simultaneous control of all CUSUM statistic $\bC({\bf Y};{\bf t})$ in Lemma~\ref{lem:concentration:N_t} that can be of independent interest. Despite the fact that we introduce as least minimax formalism as possible, our viewpoint and arguments heavily borrow from the literature on minimax testing separation rates~\cite{Ingster1993I,baraud02}.

\subsection{Notation and organization of the paper}

In the sequel, we use bold letters for vectors (e.g. $\btheta, \bmu, \bY,\btau, \ldots$), whereas  $\|.\|$ stands for the Euclidean norm. For any finite set $S$, we write $|S|$ for its cardinality.
As usual, we denote by $\Phi$ the probability distribution function of a standard univariate normal distribution. Besides, $\overline{\Phi}$ stands for the corresponding tail distribution function.\\
In the sequence, $c$, $c'$, $c_1$ stand for positive numerical constant whose value may change from line to line. Given some quantity $L$, we write $c_L$ for a positive function that only depends on $L$.
 We write $u \lesssim v$ (resp. $ u \gtrsim v $) when there exists a numerical constant $c>0$ such that $u\leq cv$ (resp. $u\geq c v$). The notation $u\asymp v$ means that both $u\lesssim v$ and $u\gtrsim v$.
 Given $x$ in $\mathbb{R}$, we write $x_+=\max(x,0)$ for the positive part of $x$, and $\lfloor x\rfloor$ (resp. $\lceil x \rceil$) for the largest (resp. smallest) integer smaller than (resp. larger than) or equal to $x$.
 Given two real numbers $x$ and $y$, $x\vee y$ (resp. $x\wedge y$) denotes the maximum (resp. minimum) value between $x$ and $y$. When it is clearer, we sometimes use $\max(x,y)$ and $\min(x,y)$.\\ Given a vector $\btheta$,  we write $\P_{\btheta}$ for the distribution of ${\bf Y}$. 

\medskip 

Section~\ref{sec:one_change_point} is dedicated to testing and estimation problems when $\btheta$ contains at most one single change-point.  Turning to the multiple change-points problem, we establish impossibility results in Section~\ref{sec:lower_multiple}, which lead us to defining desiderata for a suitable change-point procedure. In Section~\ref{sec:least_squares}, we establish that penalized least-squares change-point procedures with a multiscale penalty achieve the desired specifications, whereas we prove similar results for a post-processing procedure based on the CUSUM statistic in Section~\ref{sec:post_proc}. Some extensions and open problems are discussed in Section~\ref{sec:discussions}. The proofs are postponed to the end of the manuscript.

\section{Single change-point analysis}\label{sec:one_change_point}
This section deals with the case where the mean vector $\btheta$ contains at most one change-point, that is when $\btheta$ belongs to $\Theta_0\cup\Theta_1$.  The corresponding model \eqref{eq:univariate_model} has often been called the At Most One Change (AMOC) model in the change-point literature (see  e.g. \cite{Csorgo1997}). 
To alleviate the notation, when $\btheta$ belongs to $\Theta_1$, we simply write $\tau^*$ for $\tau^*_1$ and $\Delta$ for $\Delta_1=\mu_2-\mu_1$ in this section and the corresponding proofs.
Recall (see \eqref{def:energy1}) that the energy $\bE_1(\btheta)$ of the  change-point is  defined by
\[
\bE_1(\btheta)= |\Delta|\sqrt{\frac{(\tau^*-1)(n+1-\tau^*)}{n}}\enspace.
\]
As explained in the introduction, the change-point detection problem when $\btheta$ belongs to $ \Theta_0\cup \Theta_1$ is formalized as the problem of testing the null hypothesis $(H_0)\ \{\btheta\in\Theta_0\}$ versus the alternative $(H1)\  \{\btheta\in \Theta_1\}$, 
while the change-point localization problem is treated as a problem of estimation of $\tau^*$ when $\btheta$ is assumed to belong to $\Theta_1$.

We first state impossibility results in form of minimax lower bounds and then show that matching upper bounds can be obtained from test statistics and estimators based on penalized least-squares criteria. More precisely,  we derive tight lower and upper bounds for the detection rate with the tight constants. Finally, we built optimal  confidence intervals for $\tau^*$. 

\subsection{Impossibility results for the detection and localization problems}\label{sec:single:lower}

The impossibility results are established in the specific setting where the noise is Gaussian. Hence, we assume throughout this subsection that $\boldsymbol{\epsilon}\sim \mathsf{N}(0,\bI_n)$. This assumption is henceforth referred to as $(\mathcal{A}_G)$.

\subsubsection{Minimax lower bound for the detection problem}\label{sec:lower:single:test}

Considering the problem of testing $(H_0)\ \{\btheta\in\Theta_0\}$ versus $(H1)\  \{\btheta\in \Theta_1\}$, we want to assess the minimal energy $\bE_1(\btheta)$ required so that a test is able to reject the null with high probability. We start with a simple and known observation, see e.g.~\cite{enikeeva2018high}.
Fix an integer $\tau$ in $\{2,\ldots,n\}$ and a real number $\delta \ne 0$, then $\Theta_1$ contains 
\[
\Theta_{1}[\tau,\delta]=\{\btheta\in \Theta_1: \tau^*=\tau,\; \Delta=\delta\}\enspace.
\]
Testing  $(H_0)$ versus the alternative $(H_{1,\tau,\delta})$ $\{\btheta\in\Theta_{1}[\tau,\delta]\}$ is arguably simpler than testing $(H_0)$ versus $(H_1)$ and the minimal energy requirement for rejecting with high probability is therefore smaller. 
The level-$\alpha$ Likelihood Ratio Test (LRT) for this simple (but unrealistic) testing problem rejects $(H_0)$ in favor of the alternative $(H_{1,\tau,\delta})$ when the CUSUM statistic $\bC(\bY;(1, \tau,n+1))$ with
\[
\bC(\bY,(1,\tau,n+1))=\left(\frac{\sum_{i=\tau}^nY_i}{n+1-\tau}- \frac{\sum_{i=1}^{\tau-1}Y_i}{\tau -1}\right) \left(\frac{1}{\tau -1}+ \frac{1}{n+1-\tau}\right)^{-1/2}\enspace,
\]
has an absolute value larger than a critical value depending on $\alpha$. Since under the null hypothesis $(H_0)$, $C(\bY,(1,\tau,n+1))$ follows a standard normal distribution, this critical value can be taken equal to $t_\alpha=\overline{\Phi}^{-1}(\alpha/2)$. Besides, since under the alternative $(H_{1,\tau,\delta})$, $\bC(\bY,(1,\tau,n+1))$ has a $\mathsf{N}(b_{n,\tau},1)$ distribution with a bias $b_{n,\tau,\delta}$ such that $|b_{n,\tau,\delta}|=\bE_1(\btheta)$, the type II error probability of the LRT equals  $\overline{\Phi}(|b_{n,\tau,\delta}|-t_\alpha) +\overline{\Phi}(|b_{n,\tau,\delta}|+t_\alpha)$. This probability is small only if $\bE_1(\btheta)$ is large enough (compared to the critical value $t_\alpha$). From the fundamental Neyman-Pearson Lemma, it is the smallest possible type II error probability for any $\alpha$-level test. As a consequence, noticing that when $\alpha$ is small, $\overline{\Phi}^{-1}(\alpha)$ is of the order of $\sqrt{2\log(1/\alpha)}$, a necessary condition for a change-point to be reliably detected by a level-$\alpha$ test is that its energy is  large compared to $\sqrt{2\log(1/\alpha)}$.

In a more realistic setting where the position and the height of the change-point are unknown to the statistician, that is when considering the initial problem of testing $(H_0)$ versus $(H_1)$, a slightly higher energy is necessary for a change-point to be reliably detected, as formalized  in the next proposition.

\begin{prp}\label{prp:lower_test_one_change-point}
Assume that $\bepsilon$ in~\eqref{eq:univariate_model} satisfies ($\mathcal{A}_G$).
There exist positive numerical constants $c$ and $n_0$ such that for all $\kappa$ in $(0,2/3)$, for all $n\geqslant n_0$, and for any test $\mathscr{T}$ of $(H_0)$ versus $(H_1)$, one has
\beq\label{eq:lower_one_change}
\sup_{\btheta\in \Theta_0}\P_{\btheta}[\mathscr{T}=1] + \sup_{\btheta\in \Theta_1,\ \bE_1(\btheta)\geqslant \sqrt{2(1-\kappa)(1-n^{-1/2})\log \log (n)}}\P_{\btheta}[\mathscr{T}=0]\geqslant 1- c \left(\log (n)\right)^{-\frac{\kappa^2}{8(1-\kappa)}}\enspace . 
\eeq
\end{prp}
The lower bound~\eqref{eq:lower_one_change} implies that the sum of the type I and II error probabilities for vectors $\btheta$ with $\bE_1(\btheta)\geqslant \sqrt{2(1-\kappa)(1-n^{-1/2})\log \log (n)}$ is close to one. This means that no test performs better than random guess. 
Let us interpret this proposition in an asymptotic setting. Taking e.g. $\kappa=\log^{-1/3}(n)$, we deduce from~\eqref{eq:lower_one_change} that
any level-$\alpha$ test $\mathscr{T}$ of $\ac{\btheta\in \Theta_0}$ versus $\ac{\btheta\in \Theta_1}$ is not able to detect  a change-point with energy $\sqrt{2(1-o(1))\log\log(n)}$ with probability uniformly  higher than $1-\alpha-o(1)$. Adopting the separation rate terminology of the minimax testing literature, this enforces that the energy minimax separation rate is of the order of $\sqrt{2\log\log(n)}$. We establish in the next subsection that the leading  constant $2$ is tight.

 Gao et al.~\cite{Gao2020} have established a similar impossibility result in an asymptotic framework but with a suboptimal leading constant $c < \sqrt{2}$.  In an independent work, Han et al.~\cite{liu2019minimax} have recently considered the counterpart of the single change-point detection problem for multivariate time series $\bY$ with values in $\mathbb{R}^p$. Letting both $p$ and $n$ go to infinity, they provide a sharp characterization of the minimal energy for change-point detection. 

\medskip

Closely examining the proof of Proposition~\ref{prp:lower_test_one_change-point}, we see that the result of \eqref{eq:lower_one_change} is still valid even if we restrict our attention to change-point positions $\tau^*$ that are smaller than $\sqrt{n}$. Intuitively, this $\sqrt{2\log\log(n)}$ price rather arises because there are $\log(n)$ possible order of magnitudes for $\tau^*$. We come back to this point below when we build an optimal test.

\subsubsection{Minimax lower bound for the localization problem}\label{sec:lower:loc:one}

Let us now focus on a minimax lower bound for the localization problem, viewed as a problem of estimation of the true change-point position $\tau^*$ of $\btheta$, once it is assumed to belong to $\Theta_1$.
The arguments provided below are standard and can be found e.g. in~\cite{Wang2018}. They are repeated here for the sake of completeness and because a slightly tighter version than~\cite{Wang2018} is required to handle the multiple change-points case. 

Throughout this subsection, we use the notation $\btheta(\tau^*,\bmu)$ to stress the dependency of $\btheta$ on these two quantities.

\begin{lem}\label{lem:LBCPE1}
Assume that $\bepsilon$ in~\eqref{eq:univariate_model} satisfies ($\mathcal{A}_G$). 
Let $\bmu=(\mu_1,\mu_2)$ denote a couple of real numbers such that $\mu_1\neq\mu_2$, and $\Delta=\mu_2-\mu_1$. 
There exist positive constants $c$ and $ c'$ such that, for $n>4\Delta^{-2}$ and for any  $x$ in $[1/2,n/2-1-2\Delta^{-2})$, 
\[
\inf_{\widehat\tau}\sup_{\tau^*\in\{2,\ldots,n\}}\bbP_{\btheta(\tau^*,\bmu)} \left(|\widehat{\tau}-\tau^*|\geq 2\Delta^{-2}+x\right)\geqslant ce^{-c'x\Delta^2}\enspace,
\]
where the infimum is taken over all possible estimators $\widehat{\tau}$ of $\tau^*$.
\end{lem}

The probability $\mathbb{P}_{\btheta(\tau^*,\bmu)} (|\widehat{\tau}- \tau^*|\geq 4\Delta^{-2}+x)$ is therefore at best exponential in $x$ with rate $\Delta^2$, which in particular leads to the following lower bound for the minimax risk:
$$\inf_{\widehat\tau}\sup_{\tau^*\in\{2,\ldots,n\}}\bbE_{\btheta(\tau^*,\bmu)} \left[|\widehat{\tau}-\tau^*|\right]\geqslant \frac{c}{ \Delta^{2}} \enspace,$$
where $\bbE_{\btheta(\tau^*,\bmu)}$ denotes the expectation with respect to the probability distribution $\bbP_{\btheta(\tau^*,\bmu)}$. As a consequence, no estimator can estimate $\tau^*$ at a rate smaller than $1/\Delta^2$ when the mean values vector $\bmu=(\mu_1,\mu_2)$ is fixed, with difference $\Delta$. 
\subsection{Optimal detection and localization by penalized least-squares}\label{sec:least_one}

In this section, we construct a change-point detection procedure and a change-point estimator both based on the classical  penalized least-squares minimization approach.  We do not restrict the model to the Gaussian assumption $(\cA_G)$ anymore, and we therefore consider an observed random vector $\bY$ such that \eqref{eq:univariate_model} holds with $\btheta$ in $\Theta_0\cup \Theta_1$, whose probability distribution is still denoted by $\P_{\btheta}$.

Let $L>0$ and define for any $\tau$ in $\{2,\ldots,n\}$, the following  penalized least-squares criterion 
\beq\label{eq:definition_criterion_penalized}
\Cr_1(\bY, \tau):= \|\bY - \bPi_{\tau}\bY\|^2+  L \pen_1(\tau)\enspace,
\eeq  
where $\bPi_{\tau}$ denotes the orthogonal projection onto the linear subspace of $\Theta_0\cup\Theta_1$, composed of vectors $\btheta$ in $\bbR^n$ with a single change-point $\tau$. The penalty term $\pen_1(\tau)$, chosen as
\beq\label{eq:multiscale_penalty}
\pen_1(\tau)= 2\log\log\left(e\max\left\{\left(\tau\wedge  \frac{n}{\tau}\right),\left( (n+1-\tau)\wedge\frac{n}{n+1-\tau} \right)\right\}\right)\enspace,
\eeq
is of multiscale type. It is worth noticing that $\pen_1(\tau)$ is bounded from above by  $2\log\log (en)=2\log \log (n)+o_n(1)$ which roughly corresponds to the squared lower bound for the minimax detection rate obtained in Section \ref{sec:lower:single:test}.  Nevertheless, this upper  bound is sometimes  pessimistic. In particular, for $\tau\asymp 1$, $n-\tau\asymp 1$, or  $\tau \asymp n$,  $\pen_1(\tau)$ is of the order of a constant.

 \subsubsection{Detection}\label{sec:CPDTest}

Let us come back to the problem of detecting the existence of a change-point, that is of  testing  $(H_0)\ \{\btheta\in\Theta_0\}$ versus $(H_1)\ \{\btheta\in \Theta_1\}$.
Let the penalty parameter $L$ be in $(1,2]$ and consider the test statistic:
\beq\label{eq:definition:stat_global_1_change-point}
T(\bY) =\min_{\tau\in \{2,\ldots,n\}} \left\{- \|(\bPi_{\tau}-\bPi_{0})\bY\|^2 + L^2\pen_1(\tau)\right\}=\min_{\tau\in \{2,\ldots,n\}} \Cr_1(\bY, \tau) - \|\bY\|^2 +  \|\bPi_{0}\bY\|^2 \enspace. 
\eeq
For $\alpha$ in $(0,1)$, we can now introduce the test $\mathscr{T}_\alpha$ defined  by 
\beq\label{eq:definition:test_global_1_change-point}
\mathscr{T}_\alpha=\1_{T(\bY)\leq -L^2(C_\alpha+C_L)}\enspace,\quad \textrm{with} \quad C_\alpha=6\log \left(\frac{12}{\alpha}\right)\quad \textrm{and}\quad C_L=\frac2{L}\log\left(\frac{L}{L-1}\right)-2\log\log (L)\enspace.
\eeq
The test $\mathscr{T}_{\alpha}$ can be interpreted as the aggregation of a collection of tests of $\{\btheta \in \Theta_0\}$ versus $\{\btheta\in \Theta_1 \textrm{  with }\tau^*= \tau\}$  over $\tau=2,\ldots, n$. As for the estimation procedure introduced above, the threshold $-L^2\pen_1(\tau)-  L^2(C_\alpha+C_L)$ depends on $\tau$, thereby giving a multiscale taste to the procedure. One  deduces from simple linear algebra that $\|(\bPi_{\tau}-\bPi_{0})\bY\|^2= \bC^2(\bY;(1,\tau,n+1))$. Hence,  $\mathscr{T}_{\alpha}$ interprets as a max penalized CUSUM statistic, with position-dependent penalties $\pen_1(\tau)$. This position-dependent penalization approach differs from the usual max (non-penalized) CUSUM testing procedure in the literature~\cite{Csorgo1997}.

\begin{prp}\label{prp:test_one_change-point}
For any $\alpha$ in $(0,1)$, the test $\mathscr{T}_\alpha$ of $\{\btheta\in\Theta_0\}$ versus $\{\btheta\in \Theta_1\}$ defined by \eqref{eq:definition:stat_global_1_change-point} and \eqref{eq:definition:test_global_1_change-point} with $L$ in $(1,2]$, is of level $\alpha$. Moreover, for any $\beta$ in $(0,1)$, and any $\btheta$ in $\Theta_1$ such that
 \beq\label{eq:condition_lower_one_change-point}
 \bE^{2}_1(\btheta) \geqslant L^3\pen_1(\tau^*) + \frac{2L}{L+1}\log\left(\frac{2}{\beta}\right)+L^3(C_\alpha+C_L) \enspace,
 \eeq
 $\mathscr{T}_\alpha$ has type II error probability  smaller than $\beta$.
 \end{prp}
Choosing $L$ close to one in Condition \eqref{eq:condition_lower_one_change-point} and letting $n$ go to infinity entail that a change-point can be detected with high probability if its energy $\bE_1(\btheta)$ is higher than $(1+o(1))\sqrt{2\log\log(n)}$.
This matches the impossibility result of Proposition~\ref{prp:lower_test_one_change-point} and means that $\mathscr{T}_\alpha$ is minimax optimal. In the minimax separation rate formalism, Proposition ~\ref{prp:lower_test_one_change-point} and Proposition \ref{prp:test_one_change-point} together imply that the energy minimax separation rate is equivalent to $\sqrt{2\log\log(n)}$. In the fact, the same rate  $\sqrt{2\log\log(n)}$ was already obtained by  \cite{enikeeva2018high,liu2019minimax}  using the max CUSUM statistic, which is equivalent to setting $\pen_1(\tau)$ to $0$ in our test. However, Proposition~\ref{prp:test_one_change-point} is much more optimistic than~\cite{enikeeva2018high,liu2019minimax}. Although we pay a $\sqrt{2\log\log(n)}$ price for $\tau^*= n^{\zeta}$ with $\zeta$ in $(0,1)$, which is unavoidable, the energy requirement~\eqref{eq:condition_lower_one_change-point} is  weaker at some other positions.  In particular, it is of constant order when either $\tau^*\asymp n$ or $\tau^*\asymp 1$. In other words, for most change-point positions, the requirement $\bE_1(\btheta)> \sqrt{2\log\log(n)}$ can be mitigated. This phenomenon is due to careful choice $\pen_1(\tau)$ of the CUSUM statistic in the test $\mathscr{T}_{\alpha}$. Regarding considerations on the LRT of $(H_0)$ versus $(H_{1,\tau,\delta})$ in Section~\ref{sec:lower:single:test}, the requirement that $\bE_1(\btheta)$ is large compared to $\sqrt{\log(1/\alpha)}$ (involved in $C_\alpha$) is also unavoidable.

\subsubsection{Localization}

When $\btheta$ is assumed to belong to $\Theta_1$ and the energy $\bE_1(\btheta)$ is large enough, we have proved above that the change-point $\tau^*$ can be reliably detected, so that one can aim at localizing  $\tau^*$. 
Consider the estimator $\widehat{\tau}$ minimizing among all $\tau$ in $\{2,\ldots, n\}$ the penalized least-squares criterion $\Cr_1(\bY, \tau)$ introduced in \eqref{eq:definition_criterion_penalized} with the multiscale penalty \eqref{eq:multiscale_penalty}.

\begin{prp}\label{prp:tau_hat} Fix the tuning parameter $L$ in $(1,2]$ and consider  the  minimizer $\widehat{\tau}$ of $\Cr_1(\bY, \tau)$. There exist positive constants $c$, $c_L$, and $c'_L$ such that the following holds for all $n\geqslant c'_L$. 
If
\beq\label{eq:condition_signal}
\bE^{2}_1(\btheta)>L^2\pen_1(\tau^*) + c_L\enspace,
\eeq
then, for any $x$ in $(0, (\bE^{2}_1(\btheta)- L^2\pen_1(\tau^*)-c_L)/c_L)$, with probability larger than $1-32e^{-x}$,
\beq\label{eq:error_bound_localization}
|\widehat{\tau}-\tau^*|\leq c\frac{1\vee x}{ \Delta^2} \quad \text{ and }\quad \frac{(\widehat{\tau}-1)(n+1-\tau^*)}{(\tau^*-1)(n+1-\widehat{\tau})}\in (1/2,2)\enspace.
\eeq

\end{prp}

Proposition \ref{prp:tau_hat} is proved in Section~\ref{Sec:Proof:PropPenLS1Jump}.
Condition~\eqref{eq:condition_signal} requires that the squared energy $\bE^2_1(\btheta)$ is larger than $L^2 \pen_1(\tau^*)$. 
Choosing $L$ close to $1$, we have already seen that this condition is  sharp as it corresponds to the sufficient condition for the arguably simpler problem of detecting the existence of the change-point in Proposition~\ref{prp:lower_test_one_change-point}.
Thus, the above result entails than once a change-point is detectable, it can also be estimated by $\widehat{\tau}$ with a rate of order $\Delta^{-2}$, which is optimal as discussed in Section~\ref{sec:lower:loc:one}. In fact, the tail distribution of the error is optimal. Indeed, Proposition \ref{prp:tau_hat} implies that $|\widehat{\tau}-\tau^*|\lesssim \Delta^{-2}+x$ with probability higher than $1-e^{-c'x\Delta^{2}}$, which is the optimal dependency in $x$ for this probability according to Lemma~\ref{lem:LBCPE1}.
In the specific case where the height is high ($|\Delta|\gtrsim 1$), we use that both $\widehat{\tau}$ and $\tau^*$ are integers to conclude that $\widehat{\tau}=\tau^*$ with probability higher than $1- ce^{-c'\Delta^2}$, which is again optimal by Lemma~\ref{lem:LBCPE1}.

\subsection{Confidence intervals for $\tau^*$}

To conclude the present single change-point analysis, we build a confidence interval $I_{\widehat{\tau}}$ for $\tau^*$. 
Its construction can be heuristically decomposed in two steps. First, one tests the existence of a change-point. This step relies on a statistic that slightly differs from \eqref{eq:definition:test_global_1_change-point}. The second step then depends on the conclusion of the test. If no change-point is detected, then $\tau^*$ cannot be localized and we set $I_{\widehat{\tau}}:= \{2,\ldots, n\}$.  
If a change-point has been detected, we build a confidence interval $I_{\widehat{\tau}}$ whose width is of order $\Delta^{-2}$ according to the error bound of Proposition \ref{prp:tau_hat}. As the error bound depends on the unknown height $\Delta$, this step requires to estimate $\Delta$. 

\medskip

Fix $L$ and $\kappa$ in $(1,2)$, and let
\[T_{IC}(\bY)=  \inf_{\tau\in
\{2,\ldots,n\}}  \left
\{- \|(\bPi_{\tau}-\bPi_{0})\bY\|^2 + (1+\kappa)L^2\pen_1(\tau)\right\}\enspace,
\]
where $\pen_1$ is the multiscale penalty proposed in \eqref{eq:multiscale_penalty}. Let $\widehat{\tau}$ still denote a minimizer among all $\tau$ in $\{2,\ldots,n\}$ of the criterion given by \eqref{eq:definition_criterion_penalized}. This estimator $\widehat{\tau}$ is then plugged into the basic empirical means to estimate $\Delta$ by 
\beq \label{eq:definition_widehat_Delta}
\widehat{\Delta}=  \frac{ \sum_{i=\widehat{\tau}}^{n}Y_i}{n+1-\widehat{\tau}}- \frac{\sum_{i=1}^{\widehat{\tau}-1}Y_i}{\widehat{\tau}}\enspace .
\eeq
For some numerical constants $\underline{c}$, $\underline{c}_{L,\kappa}$ that are specified in the proof of Proposition~\ref{prp:IC_one_change-point} in Section~\ref{Sec:ProofICSingleJump}, let 
\beq\label{eq:definition_IC}
 I_{\widehat{\tau}}:= 
\begin{cases}
  \left[\widehat{\tau}-\frac{ \underline{c} \log(e/\alpha)}{|\widehat{\Delta}|^2}; \widehat{\tau}+ \frac{ \underline{c}\log({e}/{\alpha})}{|\widehat{\Delta}|^2}\right] &\quad \text{if}\quad T_{IC}(\bY)< -\underline{c}_{L,\kappa}\log(e/\alpha)\enspace,\\
  [2,n]&\quad \text{otherwise}\enspace.
\end{cases}
\eeq

\begin{prp}\label{prp:IC_one_change-point}
Let $I_{\widehat{\tau}}$ denote the confidence interval introduced in~\eqref{eq:definition_IC}. 
There exist constants $\underline{c}$, $\underline{c}_{L,\kappa}$, $c_{L,\kappa}$, $c'$, and $c'_{L}$ such that the following holds for all $n\geq c'_L$. 
For any $\btheta$ in $\Theta_1$ with change-point position $\tau^*$ and change-point height $\Delta$, $$\P_\btheta \left(\tau^*\in I_{\widehat{\tau}}\right)\geqslant 1-\alpha\enspace.$$ 
Besides, if 
\beq\label{eq:condition_large_energy_IC_2}
\bE^{2}_1(\btheta)\geqslant (1+\kappa)L^2 \pen_1(\tau^*)+ c_{L,\kappa}\log({e}/{\alpha})\enspace , 
\eeq
 then, with probability larger than $1-\alpha$, one has
$$|I_{\widehat{\tau}}|\leq c'\log({e}/{\alpha})\Delta^{-2}\enspace.$$
\end{prp}
If one takes $\kappa$ and $L$ close to one in Condition \eqref{eq:condition_large_energy_IC_2},  then Proposition~\ref{prp:IC_one_change-point} shows that, if the squared energy of the change-point is larger than $\pen_1(\tau^*)$, the width of the  confidence interval is at most proportional to $\log(e/\alpha) \Delta^{-2}$ with probability larger than $(1-\alpha)$. Conversely, Lemma~\ref{lem:LBCPE1} implies that no estimator is able to localize $\tau^*$ at the precision  $c'' \log(e/\alpha) \Delta^{-2}$ with probability higher than $1-2\alpha$. As a consequence, no $(1-\alpha)$-level confidence interval $\tilde{I}$ of $\tau^*$ has a width smaller than $c'' \log(e/\alpha) \Delta^{-2}$ with probability higher than $1-\alpha$. This entails that the width of the confidence interval $I_{\widehat{\tau}}$ is  optimal both with respect to $\Delta$ and $\alpha$. 

Finally, since $\tau^*$ is an integer, note that $I_{\widehat{\tau}}$ can be reduced to the set of integers inside 
$I_{\widehat{\tau}}$. Thus, for a high-energy change-point and $|\Delta|\gtrsim 1$ we have $I_{\widehat{\tau}}=\{\tau^*\}$ with probability higher than $(1-2\alpha)$ as long as $\alpha \geq ce^{-c'\Delta^2}$.

\section{Multiple change-points ($K>1$): preliminaries and impossibility results}\label{sec:lower_multiple}

\subsection{Notation and preliminaries}

It has been underlined in the above section that the energy $\bE_1(\btheta)$ of a single change-point $\tau_1^*$ with height $\Delta_1=\mu_2-\mu_1$ in $\btheta$ defined by $\bE_1(\btheta)= |\Delta_1|\sqrt{(\tau_1^*-1)(n+1-\tau_1^*)/n}$ (see \eqref{def:energy1}) plays a crucial role in the analysis. 
Extending the single change-point analysis to the multiple change-points case necessarily poses the question of identifying the quantity which will allow to measure the difficulty for multiple change-points to be detected or localized. Most recent papers in the literature (see \cite{frick2014multiscale,Wang2018,baranowski2019narrowest,wang2020}) extend the notion of energy using $\bE_{\min}(\btheta)$ (as in \eqref{eq:definition_energie_min}). 
As our main purpose here is to give tighter guarantees of detection and localization, we use a more local notion of energy. Recall that $\cT_3$ refers to the collection of  triplets of integers $\bt=(t_1,t_2,t_3)$ such that $1\leq t_1< t_2< t_3\leq  n+1$. For $\btheta$ in $\bbR^n$ and $\bt=(t_1,t_2,t_3)$ in $\cT_3$, define the energy
\beq\label{eq:definition_energie}
\bE(\btheta,\bt)= \left| \frac{\sum_{i=t_2}^{t_3-1}\theta_{i}}{t_3-t_2} - \frac{\sum_{i=t_1}^{t_2-1}\theta_{i}}{t_2-t_1}\right|\sqrt{ \frac{(t_2-t_1)(t_3-t_2)}{t_3-t_1}}\enspace, 
\eeq
which is a weighted difference of means on $[t_1,t_2)$ and on $[t_2,t_3)$.
Given $\btheta$ in $\Theta_K$ with change-points vector  $\btau^*$ and an integer $1\leq k\leq K$, define the energy of $\btheta$ at the $k$-th change-point $\bE_{k}(\btheta)$ by 
$\bE_k(\btheta)=\bE[\btheta,(\tau^*_{k-1},\tau^*_k,\tau^*_{k+1})]$.  More simply, we refer to this quantity as the \emph{energy of the $k$-th change-point}. This extends the definition of $\bE_1(\btheta)$  when $K=1$. Note that 
$$\bE_k(\btheta)=|\Delta_k|\sqrt{ \frac{(\tau^*_{k}-\tau^*_{k-1})(\tau^*_{k+1}-\tau^*_k)}{\tau^*_{k+1}-\tau^*_{k-1}}}\enspace,$$ which is of order $\Delta_k\sqrt{\ell_k}$ (in fact $|\Delta_k|\sqrt{\ell_k/2} \leqslant \bE_k(\btheta)\leqslant |\Delta_k|\sqrt{\ell_k}$), where $\Delta_k$ denotes the height of the change-point and $\ell_k:=(\tau^*_{k}- \tau^*_{k-1})\wedge (\tau^*_{k+1}- \tau^*_{k})$ is the smallest length of the segments adjacent to the change-point, more simply named the \emph{$k$-th change-point length}.

\medskip 

Intuitively, one cannot expect to properly localize  a change-point if it cannot at least be reliably detected. 
From the previous section (see in particular Propositions \ref{prp:lower_test_one_change-point} and \ref{prp:test_one_change-point}), one knows that only  change-points whose energy is high enough are to be detected. This is why we introduce the following definition.
Given $\kappa>1$ and $q>0$, we say that $\tau^*_k$ is a \emph{$(\kappa,q)$-high-energy change-point} if 
\beq\label{eq:definition_high_energy}
\bE_k(\btheta) > \kappa \sqrt{2\log\left(\frac{n(\tau^*_{k+1}-\tau^*_{k-1})}{(\tau^*_{k+1}-\tau^*_{k})(\tau^*_{k}-\tau^*_{k-1})}\right)+ q}\enspace . 
\eeq
In other words, the energy $\bE_k(\btheta)$ is said to be high if  $\Delta^{2}_k\ell_k$ is large compared to $\log (n/\ell_k) +q$. 
Note that if the change-point length $\ell_k$ is smaller than $n^{\zeta}$ with $\zeta$ in $ [0,1)$, the logarithmic term is of the order of $\log (n)$, whereas if this length is proportional to $n$, the logarithmic term is of the order of a constant.

\subsection{Impossibility results}\label{sec:multiple:lower}

As in Section~\ref{sec:single:lower}, we assume throughout this section that the noise in the model \eqref{eq:univariate_model} is Gaussian, that is assumption ($\cA_G$) is fulfilled.  Besides, we sometimes use the $\btheta(\btau^*,\bmu)$ instead of $\btheta$ to stress the dependency of $\btheta$ on the parameters $\btau^*$ and $\bmu$.

\subsubsection{Low-energy change-points cannot be detected}\label{sec:multiple:lower:detection}

The targeted property for detecting a specific change-point among multiple change-points is formalized in Section~\ref{sec:specifications} below. In this subsection, we explain why this strong detection property is not achievable when the change-point has not a high-energy. More specifically, we consider the weaked possible notion of change-point detection, that is we test whether $(H_0)\ \{\btheta\in \Theta_0\}$ versus $(H_1)\ \{\btheta\notin \Theta_0\}$ and we establish that, for some alternatives, only vectors $\btheta$ with at least a high-energy change-point can be reliably detected.  In particular, we justify the two different quantities $q$ and $\log\left(n(\tau^*_{k+1}-\tau^*_{k-1})/((\tau^*_{k+1}-\tau^*_{k})(\tau^*_{k}-\tau^*_{k-1}))\right)$ in the definition~\eqref{eq:definition_high_energy} of high-energy change-points. Part of the results belong to the statistical folklore and asymptotic versions have already appeared for instance in~\cite{Chan2013}. Still, we state and prove here non-asymptotic counterparts to unify our presentation and for the sake of completeness.

\medskip 

We showed in Section~\ref{sec:lower:single:test} that, when testing $\{\btheta\in \Theta_0\}$ versus $\{\btheta\in \Theta_1\}$, even in the most simple situation where  the position of the tentative change-point is known, it  can only be reliably detected if its energy is large enough. More precisely, we established that any level-$\alpha$ test is only able to reject the null hypothesis with probability higher than $1-\beta$, when the energy of the change-point is at least of the order $\sqrt{\log(1/\min(\alpha,\beta))}$.  
Coming back to the definition~\eqref{eq:definition_high_energy} of high-energy change-points, we deduce that a single change-point $\tau^*_{1}$ whose length $\ell_1$ is proportional to $n$, can only be detected with probability higher than $1-ce^{-c'q}$ (for some $q>0$) if it has $(1,q)$-high energy.

\medskip 

Now, we turn to the logarithmic term in the definition~\eqref{eq:definition_high_energy} of high-energy change-points. This quantity 
only arises when there are at least two change-points. Following a long line of literature~\cite{MR2797847,arias2005near,Chan2013}, we consider the problem of segment detection. Given a positive integer $r$ (less than $\lfloor n/4\rfloor $) and $\delta>0$, we consider the collection $\Theta[r,\delta]\subset \Theta_2$ defined by 
\beq\label{eq:definition_r_delta}
\Theta[r,\delta]=\Big\{\btheta\in\bbR^n: \exists \tau\in \{\lfloor{n/4}\rfloor+1 ,\ldots,\lfloor{3n/4}\rfloor-r \}, \text{ such that }
\theta_i=\delta\1_{i\in \{\tau,\ldots,\tau+r-1\}}\Big\}\enspace. 
\eeq
A vector $\btheta$ in $\Theta[r,\delta]$ equals $\delta$ on  the segment $[\tau, \tau+r-1]$ and is null outside this segment. Note that this segment belongs to $[n/4; 3n/4]$ and  is therefore away from the endpoints. The problem of (single) segment detection is that of testing $\{\btheta=0\}$ versus $\{\btheta\in \Theta[r,\delta]\}$. Typical asymptotic results\footnote{In the literature, the segments are usually not restricted to be away from the endpoints, but one can readily adapt the minimax lower bounds to our setting.  } (see e.g. Proposition~1 in~\cite{MR2797847}) consider sequences $r_n$ and $\delta_n$ and state that, for $r_n/n\rightarrow 0$, then the sum of the type I and type II error probabilities of any test converges to one when $\delta_n\sqrt{r_n}
- \sqrt{2\log(n/r_n)}$ goes to $-\infty$. This entails that no test performs asymptotically better than pure random guess. 
For the sake of completeness, we provide a non-asymptotic counterpart.

\begin{prp}\label{prp:lower_detection_multiple}
Assume that $\bepsilon$ in \eqref{eq:univariate_model} satisfies ($\mathcal{A}_G$).
There exist positive numerical constants $c$, $c'$ and $n_0$ such that for all $n\geqslant n_0$, any $1\leq r\leq \lfloor n/4 \rfloor $, and any $\xi$ in $(0,1)$, the following holds. If $\delta\sqrt{r} \leq \sqrt{2(1-\xi)\log(n/(2r))}$, then any test $\mathscr{T}$ of $\{\btheta=0\}$ versus $\{\btheta\in \Theta[\delta,r]\}$ satisfies
\beq\label{eq:lower_multiple_change}
\P_{0}[\mathscr{T}=1] + \sup_{\btheta\in \Theta[\delta,r]}\P_{\btheta}[\mathscr{T}=0]\geqslant 1- c\left(\frac{r}{n}\right)^{c'\xi^2} \enspace . 
\eeq
\end{prp}
We recover that $\delta\sqrt{r}$ needs to be at least of the order of $\sqrt{2\log(n/r)}$ so that the segment is reliably detected. Let us now interpret this impossibility result in terms of energy and high-energy change-points. Any $\btheta$ in $\Theta[r,\delta]$ contains two change-points 
such that, for $k=1,2$, $(\tau^*_{k+1}-\tau_k^*)(\tau^*_{k}-\tau_{k-1}^*)/(\tau^*_{k+1}-\tau_{k-1}^*)$ belongs to $[r(1-r/(\lfloor n/4\rfloor +r)),r]$. As a consequence, their respective energies $\bE_1(\btheta)$ and $\bE_2(\theta)$ belong to $[\delta\sqrt{r(1-r/(\lfloor n/4\rfloor +r)} , \delta \sqrt{r}]$. Rephrasing Proposition~\eqref{eq:lower_multiple_change} we deduce that when, for some $\xi'$ in $(0,1)$ and some $r < c_{\xi} n$, the energy of both change-points is not $(\xi',1)$-high then it is impossible to reliably test whether $\{\btheta\in \Theta_0\}$ versus $\{\btheta\in \Theta_2\}$.

\medskip

Consider some $q>1$. In summary, even in very simple settings such as $\btheta$ in $\Theta_1$ or $\btheta$ in $\Theta_2$, the energy  a change-point has to be at least $(1,q)$-high, so that a test of $\{\btheta\in \Theta_0\}$ versus $\{\btheta\in \Theta_1\cup\Theta_2\}$ of level $ce^{-c'q}$ detects a change-point with probability at least $1-ce^{-c'q}$.

\subsubsection{Localization error}

We now provide lower bounds for the problem of localizing a high-energy change-point $\tau^*_k$.   
To derive the optimal rate, we consider, as in Section \ref{sec:lower:loc:one},  a favorable situation where almost all parameters are known. 
Fix $\bmu=(\mu_1,\ldots, \mu_{K+1})$ with $\Delta_k=\mu_{k+1}-\mu_k\neq 0$ for all $k$ in $\{1,\ldots,K\}$.  
Given  a vector $\bt=(t_1,\ldots,t_{K+1})$ of integers such that $1=t_1< t_2 < \ldots< t_{k+1}=n$, consider the partition $\cI_\bt=(I_k)_{k\in\{1,\ldots,K\}}$ of $\{2,\ldots, n\}$, where, for all $k$, $I_k=\{t_k+1,\ldots, t_{k+1}\}$. 
Denote by 
\[
\Theta_K[\cI_\bt,\bmu]=\{\btheta\in \Theta_K,\  \btheta_i=\mu_k\ \forall i\in\{\tau^*_{k-1},\ldots,\tau^*_k-1\},\text{ with } \tau^*_k\in I_k\ \forall k\in \{1,\ldots,K\}\}\enspace . 
\]
Suppose that the statistician knows that $\btheta$ is in $\Theta_K[\cI_\bt,\mu]$, that is, she  knows the vector $\bmu$ and she knows moreover that, for each $k$ in $\{1,\ldots, K\}$, $\tau^*_k$ belongs to $I_k$. 
Then, $\bY_k= (Y_i)_{i\in I_k}$ is a sufficient statistic for estimating $\tau^*_k$. 
Since the $\bY_k$'s are independent, all $K$ estimation problems of $\tau^*_k$ can be considered independently. 
Arguing as in Section \ref{sec:lower:loc:one}, we derive  the following result.

\begin{prp}\label{prp:lower_loc_multiple}
Assume that $n\geq 3$ and $\bepsilon$ in \eqref{eq:univariate_model} satisfies ($\mathcal{A}_G$).
There exist universal constants $c_1$-- $c_5$ such that the following holds. 
Fix $K$ in $\{2,\ldots, n-1\}$, a vector $\bmu=(\mu_1,\ldots, \mu_{K+1})$ in $\mathbb{R}^{K+1}$ with $\Delta_k=\mu_{k+1}-\mu_k\neq 0$ for all $k$ in $\{1,\ldots,K\}$, and a partition $\cI_\bt=(I_k)_{k\in\{1,\ldots,K\}}$ of $\{2,\ldots, n\}$ as described above and such that $|I_k|> 4\Delta^{-2}_k+2$ for any $k=1,\ldots, K$.\\
For any $k$ in $\{1,\ldots,K\}$, for any estimator $\widehat{\tau}_k$ of $\tau^*_k$ and any $x$ in $[1/2,|I_k|/2-1-2\Delta_k^{-2})$, one has 
\beq\label{eq:lower_risk_ke}
\sup_{\btheta\in \Theta_K[\cI,\mu]}\bbP_{\btheta}\left(|\widehat{\tau}_k-\tau^*_k|>2\Delta_k^{-2}+x\right)\geqslant c_1e^{-c_2x\Delta_k^2}\enspace.
\eeq
The $\bbL_1$-Wasserstein minimax risk over $ \Theta_K[\cI,\mu]$ is bounded from below as follows 
\beq\label{eq:lower_risk_l2}
\inf_{\widehat{\btau}\in \mathbb{N}^K}\sup_{\btheta\in \Theta_K[\cI,\mu]}\mathbb{E}_{\btheta}\left[d_W(\widehat{\btau},\btau^*)\right]\geq c_3 \sum_{k=1}^K \left[\frac{e^{-\Delta^{2}_k/8}}{|\Delta_k|}\1_{|\Delta_k|\geq 2}+  \Delta_k^{-2}\1_{|\Delta_k|\leq  2}\right]\enspace . 
\eeq
Assume that all $\Delta_k$ are equal to some common value $\Delta>0$. If the partition $\cI_\bt$ further satisfies $\min_{k\in\{1,\ldots,K\}}|I_k|\geq c_4 \Delta^{-2} \log K $, then the Hausdorff minimax risk over $ \Theta_K[\cI,\mu]$ satisfies 
\beq\label{eq:lower_risk_haussdorf}
\inf_{\widehat{\btau}\in \mathbb{N}^K}\sup_{\btheta\in \Theta_K[\cI,\mu]}\mathbb{E}_{\btheta}\left[d_H(\widehat{\btau},\btau^*)\right]\geq c_5 \left[\frac{Ke^{-\Delta^{2}/8}}{|\Delta|}\1_{|\Delta|\geq 2\sqrt{2\log(K)}}  + \frac{\log(K)}{\Delta^{2}}\1_{|\Delta|<  2\sqrt{2\log(K)}}\right]\enspace . 
\eeq
\end{prp}

The bottom line of this proposition is that at best, for any estimator $\widehat{\btau}\in \mathbb{N}^K$, the errors $|\widehat{\tau}_k-\tau^*_k|$ behave independently and have an exponential tail of the form~\eqref{eq:lower_risk_ke}. The first lower bound \eqref{eq:lower_risk_ke} for estimating a single change-point $\tau^*_k$ is a straightforward consequence of Lemma~\ref{lem:LBCPE1} and analogous results may be found e.g. in~\cite{Wang2018}. However, the lower bounds~\eqref{eq:lower_risk_l2} and~\eqref{eq:lower_risk_haussdorf} for the Wasserstein and Hausdorff risk are, up to our knowledge, new.

\subsection{Which requirements for a good change-point detection and localization procedure?}\label{sec:specifications}

The obtained lower bounds lead us to reflect upon the properties a good change-point detection and localization  estimator $\widetilde{\btau}$ should satisfy, and to specify them.
First, it should not overestimate the number of true change-points. Besides, one cannot hope that $\widetilde{\btau}$ detects all true change-points $\tau^*_k$. In view of Proposition \ref{prp:lower_detection_multiple}, it is more reasonable to require that $\widetilde{\btau}$ detects all high-energy change-points. Finally, those high-energy change-points should be localized at the parametric rate $\Delta_k$ as explained above.
More precisely, our target estimator $\widetilde{\btau}$ will satisfy, over a large probability event $\cA$ the three following properties.
In the sequel, the notation $\{\tilde{\btau}\}$ is short for the set $\{\tilde{\tau}_l,\ l\in\{1,\ldots, |\tilde{\btau}|\}\}$. 

\medskip

\noindent 
{\bf (NoSp). No  spurious change-point is detected:}

\begin{equation}\label{eq:def:nosp}
\left\{\begin{array}{c}
\Big|\big\{\widetilde{\btau}\big\} \cap \left(\frac{\tau^*_{k-1}+\tau^*_k}{2},\frac{\tau^*_{k}+\tau^*_{k+1}}{2}\right]  \Big|\leq 1\enspace,\ \text{for all}\ k \text{ in }\{2,\ldots, K-1\}\enspace;\\
\Big|\big\{\widetilde{\btau}\big\} \cap \left[2,\frac{\tau^*_{1}+\tau^*_{2}}{2}\right]  \Big|\leq 1 \enspace; \enspace \Big|\big\{\widetilde{\btau}\big\} \cap \left(\frac{\tau^*_{K-1}+\tau^*_K}{2},n\right]  \Big|\leq 1 \enspace .
\end{array}
\right.
\end{equation}

\medskip

\noindent 
{\bf (Detec[$\kappa,q,c$]).  High-energy change-points are detected}: for all $k$ in $\{1,\ldots, K\}$, if $\tau^*_k$ is a $(\kappa,q)$-high-energy change-point then 
\beq\label{eq:def:detec}
d_{H,1}(\widetilde{\btau}, \tau^*_k)\leq\min\left\{ \frac{\tau^*_{k+1}-\tau^*_k}{2},\frac{\tau^*_{k}-\tau^*_{k-1}}{2}, c \frac{\log\left(1\vee n\Delta_k^2\right)+ q}{\Delta_k^{2}} \ \right\}\enspace .
\eeq

\medskip 

\noindent 
{\bf (Loc[$\kappa,q,c,c'$]). High-energy change-points are localized at the optimal rate}: any $(\kappa,q)$-high-energy change-point $\tau^*_k$ satisfies 
\beq\label{eq:def:loc}
 \P\left(d_{H,1}(\widetilde{\btau}, \tau^*_k)\1_{\cA}\geq c {x }{\Delta_k^{-2}}\right)\leq c'e^{-x},\, \quad \forall x\geq 1 \enspace .
\eeq

Property  {\bf (NoSp)} is slightly stronger than requiring that $|\widetilde{\btau}|\leq K$, since it requires that, in each segment $\big((\tau^*_{k-1}+\tau^*_k)/2,(\tau^*_{k}+\tau^*_{k+1})/2]$, the number of change-points is not overestimated.\\
Property {\bf (Detec[$\kappa,q,c$])} specifies that, for any high-energy change-point $\tau^*_k$, there exists at least one predicted change-point $\tilde{\tau}_l$ which is closer to $\tau^*_k$ than any other true change-point. The last term $d_k= c\Delta_k^{-2}[\log\left(1\vee n\Delta_k^2\right)+ q]$   in the upper bound \eqref{eq:def:detec} provides a quantitative bound for $d_{H,1}(\widetilde{\btau}, \tau^*_k)$. Hence, {\bf (Detec[$\kappa,q,c$])}  provides an uniform control on the localization error of each high-energy change-point. In fact, it suffices to prove  {\bf (Detec[$\kappa,q,+\infty$])} to ensure  ({\bf Detec}[$\kappa,q,c$]) for some $c$ large enough. Indeed, inserting in $\btau^*$ two change-points $\tau_-= \tau^*_k- \lceil 2 d_k\rceil $ and $\tau_+= \tau^*_k+ \lceil 2 d_k\rceil$ with a null change-point height does not change the distribution. If $c$ is chosen large enough in~\eqref{eq:def:detec}, then the energy of $\tau^*_k$ relative to  the new change-point vector remains  $(\kappa,q)$-high. As a consequence, {\bf (Detec[$\kappa,q,+\infty$])} entails that, under the event 
$\mathcal{A}$, $d_{H,1}(\widetilde{\btau}, \tau^*_k)\leq d_k$.

As for  Property {\bf (Loc[$\kappa,q,c,c'$])}, it enforces that any ($\kappa,q$)-high-energy jump $\tau^*_k$ is estimated at the optimal rate (as given by Proposition~\ref{prp:lower_loc_multiple}) in the event $\mathcal{A}$. Since $d_{H,1}(\widetilde{\btau},\tau^*_k)$ is a non-negative integer, this property implies that, for $\Delta_k$ large enough, one has $d_{H,1}(\widetilde{\btau},\tau^*_k)=0$ with probability higher than $1-c' e^{-\Delta_k^2/c}- \P(\cA^c)$. Let us emphasize that, in~\eqref{eq:def:loc}, the subscript $k$ is fixed inside the probability, whereas ({\bf Detec}[$\kappa,q,c$]) provides an uniform bound with respect to all high-energy jumps.

In the next two sections, we introduce and analyze change-point detection procedures bearing in mind these three specifications.

\section{Penalized least-squares}\label{sec:least_squares}

\subsection{Definition}

Given a change-point  vector $\btau$, let $\bPi_{\btau}$ denote the projector onto the space of piece-wise constant vectors on $\{\tau_i,\ldots,\tau_{i+1}-1\}$ for $i=0,\ldots, |\btau|$. Then, we consider the penalized least-squares criterion defined by 
\beq\label{eq:definition_criterion}
\Cr_{0}(\bY,\btau):= \|\bY-\bPi_{\btau}\bY \|^2 + L\pen_0(\btau,q)\enspace , 
\eeq
where
\beq\label{eq:definition_penalite}
\pen_0(\btau,q)= q|\btau|+ 2 \sum_{i=1}^{|\btau|+1}\log\left(\frac{n}{\tau_{i}-\tau_{i-1}}\right) \enspace ,
\eeq
 and $L>1$, $q>0$  are tuning parameters. In the above definition, recall that we take the convention $\tau_0=1$ and $\tau_{|\btau|+1}=n+1$. 
Then, the \emph{penalized least-squares estimator} $\widehat{\btau}$ is defined as  any  minimizer of $\Cr_{0}(\bY,\btau)$. The dependency of $\widehat{\btau}$ on the tuning parameters $L$ and $q$ is left implicit. If the noise $\bepsilon$ in the model \eqref{eq:univariate_model} satisfies the Gaussian assumption $(\cA_G)$, $\widehat{\btau}$ is simply a penalized maximum log-likelihood estimator. In the sequel, $\pen_0(\btau;q)$ is referred to as a multiscale penalty because its value depends on the scales of $\tau_i-\tau_{i-1}$. 

\medskip 

\noindent 
{\bf Computational complexity}. The multiscale penalty~\eqref{eq:definition_penalite} is additive. The cost  of a segmentation in~\eqref{eq:definition_criterion} is the sum of the cost of its segments and one can apply Bellman's dynamic programming algorithm~\cite{bellman1961approximation} to compute $\widehat{\btau}$ with at most $O(n^2)$ operations. Furthermore, one may adapt pruning methods such as PELT~\cite{killick2012optimal} or~\cite{rigaill2010pruned} to reduce the complexity to linear in best-case scenarios.

\medskip 

\noindent 
{\bf Penalty choice}. The multiscale penalty  \eqref{eq:definition_penalite} may be contrasted with complexity based penalties in the literature. Among these complexity based penalties, one can essentially distinguish the $\ell_0-$penalties which are linear in the $\ell_0-$norm of $D\btheta=(\theta_{i+1}-\theta_i)_{i\in\{1,\ldots,n-1\}}$ simply equal to $|\btau|$, leading to estimators that are sometimes also referred to as Ising or Potts estimators, and which of course include the BIC penalty $\pen_{BIC}(\btau)= 2|\btau|\log(n)$ used in \cite{yao1988estimating,YaoAu1989,Boysen2009,Garreau2018,lavielle2000least,wang2020}, and the penalties $\pen_{MS}(\btau)=  |\btau|(1+ c\log(n/|\btau|))$ of~\cite{birge2001gaussian,Lebarbier2005}  designed for model selection and signal estimation.
The first term in our penalty \eqref{eq:definition_penalite} exactly corresponds to a $\ell_0-$penalty. The second term gives most weight to change-points $\tau_l$ that are the closest to one of their neighbours $\tau_{l-1}$ or $\tau_{l+1}$ at least, with a largest possible contribution of order $\log(n)$, and less weight to change-points $\tau_l$ such that $(\tau_l-\tau_{l-1})\wedge (\tau_{l+1}-\tau_l)$ is proportional to $n$, which occurs for instance if there is a fixed number of equi-spaced change-points.  Consequently, our penalty $\pen_{0}(\btau,q)$ is, up to multiplicative constants, always smaller than the BIC penalty, and favours equi-spaced change-point vectors.
The general form of $\pen_{0}(\btau,q)$ is related to the definition \eqref{eq:definition_high_energy} of high-energy change-points and to the target property {\bf (Detec[$\kappa,q,c$])}. It is justified below, where we also explain why the previous BIC and model selection penalties lead to suboptimal performances (see Section~\ref{sec:large_penalty}).

\medskip



\noindent
{\bf Connection with CUSUM statistics}. There are deep connections between penalized least-squares criteria such as $\Cr_{0}(\bY,\btau)$ and CUSUM statistics as pointed out for instance in Wang et al.~\cite{wang2020}.  Recall that, given $\bt=(t_1,t_2,t_3)$ in $\cT_3$, the CUSUM statistic at $\bt$ is defined in~\eqref{eq:definition_cusum} as the empirical counterpart of the weighted difference in means on $[t_1,t_2)$ and on $[t_2,t_3)$.
First recall that maximizing the CUSUM statistic $\bC(\bY,\bt)$ over all $\bt$ of the form $(1,\tau,n+1)$, with $\tau$ in $\{1,\ldots,n\}$ is equivalent to computing the least-squares estimator  $\arg\min_{\tau\in\{1,\ldots,n\}}\|\bY-\bPi_{\tau}\bY \|^2$. More generally, local differences of the criterion $\Cr_0(\bY,\btau)$ have a natural expression in terms of the CUSUM statistic.  For any integer $1\leq l\leq |\btau|$, let $\btau^{(-l)}$ denote  the vector $(\tau_1,\ldots, \tau_{l-1},\tau_{l+1}, \ldots, \tau_{|\btau|})$ where $\tau_l$ has been removed. 
From Pythagorean equality, writing $\delta_{l-1}= \tau_{l}-\tau_{l-1}$, $\delta_{l}=\tau_{l+1}-\tau_{l}$, $S_{l-1}= \sum_{i=\tau_{l-1}}^{\tau_{l}-1} Y_i$, and $S_{l}= \sum_{i=\tau_{l}}^{\tau_{l+1}-1} Y_i$, we obtain
 \beqn 
 \|\bY-\bPi_{\btau}\bY \|^2- \|\bY-\bPi_{\btau^{(-l)}}\bY \|^2&=& \|\bPi_{\btau^{(-l)}}\bY \|^2 - \|\bPi_{\btau}\bY \|^2\\
 &= &(\delta_{l-1}+\delta_l)\left(\frac{S_{l-1}+S_{l}}{\delta_{l-1}+\delta_l}\right)^2- \delta_{l-1}\left(\frac{S_{l-1}}{\delta_{l-1}}\right)^2-\delta_l\left(\frac{S_l}{\delta_l}\right)^2\\
 &=& -\frac{\delta_{l-1}\delta_l}{\delta_{l-1} +\delta_l}\left(\frac{S_{l-1}}{\delta_{l-1}}-\frac{S_l}{\delta_l}\right)^2\\
 &=& -\bC^2(\bY,(\tau_{l-1},\tau_{l},\tau_{l+1}))\enspace .
\eeqn 
From  the definition~\eqref{eq:definition_criterion}, we therefore deduce the following lemma.
\begin{lem}\label{lem:decomposition_critere_tau_tau_}
 For any $\btau$ and any $l=1,\ldots, |\btau|$, we have 
\beq\label{eq:decomposition_critere_tau_tau_-l}
\Cr_0(\bY,\btau) - \Cr_0(\bY,\btau^{(-l)})= -\bC^2\big(\bY,(\tau_{l-1},\tau_{l},\tau_{l+1})\big) +  L \left[2\log\left(\frac{n(\tau_{l+1}-\tau_{l-1})}{(\tau_{l+1}-\tau_{l})(\tau_{l}-\tau_{l-1})}\right)+ q\right]\enspace . 
\eeq 
\end{lem}
Heuristically, the above decomposition justifies the choice~\eqref{eq:definition_penalite} for the penalty $\pen_0(\btau,q)$. Indeed, if we assume that the noise vector $\bepsilon$ in $\eqref{eq:univariate_model}$ is equal to zero, then
$\bC^2\big(\bY,(\tau_{l-1},\tau_{l},\tau_{l+1})\big)$ simplifies as $\bE^2\big(\btheta,(\tau_{l-1},\tau_{l},\tau_{l+1})\big))$ and \eqref{eq:decomposition_critere_tau_tau_-l} suggests that, if $\tau_l$ is a $(\sqrt{L},q)$-high-energy change-point, it has to be included in $\widehat{\btau}$. 

\medskip

When $\btheta$ is constant on $[t_1,t_3)$, then the CUSUM statistic $\bC(\bY,\bt)$ does not depend on $\btheta$ and simplifies as the pure noise statistic:
\beq\label{eq:definition_N}
  \quad \quad \bN(\bt):=  \left[\frac{ \sum_{i=t_2}^{t_3-1}\epsilon_i}{t_3-t_2}- \frac{ \sum_{i=t_1}^{t_2-1}\epsilon_i}{t_2-t_1}\right]\sqrt{ \frac{(t_2-t_1)(t_3-t_2)}{t_3-t_1}}\enspace .
\eeq
An uniform control control of $\bN(\bt)$ over $\bt$ in $\cT_3$ is at the heart of the analysis of $\widehat{\btau}$, as explained in the next subsection.

\subsection{Large penalty parameter $L$}\label{sec:large_penalty}

In this subsection, we investigate the properties of the penalized least-squares estimator $\widehat{\btau}$ when the tuning parameter $L$ is chosen large enough.
Define the event $\cA_q$ by
\beq\label{eq:definition_A_q_large}
\cA_q= \left\{ |\bN(\bt)|\leq 2\sqrt{2\log\left(\frac{n(t_3-t_1)}{(t_3-t_2)(t_2-t_1)}\right)+ q}, \quad \forall \bt=(t_1,t_2,t_3)\in \cT_3\right\}\enspace . 
\eeq

Lemma~\ref{lem:concentration:N_t} below states that $\cA_q$ occurs with high probability. In fact, Lemma~\ref{lem:concentration:N_t} holds even if the multiplicative factor 2 in \eqref{eq:definition_A_q_large} is replaced by any constant larger than one. Let us discuss the order of magnitude of this bound. For a fixed $\bt$ in $\cT_3$, $\bN(\bt)$ has a sub-Gaussian distribution with constant $1$. Since there are at most $n^3/6$ elements in $\cT_3$, an union bound easily leads to $\mathbb{P}\left(\sup_{\bt\in \cT_3}|\bN(\bt)|\geq \sqrt{6\log(n)+2x}\right)\leq 2e^{-x}$ for any $x>0$. Such control are for instance used in~\cite{fryzlewicz2014wild} or~\cite{wang2020}. In the event $\cA_q$, we rely on a control of $\bN(\bt)$ that depends on the specific scale of $\bt$. For a small $(t_3-t_2)\wedge (t_2-t_1)$, the upper bound is still of the order of $\sqrt{\log(n)}$, whereas for $(t_3-t_2)\wedge (t_2-t_1)$ proportional to $n$, the upper bound in $\cA_q$ behaves like a constant. Intuitively, this improved bound is related to the fact that the random variables $\bN(\bt)$ at a large scale are highly correlated so that a peeling argument reduces from $\sqrt{\log(n)}$ to a constant. Such behaviour is reminiscent of other multiscale statistics considered in~\cite{frick2014multiscale}.

If the noise vector $\bepsilon$ has a standard Gaussian distribution then, for some $\alpha$ in $(0,1)$, one can compute by Monte-Carlo the minimal
$q_{1-\alpha}$ such that $\mathbb{P}(\cA_{q_{1-\alpha}})\geq  1-\alpha$.

\subsubsection{First step analysis: global optimality}

One may easily deduce from the event $\cA_q$ that  $\widehat{\btau}$ (with $L>4$)  does not overestimate too much the true number of change-points: in each segment $[\tau^*_k, \tau^*_{k+1}]$, the penalized least-squares estimator $\widehat{\btau}$ contains at most two change-points. Indeed, consider any change-point vector $\btau$ that contains three such change-points $\tau_{l-1},\tau_l, \tau_{l+1}$ in $[\tau^*_k, \tau^*_{k+1}]$ then it follows from \eqref{eq:decomposition_critere_tau_tau_-l} that on the event $\cA_q$, 
\[
\Cr_0(\bY,\btau) - \Cr_0(\bY,\btau^{(-l)})= -\bC^2\big(\bY,(\tau_{l-1},\tau_{l},\tau_{l+1})\big) +  L\left[2\log\left(\frac{n(\tau_{l+1}-\tau_{l-1})}{(\tau_{l+1}-\tau_{l})(\tau_{l}-\tau_{l-1})}\right)+ q\right]> 0 
\]
which implies that $\btau \neq \widehat{\btau}$. This simple observation implies that, on $\cA_q$, $|\widehat{\btau}|\leq 2K$. Refining these arguments, we are able to establish in the proof of the next proposition that any $\btau$ with $|\btau|>K$ can be locally modified, by adding and/or removing some change-points, so that the criterion strictly decreases under $\mathcal{A}_q$. Such arguments based on local modifications were previously used in~\cite{wang2020}, but we need to sharpen it to recover the tight logarithms.

Given any change-point vector $\btau$ and some integer $k$ such that $1\leq k\leq |\btau^*|=K$, we write $\btau^{(k)}$ for the (reordered) concatenation of $\btau$ and $\tau^*_k$. For $1\leq l \leq |\btau|$, we say that $\tau_l$ is a $(\kappa,q,\btau)$-high-energy change-point if 
$$\bE(\btheta,(\tau_{l-1},\tau_l,\tau_{l+1}))> \kappa \sqrt{2\log\left(\frac{n(\tau_{l+1}-\tau_{l-1})}{(\tau_{l+1}-\tau_{l})(\tau_{l}-\tau_{l-1})}\right)+ q}\enspace.$$

\begin{prp} \label{prp:rough_analysis}
There exist universal constants $c$, $c'$, $L_0$, $q_0$ and $n_0$ such that the following holds for any  $L\geq L_0$, any $q> q_0$ and any $n>n_0$. For any integer $K>0$ and $\btheta$ in $\Theta_K$, there exists $\kappa_L>0$ (only depending on $L$) and  an event $\cA_q$ (defined in \eqref{eq:definition_A_q_large}) of  probability higher than $1- ce^{-c'q}$  on which the  penalized least-squares estimator $\widehat{\btau}$ satisfies:
\begin{enumerate}
 \item [(a)] {\bf (NoSp)} No  spurious change-point is detected: for all $k=1,\ldots, K$,  
\[\left\{\begin{array}{c}
\Big|\big\{\widehat{\btau}\big\} \cap \left(\frac{\tau^*_{k-1}+\tau^*_k}{2},\frac{\tau^*_{k}+\tau^*_{k+1}}{2}\right]  \Big|\leq 1\enspace , \text{ for all }k \text{ in }\{2,\ldots, K-1\}\enspace;\\
\Big|\big\{\widehat{\btau}\big\} \cap \left[2,\frac{\tau^*_{1}+\tau^*_{2}}{2}\right]  \Big|\leq 1 \enspace; \enspace \Big|\big\{\widetilde{\btau}\big\} \cap \left(\frac{\tau^*_{K-1}+\tau^*_K}{2},n\right]  \Big|\leq 1 \enspace ;
\end{array}
\right.\]

\item [(b)] {\bf (Detec[$\kappa_L,q,\kappa_L$])}  High-energy change-points are detected: for all $k $ in $\{1,\ldots,K\}$ such that $\tau^*_k$ is a $(\kappa_{L},q)$-high-energy change-point, we have
\beq\label{eq:condition_high_energy_discovery2}
d_{H,1}(\widehat{\btau},\tau^*_k)\leq\min\left\{ \frac{\tau^*_{k+1}-\tau^*_k}{2},\frac{\tau^*_{k}-\tau^*_{k-1}}{2}, \kappa_L \frac{\log\left(n\Delta_k^2\right)+ q}{\Delta_k^{2}} \ \right\}\enspace ; 
\eeq
 \item [(c)] For all $k$ in $\{1,\ldots, K\}$, either $\tau^*_k$ belongs to $\widehat{\btau}$ or it is not a $((\sqrt{L}+2),q,\widehat{\btau}^{(k)})$-high-energy change-point. 
\end{enumerate}
\end{prp}

Increasing the value of $q$ in the penalty function $\pen_0(\btau, q)$ leads to a more conservative procedure which achieves  ({\bf NoSp}) with a higher probability guarantee, but $\tilde{\tau}$ provably detects jumps with a higher energy requirement. For a fixed but small probability $\alpha$ in $(0,1)$, one can take $q$ of order $\log(1/\alpha)$ so that ({\bf NoSp}) and ({\bf Detec[$\kappa_L,q,\kappa_L$]}) hold with probability higher than $1-\alpha$. In view of the impossibility results in Section~\ref{sec:multiple:lower:detection}, the properties ({\bf NoSp}) and ({\bf Detec[$\kappa_L,q,\kappa_L$]}) cannot be fulfilled with probability higher than $1-ce^{-c'q}$. The penalized least-squares estimator $\widehat{\btau}$ simultaneously achieves ({\bf NoSp}) and ({\bf Detec[$\kappa_L,q,\kappa_L$]}) with probability at least $1-ce^{-c'q}$.  Up to the constants $c$ and $c'$, we have established in Section~\ref{sec:multiple:lower:detection} that no change-point detection procedure is able to simultaneously achieve ({\bf NoSp}) and ({\bf Detec[$\kappa_L,q,\kappa_L$]}) with a probability higher than that.

\medskip 

This result contrasts with the previous work of \cite{frick2014multiscale}. In \cite{frick2014multiscale}, the SMUCE estimator is proved to achieve an uniform bound slightly different from \eqref{eq:condition_high_energy_discovery2}. Assuming that $
\bE_{min}^2(\btheta)=\min_{k\in \{1,\ldots,K\}}(\tau^*_{k+1}-\tau^*_k)\min_{k}\Delta^2_k$ (c.f. \eqref{eq:definition_energie_min}) is large compared to $\log (n/\min_k(\tau^*_{k}-\tau^*_{k-1}))$, the authors establish  (see their Theorem 2.8), that SMUCE detects all change-points with high probability. If all change-points have similar heights (that is all $|\Delta_k|$ are of the same order) and if all segments are of similar lengths (that is all ($\tau^*_{k+1}-\tau^*_k$) are of the same order), then their hypothesis boils down to assuming that all change-points have a high-energy. Hence, the main novelties of our result in Proposition~\ref{prp:rough_analysis} in comparison with~\cite{frick2014multiscale} are that: (i) it better handles non evenly spaced jumps and (ii) more importantly, it does not require that all jumps have a high energy.

 Other results for Wild Binary Segmentation estimators in \cite{fryzlewicz2014wild,wang2020,baranowski2019narrowest}) or $\ell_0$-penalized least-squares estimators in \cite{wang2020} assume that $\bE_{min}^2(\btheta)$ is higher than $\log(n)$ which, for well-spaced change-points (and a constant $K$), is suboptimal by a $\log(n)$ factor. Besides, all those works also require that all jumps have a high-energy.

\subsubsection{Second step analysis: local optimality}

\begin{prp}\label{prp:local_analysis}
Let $L_0$, $q_0$ and $\cA_q$ be defined as in Proposition \ref{prp:rough_analysis}.  If we fix $L\geq L_0$ and $q>q_0$, there exist $\kappa_L>0$ and $c$, $c'>0$ such that the following holds.  For any  $\btheta$ in $\mathbb{R}^n$ and any $(\kappa_L,q)$-high-energy change-point $\tau^*_k$, the penalized least-squares estimator $\widehat{\btau}$ satisfies
\[
 \P\left(d_{H,1}(\widehat{\btau}, \tau^*_k)\1_{\cA_q}\geq c x\Delta_k^{-2}\right)\leq c'e^{-x}\quad \forall x\geq 1\enspace.
\]
\end{prp}

As explained above, Proposition \ref{prp:local_analysis} improves the results of \cite{frick2014multiscale} and \cite{lavielle2000least} in the sense that it designs nonasymptotic guarantees for a specific change-point to be localized in an optimal way. More precisely, it shows the considered penalized least-squares estimator $\widehat{\btau}$ allows each change-point to be estimated at the optimal parametric rate $\Delta_k^{-2}$, as soon as it is has high-energy. Moreover, it implies in particular that, when $|\Delta_k|$ is larger than one, we have    $\widehat{\tau}_l=\tau^*_k$, for some $l$, on the intersection of $\cA_q$ and an event of probability higher than $1-c'e^{-\Delta_k^{2}/c}$. Up to our knowledge, this result, combined with Proposition \ref{prp:lower_loc_multiple}, provides the first matching upper and lower bounds for high-energy change-points localization. 

\medskip

\noindent 
{\bf From a global to a local problem}. In view of Properties {\bf (NoSp)} and {\bf (Detec[$\kappa_L,q,\kappa_L$])} of $\widehat{\btau}$ and the corresponding minimax lower bounds, it is worth emphasizing  how and whether the detection and localization  of a change-point  $\tau^*_k$ depend  on its height $\Delta_k$ and the segment lengths $(\tau^*_{k}- \tau^*_{k-1})$ and $(\tau^*_{k+1}- \tau^*_{k})$). First, its energy $\bE^2_{k}(\btheta)$, which depends on both its height $\Delta_k$ and the segment lengths $(\tau^*_{k}- \tau^*_{k-1})$ and $(\tau^*_{k+1}- \tau^*_{k})$, must be high enough so that the change-point is detected. 
Then, when the change-point energy is high enough, that is once it can be detected, estimating $\tau^*_k$ becomes a local problem and   the localization error $d_{H,1}(\widehat{\btau},\tau^*_k)$ only depends on the change-point height, and not on the segment lengths. 

\medskip

\noindent 
 {\bf Hausdorff and Wasserstein bounds}.  Taking an union bound over all high-energy change-points yields 
\[
 \max_{k,\ \tau_k^*\text{ has high energy}}  d_{H,1}(\wh{\btau},\tau^*_k) \1_{\cA_q} =O_{P}(1) \max_{k,\ \tau_k^* \text{ has high energy}}  \left(Ke^{-c''\Delta_k^{2}} \wedge   \frac{\log K}{ \Delta_k^{2}}\right)\enspace.
\]
Now, assuming that all change-points have high-energy, we deduce that, on the event $\cA_q$ of Proposition \ref{prp:rough_analysis} and Proposition \ref{prp:local_analysis}, $\widehat{K}=K$ and 
\begin{eqnarray}
 \E\left[d_W\left(\widehat{\btau},\btau^*\right)\1_{\cA_q}\right]& \lesssim& \sum_{k=1}^{K} \left(e^{-c''\Delta_k^{2}}\wedge \frac{1}{\Delta_k^{2}}   \right)\enspace , \label{eq:Wasserstein}\\
 \E\left[d_H\left(\widehat{\btau},\btau^*\right)\1_{\cA_q}\right]& \lesssim& \max_{k\in\{1,\ldots,K\}}  \left(Ke^{-c''\Delta_k^{2}} \wedge   \frac{\log K}{ \Delta_k^{2}}\right)\enspace . \label{eq:Haussdorf}
  \end{eqnarray}
  When  all change-points $\btau_k^*$ have a common height value $\Delta_k=\Delta$, these two bounds turn out to be in view of the impossibility result \eqref{eq:lower_risk_l2} and \eqref{eq:lower_risk_haussdorf} in  Section~\ref{sec:lower_multiple}. Note that the Hausdorff bound \eqref{eq:Haussdorf} can be slightly improved when the change-points heights $\Delta_k$ are heterogeneous, by using an union bound that puts more weights to small $\Delta_k$'s.

\medskip

Up to our knowledge, the Wasserstein risk has only been investigated for a Wild Binary Segmentation estimator in \cite{Wang2018} in a high dimensional framework, but the derived upper bound (see their Corollary 6) is of the  order
of $K n^4/  [(\min_{k\in\{1,\ldots,K\}} \Delta_k)^2 (\min_{k\in\{0,\ldots,K\}} (\tau^*_{k+1}-\tau^*_{k})^4]$, which is is larger than $K^5 /  [(\min_{k\in\{1,\ldots,K\}} \Delta_k)^2]$ and is therefore suboptimal.

The Hausdorff loss of the WBS estimator~\cite{fryzlewicz2014wild, wang2020}, SMUCE~\cite{frick2014multiscale}, and the BIC-penalized least-square estimators~\cite{wang2020}, or other state-of-the art procedures~\cite{baranowski2019narrowest,kovacs2020seeded} are provably upper bounded by $\log (n) /  (\min_{k\in\{1,\ldots,K\}} \Delta_k)^2$ under a more restrictive assumption than the present high-energy requirement. 
Up to our knowledge, the only comparable Hausdorff bound has been recently and independently established by Cho and Kirch~\cite{cho2019localised} under the stronger assumptions that $\min_k\bE_k(\btheta)\gtrsim \sqrt{\log(n)}$.

\subsubsection{Comparison with complexity-based penalized estimators}

Let us investigate the behaviour of the penalized least-square estimator if one had chosen a $\ell_0$-type penalty of the form $L |\btau|\log (n)$, equal to the BIC penalty $\pen_{BIC}$ up to the tuning parameter $L$,  as studied in~\cite{wang2020}, instead of our multiscale penalty $\pen_0(\btau,q)$). It is first worth emphasizing that the corresponding estimator $\widehat{\btau}_{BIC}$ does not satisfy the property ({\bf Detec}). Arguing as in \eqref{eq:decomposition_critere_tau_tau_-l}, we derive from the definition of $\widehat{\btau}_{BIC}$ that for any $l\leq |\widehat{\btau}_{BIC}|$, 
\[
\bC^2\big(\bY,((\widehat{\tau}_{BIC})_{l-1},(\widehat{\tau}_{BIC})_{l},(\widehat{\tau}_{BIC})_{l+1})\big) >  L\log (n)  \enspace .
\]
Unless $L$ is too small, this implies that $\widehat{\btau}_{BIC}$ is only able to detect change-points whose energy is higher than $\sqrt{\log(n)}$. By constrast, when all the segment sizes $\tau^*_{k+1}-\tau^*_{k}$ for $k$ in $\{0,\ldots,K\}$ are proportional to $n$, our penalized least-squares estimator $\widehat{\btau}$ is able to detect change-points with constant energy (see Proposition~\ref{prp:rough_analysis}).
Then, inspecting the proof of Proposition~\ref{prp:local_analysis}, one can see that, when suitably tuned, the estimator $\widehat{\btau}_{BIC}$ is able to localize change-points whose energy is large compared to $\sqrt{\log(n)}$ (that is, change-points that are detected) at the parametric rate $\Delta_{k}^{-2}$.

Since it also satisfies  (${\bf NoSp}$), we can conclude that $\widehat{\btau}_{BIC}$: (i) satisfies $({\bf NoSp})$, (ii) does not satisfy  ({\bf Detec}) but is able to detect change-point whose energy is large compared to $\sqrt{\log(n)}$, (iii) is able to localize such (very) high-energy change-point at the parametric rate $\Delta_{k}^{-2}$ as in Proposition~\ref{prp:local_analysis}.

Turning now to the model selection penalty $\pen_{MS}(\btau)=  |\btau|(1+ c\log(n/|\btau|))$ of~\cite{birge2001gaussian,Lebarbier2005}, we observe that  the corresponding estimator $\widehat{\btau}_{MS}$ partly shares the same weakness as $\widehat{\btau}_{BIC}$: when $|\widehat{\btau}_{MS}|$ is small, the penalty is close to $L|\btau|\log (n)$ and the estimator is overconservative. This is for instance the case when there are $K$ equi-spaced change-points with a small $K$. In other settings, $\widehat{\btau}_{MS}$ detects many spurious change-points. As an example, consider the situation where, in the first half of the sample, there is a high change-point (with a height much larger than $\sqrt{\log (n)}$) every five points and the signal is constant in the second half of the sample. Then, $\widehat{\btau}_{MS}$ will detect all those change-points in the first half so that $|\widehat{\btau}_{MS}|\geq n/10$. As a consequence, $\widehat{\btau}_{MS}$ nearly behaves as the penalized least-squares estimator with a penalty of the order $c'|\btau|$ and one can then show that the number of spurious change-points in the second half is proportional to $n$.

\subsection{Near-Minimal Penalty}

In the previous subsection, we considered large values $L>L_0$ of the tuning parameter  in  \eqref{eq:definition_criterion}. In practice, such large choice of $L$ may lead to too conservative procedures. One may then wonder how small one can take $L$ while still ensuring that $\widehat{\btau}$ does not overestimate too much the number of true change-points. Following the same approach as in the above subsection, one observes that $\widehat{\btau}$ does not contain more than two change-points on the segment $[\tau^*_k, \tau^*_{k+1}]$ as long as  the following event holds: 
\[
 \left\{ |\bN(\bt)|< L^{1/2} \sqrt{2\log\left(\frac{n(t_3-t_1)}{(t_3-t_2)(t_2-t_1)}\right)+ q}, \quad \forall \bt\in \cT_3\right\}\enspace . 
\]
It turns out that such an event occurs with high probability (at least for $q$ not too small) for all $L>1$. Conversely, it occurs with negligible probability when $L<1$ (see Lemma \ref{lem:concentration:N_t}  in the proofs Section for more details).  As a consequence, for $L<1$, the penalized least-squares estimator selects spurious change-points as illustrated by the next proposition.

\begin{prp}\label{prp:negative_minimal_penalty}
Let $\widehat{\btau}$ denote the penalized least-squares estimator of $\btau^*$ defined above, and assume that $\bepsilon$ in \eqref{eq:univariate_model} satisfies ($\mathcal{A}_G$). We have, for any fixed $L<1$, any $q>0$, any $n$ large enough, and any $\btheta$ in $\Theta_0$, that
 $$\P_{\btheta}(|\widehat{\btau}|\geq 1)\geq 1-n^{-1}\enspace\ .$$
\end{prp}
This leads us to more carefully consider the penalized least-squares estimator~\eqref{eq:definition_criterion} when  $L>1$. 

\begin{prp}[First step analysis] \label{prp:minimal_penalty_analysis} For any $L>1$, there exist positive constants $q'_0$, $c_1$-- $c_3$, $\kappa_L$ and $\eta_L$ in $(0,1)$  such that the following holds.  Fix any $q> q'_0 + c_1\log[((L\wedge 2)-1)^{-1}]$. There exists an event $\cA_{L,q}$  (defined in the proof) occurring with probability higher than $1-c_1e^{-c_2q}$ on which the least-squares estimator $\widehat{\btau}$ satisfies:
\begin{enumerate}
 \item[(a)] No interval $[\tau_k^*,\tau_{k+1}^*)$ contains more than two change-points, that is
 \[
\big|\big\{\widehat{\btau}\big\} \cap \big[\tau^*_{k},\tau^*_{k+1}\big) \big|\leq 2\enspace ; 
 \]
Besides, the beginning and the end of the intervals do not contain more than one change-point:
 \[
 \big|\big\{\widehat{\btau}\big\} \cap \big[\tau^*_{k},\tau^*_k+ \eta_L(\tau^*_{k+1}-\tau^*_k)\big]\big|\leq 1,\quad  \big|\big\{\widehat{\btau}\big\} \cap \big[\tau^*_{k+1}- \eta_L(\tau^*_{k+1}-\tau^*_k),\tau^*_{k+1}\big]\big|\leq 1\enspace . 
 \]

\item[(b)] {\bf (Detec[$\kappa_K,q,\kappa_L$])}. High-energy change-points are detected. For all $k$ in $\{1,\ldots, K\}$ such that $\tau^*_k$ is a $(\kappa_{L},q)$-high-energy change-point, we have
\beq\label{eq:condition_high_energy_discovery}
d_{H,1}\left(\widehat{\btau},\tau^*_k\right)\leq\min\left\{ \frac{\tau^*_{k+1}-\tau^*_k}{2}, \frac{\tau^*_{k}-\tau^*_{k-1}}{2}, \kappa_L \frac{\log\left(n\Delta_k^2\right)+ q}{\Delta_k^{2}}\ \right\}\enspace ; 
\eeq
 \item[(c)] For all $k$ in $\{1,\ldots, K\}$, either $\tau^*_k$ belongs to $\widehat{\btau}$ or it is not a $((1.1\sqrt{L}+0.9),\widehat{\btau}^{(k)},q)$-high-energy change-point. 
\end{enumerate}
\end{prp}

Here, we may have $\widehat{K}>K$ and some post-processing procedure is needed to clean the estimator so that it may satisfy {\bf (NoSp)}.

Similarly to the analysis for large tuning parameter $L$, the localization rate of a specific high-energy change-point $\tau^*_k$ is of the order of $\Delta_k^{-2}$ as stated in the next proposition.

\begin{prp}[Second step analysis]\label{prp:local_analysis_2}
 Consider any $L>1$ and fix any $q> q'_0 + c_1\log[((L\wedge 2)-1)^{-1}]$ and $\cA_{L,q}$ as in Proposition \ref{prp:minimal_penalty_analysis}. Then, there exists $\kappa_L>0$ such that for any $(\kappa_L,q)$-high-energy change-point $\tau^*_k$, the least-squares estimator $\widehat{\btau}$ satisfies 
\[
 \P\left(d_{H,1}\left(\widehat{\btau},\tau^*_k\right)\1_{\cA_{L,q}}\geq c {x }{\Delta_k^{-2}}\right)\leq c'e^{-x} \quad \forall x\geq 1+ \log(L-1)_+\enspace.\] 
\end{prp}

In summary, it turns out that properties ({\bf Detec}) and ({\bf Loc}) are valid all the way down to the minimal penalty tuning parameter $L>1$.

\section{Post-processing procedure}\label{sec:post_proc}

Given any vector $\btau$ of estimated change-points, we describe here a post-processing procedure for local improvements of $\btau$.  As e.g. in~\cite{fryzlewicz2014wild}, this post-processing procedure is a two steps procedure: the first step consists in cleaning spurious change-points of $\btau$; the second one ensures improvement of the localization of well-separated change-points. 
We underline that, in this section, the preliminary estimator  $\btau$ possibly depends on the data ${\bf Y}$.

\subsection{Pruning step}

In the pruning step, we rely on the CUSUM statistic~\eqref{eq:definition_cusum} to build confidence intervals around each $\tau_l$ for $l=1,\ldots ,|\btau|$. 
Given $\alpha$ in $ (0,1)$, let $\zeta_{1-\alpha}$ be the smallest number satisfying 
\beq\label{eq:definition_psi_alpha}
 \P\left[\sup_{\bt\in \cT_3}\left(|\bN(\bt)|-\sqrt{2\log\left(\frac{n(t_3-t_1)}{(t_3-t_2)(t_2-t_1)}\right)}\right)> \zeta_{1-\alpha}\right]\leq \alpha\enspace.
\eeq	
 From Lemma~\ref{lem:concentration:N_t}, we know that $\zeta_{1-\alpha}\leq c_1+c_2\sqrt{\log(1/\alpha)}$ for two numerical constants $c_1$ and $c_2$. If the distribution of the noise vector $\bepsilon$ in \eqref{eq:univariate_model} is a Gaussian distribution or another known distribution that can be simulated, then $\zeta_{1-\alpha}$ can be approximated by a Monte-Carlo method. In the sequel, we denote by $\cB_{1-\alpha}$ the event such that $|\bN(\bt)|\leq \sqrt{2\log(\frac{n(t_3-t_1)}{(t_3-t_2)(t_2-t_1)})})+ \zeta_{1-\alpha}$ for all $\bt$ in $\cT_3$.

\medskip

Given $\tau$ in $\{2,\ldots, n\}$ and a positive integer $r$, define $\bt^{(\tau,r)}=((\tau- r)\vee 1, \tau, (\tau+r) \wedge (n+1))$. In the simple case where   $r < \tau\wedge (n-\tau)$, then $\bt^{(\tau,r)}$ is simply the triplet  $(\tau-r,\tau,\tau+r)$ centered at $\tau$ with radius $r$. Then, we define $\widehat{r}_{\tau}$, the smallest radius $r$ such that the CUSUM statistic $\bC(\bY,\bt^{(\tau,r)})$ centered at $\tau$ is significantly large: 
 \beq\label{eq:definition_r_l}
  \widehat{r}_{\tau}= \min\left\{r,\  |\bC(\bY,\bt^{(\tau,r)})|> \sqrt{2\log\left(\frac{n(t^{(\tau,r)}_{3}-t^{(\tau,r)}_{1})}{(t^{(\tau,r)}_{3}-t^{(\tau,r)}_{2})(t^{(\tau,r)}_{2}-t^{(\tau,r)}_{1})}\right)}+ \zeta_{1-\alpha}   \right\}\enspace , 
 \eeq
with the convention $\min\{\emptyset\}=+\infty$. 
This leads us to the following confidence interval $\underline{I}_\tau$ around $\tau$. 
\beq\label{eq:definition_r_l'}
  \underline{I}_\tau= [t^{(\tau,\widehat{r}_{\tau})}_1+1,t^{(\tau,\widehat{r}_{\tau})}_3-1]\enspace ,
 \eeq
Recall from the definitions~\eqref{eq:definition_cusum} and \eqref{eq:definition_N} of $\bC(\bY,\bt)$ and $\bN(\bt)$ that, if no change-point occurs in $[t^{(\tau,r)}_1+1, t^{(\tau,r)}_3-1]$, then $\bC(\bY,\bt^{(\tau,r)})= \bN(\bt^{(\tau,r)})$. Hence, if no change-point occurs in $[t^{(\tau,\widehat{r}_{\tau})}_1+1, t^{(\tau,\widehat{r}_{\tau})}_3-1]$, then $\widehat{r}_{\tau}$ is the smallest radius such that $|\bN(\bt^{(\tau,r)})|$ is larger than  
 $$\sqrt{2\log\left(\frac{n(t^{(\tau,r)}_{3}-t^{(\tau,r)}_{1})}{(t^{(\tau,r)}_{3}-t^{(\tau,r)}_{2})(t^{(\tau,r)}_{2}-t^{(\tau,r)}_{1})}\right)}+ \zeta_{1-\alpha} \enspace,$$
 which contradicts the event $\cB_{1-\alpha}$ defined in \eqref{eq:definition_psi_alpha}. We conclude that, under $\cB_{1-\alpha}$,  all intervals $\underline{I}_{\tau}$ contain
at least one true change-point.

As a consequence, if two intervals $\underline{I}_{\tau}$ and $\underline{I}_{\tau'}$ are disjoint, then the closest true change-point from   $\tau$ differs from the closest true change-point from  $\tau'$. In a nutshell, the pruning step amounts to removing change-points in $\btau$ in such a way that the confidence intervals around the coordinates of the pruned change-point vector  do not intersect. The procedure (described in Algorithm~\ref{EWA} below) first and foremost prunes wide confidence intervals, since the corresponding estimated change-points are more prone to lie farther from a true change-point than narrow confidence intervals.

Given a vector $\btau$ and the corresponding confidence intervals $\underline{I}_{\tau_{1}}, \ldots, \underline{I}_{\tau_{|\btau|}}$, we reorder the change-points  by decreasing sizes of the corresponding confidence intervals, that is $\tau_{(1)},\tau_{(2)}, \ldots, \tau_{(|\btau|)}$ are such that $\widehat{r}_{\tau_{(1)}}\geq  \widehat{r}_{\tau_{(2)}}\geq \ldots \geq \widehat{r}_{\tau_{(|\btau|)}}$.)

\begin{algorithm}[H]
\caption{Pruning Step}
\label{EWA}
\begin{algorithmic}[1]
\STATE $\cP(\btau)=\btau$ \COMMENT{Initialization with all change-points}
	\FOR{$l = 1, \ldots, |\btau|-1, \do $}
		\IF{($\widehat{r}_{\tau_{(l)}}=\infty$) or ($\underline{I}_{\tau_{(l)}}$ intersects $\bigcup_{j>l }\underline{I}_{\tau_{(j)}}$)}
		\STATE{ Remove $\tau_{(l)}$ from  $\cP(\btau)$.}
		\ENDIF		 
	\ENDFOR
	\ENSURE $\cP(\btau)$
\end{algorithmic}

\end{algorithm}
In the next proposition, we say that a true change-point $\tau^*_k$ is reasonably well localized by $\btau$ if 
 \beq\label{eq:definition_well_localization_tau}
 d_{H,1}(\btau, \tau^*_k) < \frac{(\tau_{k}^* - \tau_{k-1}^*)\wedge (\tau_{k+1}^* - \tau_k^*)}{8}\enspace . 
 \eeq
 
 \begin{prp}\label{prp:pruned}
 There exist universal constants $\kappa$ and $c$ such that the following holds for any $\alpha$ in $(0,1)$. On the  event $\cB_{1-\alpha}$ of  probability higher than $1-\alpha$ and for any sequence $\btau$, we have
\begin{itemize}
 \item [(a)] {\bf (NoSp)} $\cP(\btau)$ does not detect any spurious change-point:
 \[\left\{\begin{array}{c}
\Big|\big\{\cP(\btau)\big\} \cap \left(\frac{\tau^*_{k-1}+\tau^*_k}{2},\frac{\tau^*_{k}+\tau^*_{k+1}}{2}\right]  \Big|\leq 1\enspace , \ \text{for all}\ k \text{ in }\{2,\ldots, K-1\}\enspace;\\
\Big|\big\{\cP(\btau)\big\} \cap \left[2,\frac{\tau^*_{1}+\tau^*_{2}}{2}\right]  \Big|\leq 1 \enspace ; \enspace \Big|\big\{\cP(\btau)\big\} \cap \left(\frac{\tau^*_{K-1}+\tau^*_K}{2},n\right]  \Big|\leq 1 \enspace .
\end{array}
 \right.\]
\item [(b)] {\bf (Detec[$\kappa,\zeta^2_{1-\alpha},c$])}  Any $(\kappa,\zeta^2_{1-\alpha})$-high-energy  change-point $\tau^*_k$ 
 that is reasonably well localized by $\btau$ (in the sense of \eqref{eq:definition_well_localization_tau}) also satisfies
\beq\label{eq:upper_distance_dP_tau}
 d_{H,1}\left(\cP(\btau),\tau^*_k\right) \leq \Big(2 d_{H,1}(\btau, \tau^*_k)\Big)\vee \left(c\frac{\log\left(n\Delta_k^{2}\right)+ \zeta^2_{1-\alpha}}{\Delta_k^{2}}\right)\enspace .	
\eeq 
\end{itemize}

 \end{prp}

Note that any high-energy change-point $\tau^*_k$ detected by $\btau$ is also detected by $\cP(\btau)$. Unfortunately, Proposition \ref{prp:pruned} does not ensure that $\cP(\btau)$ has kept the closest change-point to $\tau^*_k$ in  $\cP(\btau)$ (it is only known to satisfy \eqref{eq:upper_distance_dP_tau}).  For this reason, we introduce a second post-processing step that aim at improving the localization rate for $\tau^*_k$.

\subsection{Local Improvements}

Consider any $\tau$ in $\{2,\ldots,n\}$ and the associated confidence interval $\underline{I}_{\tau}$ defined in~\eqref{eq:definition_r_l'}. We suggest to  re-estimate $\tau$ by minimizing  a restricted least-squares type criterion over $\underline{I}_{\tau}$.
Consider the restricted vector $\bY^{(\tau,2\widehat{r}_\tau-1)}= (Y_{t^{\tau,2\widehat{r}_\tau-1}_1},\ldots, Y_{t^{\tau,2\widehat{r}_\tau-1}_3-1})$.  For a small $\widehat{r}_{\tau}$, this means that we only keep the observations  in $[\tau-2\widehat{r}_{\tau}+1, \tau+2\widehat{r}_{\tau}-2]$. 
Let $\cL(\tau)$ be the change-point estimator minimizing over $\underline{I}_{\tau}$ the restricted least-squares criterion:
\[
 \cL(\tau) \in \argmin_{\tau'\in \underline{I}_{\tau}} \|\bPi_{\tau'}\bY^{(\tau,2\widehat{r}_\tau-1)} \|^2 \enspace .
\]
Equivalently, $\cL(\tau)$ is any  maximizer of the CUSUM statistic $\bC[ \bY, (t^{\tau,2\widehat{r}_\tau-1}_1,\tau', t^{\tau,2\widehat{r}_\tau-1}_3)]$ over $\tau'\in \underline{I}_{\tau}$. 
For short, we write $\cL(\btau)$ for the vector $(\cL(\tau_l))_{l=1,\ldots, |\btau|}$ and we write $\cL \cP(\btau) = \cL(\cP(\btau))$ for the change-point vector obtained after Pruning and Local improvement of $\btau$. 
\medskip 

Re-estimating the change-point positions by a restricted least-squares criterion minimization was already proposed in the literature (see e.g;~\cite[Section 3.2]{fryzlewicz2014wild}). Nevertheless, our fitting method differs in two ways: first, we restrict the new position to belong to $\underline{I}_{\tau}$ and second, we only consider data at distance less than $2\widehat{r}_\tau$ from $\tau$ whereas~\cite{fryzlewicz2014wild} considers data that are closer to $\tau$ than any of the other estimated change-points. These two differences allow us to better handle cases where some of the true change-points have a small energy.

\begin{prp}[{\bf Loc[$\kappa,\zeta^2_{1-\alpha},c,c'$]}]\label{prp:post:localization}
 There exist $\kappa$ and $c$ such that, on the event $\cB_{1-\alpha}$ defined below~\eqref{eq:definition_psi_alpha}, the following holds for any $\btheta$ in $\mathbb{R}^n$, any $(\kappa,\zeta^2_{1-\alpha})$-high energy change-point $\tau^*_k$ and any $\tau$ in $\{2,\ldots,n\}$. If 
\[
|\tau- \tau^*_k| < \frac{(\tau_{k}^* - \tau_{k-1}^*)\wedge (\tau_{k+1}^* - \tau_k^*)}{4}\  , 
\]
Then, for any $x>0$, we have 
\[
\P_{\btheta}\left( |\cL(\tau)-\tau^*_k|\1_{\cB_{1-\alpha}}\geq c\frac{x\vee 1}{\Delta_k^{2}}\right)\leq e^{-x}\enspace .
\]
\end{prp}

\medskip

Fix any $L>1$ and consider the post-processed penalized least-squares estimator $\cL\cP(\widehat{\btau})$ with $L>1$ and some $q  \geq q'_0 + \log[(L\wedge 2-1)^{-1}]$. Then,  it follows from the above propositions that, on the event $\cA_q\cap \cB_{1-\alpha}$ of probability higher than $1-ce ^{-c'q}-\alpha$, one has
\begin{itemize}
 \item[(a)] {\bf (NoSp)}. $\cL\cP(\widehat{\btau})$ does not detect any spurious change-point.
 
 \item[(b)] {\bf (Detec[$\kappa_L,q\vee \zeta^2_{1-\alpha},c$])} Any $(\kappa_L,q\vee \zeta^2_{1-\alpha})$-high-energy  change-point $\tau^*_k$ is reasonably well localized by $\cL\cP(\widehat{\btau})$
\[
 d_{H,1}(\cL\cP(\widehat{\btau}),\tau^*_k) \leq \frac{(\tau_{k}^* - \tau_{k-1}^*)\wedge (\tau_{k+1}^* - \tau_k^*)}{2}\wedge  \left(c\frac{\log\left(n\Delta_k^{2}\right)+ q}{\Delta_k^{2}}\right)\enspace .	
\]
\item [(c)] {\bf (Loc[$\kappa_L,q\vee \zeta^2_{1-\alpha},c,c'$])} For all such $(\kappa_L,q\vee \zeta^2_{1-\alpha})$-high-energy change-points $\tau^*_k$ we have 
\[
 \P_{\btheta}\left(d_{H,1}(\cL\cP(\widehat{\btau}), \tau^*_k)\1_{\cA_q\cap \cB_{1-\alpha}}\geq c \frac{x\vee 1 }{\Delta_k^{2}}\right)\leq c'e^{-x}\enspace , \quad \quad \forall x> 1 \enspace . 
\]
\end{itemize}

The post-processed least-squares estimator with $L>1$ therefore achieves the three aforementioned properties ({\bf NoSp}), ({\bf Detec}), and ({\bf Loc}).

\subsection{Post-processing the complete change-point vector as self-standing procedure}

The post-processing method can also be used as a self-standing change-point detection procedure by simply applying it to the full vector $\btau_f=\{2,3,\ldots, n\}$ of the $n-1$  possible change-points.  From Propositions \ref{prp:pruned} and \ref{prp:post:localization}, we deduce the following.

\begin{cor}[Analysis of $\cL\cP(\btau_f)$] \label{cor:complete_change-point_vector}
Consider any $\alpha$ in $(0,1)$. There exist numerical constants $\kappa$, $c$, $c'$  and  an event $\cB_{1-\alpha}$  (defined below \eqref{eq:definition_psi_alpha}) of probability higher than $1-\alpha$,  such that on $\cB_{1-\alpha}$,  $\cL\cP(\btau_f)$ satisfies  {\bf (NoSp)}, {\bf (Detec[$\kappa,\zeta_{1-\alpha}^{2},c$])}, and {\bf (Loc[$\kappa,\zeta_{1-\alpha}^{2},c,c'$])}.
\end{cor}

In summary, $\cL\cP(\btau_f)$ achieves all the optimality performances specified in Section~\ref{sec:specifications}. In particular, its performances with respect to Hausdorff and Wasserstein risks are similar to that of the penalized least-squares estimators (see \eqref{eq:Wasserstein} and \eqref{eq:Haussdorf}).

\medskip 

\noindent 
{\bf Computational complexity.} In worst case, computing $\cL\cP(\btau_f)$ requires $O(n^2)$ operations. If we slightly modify the definition of $\widehat{r}_{\tau}$ by considering $\widehat{r}^{(d)}_{\tau}$ taking values in the dyadic set $\mathbb{D}:= \{1,2,4,\ldots, 2^{\lfloor\log_2 n\rfloor}\}$ and if we take a confidence interval $\underline{I}^{(d)}_\tau$ based on $\widehat{r}^{(d)}_{\tau}$,  Proposition~\ref{prp:pruned} and Corollary~\ref{cor:complete_change-point_vector} still remain valid for the corresponding Pruned and Locally improved $\btau_f$ denoted by $\cP^{(d)}(\btau_f)$ and $\cL\cP^{(d)}(\btau_f)$  respectively, to the price of slightly worse numerical constants $c$ and $\kappa$. Besides, $\cL\cP^{(d)}(\btau_f)$ can be computed in $O(n\log n)$ operations with a $O(n)$ space complexity. Indeed, for $\tau$ in  $(r,n-r)$,  $\bC(\bY,\bt^{(\tau+1),r})- \bC(\bY,\bt^{\tau,r})$ only depends on $Y_{\tau}$, $Y_{\tau+r}$ and $Y_{\tau-r}$. As a consequence, if all the $\bC(\bY,\bt^{\tau,r})$ are stored, then $\widehat{r}_{\tau+1}^{(d)}$ can be computed in $O(\log n)$ operations which results in $O(n\log n)$ operations to compute all the confidence intervals $\underline{I}^{(d)}_\tau$. Then, checking whether each $\underline{I}^{(d)}_{\tau_{(l)}}$ intersects $\bigcup_{j>l }\underline{I}^{(d)}_{\tau_{(j)}}$ can be done in $O(\log n)$ operations using a binary search algorithm. Finally, the local improvement step computational complexity is at worst linear.
In summary, the total  complexity is  $O(n\log(n))$ and as low as the one of some Binary Segmentation algorithms (see ~\cite{fryzlewicz2018tail,kovacs2020seeded}) or multiscale MOSUM~\cite{cho2019localised}, while it enjoys the targeted optimality properties {\bf (NoSp)}, {\bf (Detec)}, and {\bf (Loc)}.

\subsection{A simple global confidence region}

Given a change-point vector $\btau$, we define the global confidence interval $\tilde{I}^{(g)}_{\btau}$ as the union of the confidence intervals $\underline{I}_{\tau_l}$ introduced in  \eqref{eq:definition_r_l'}.
\beq\label{general_confidence_interval}
\tilde{I}^{(g)}_{\btau}= \cup_{l\in \{1,\ldots, |\btau|\}}\underline{I}_{\tau_l}\enspace . 
\eeq
It follows from the definition \eqref{eq:definition_r_l'} of $\underline{I}_{\tau}$ that, with probability higher than $1-\alpha$,  each segment  of $\tilde{I}^{(g)}_{\btau}$ contains at least one true change-point. For suitable estimators $\btau$ we have the following.

\begin{cor} \label{cor:general_confidence_interval}
Let $\btau$ be one of the three following estimators:
\begin{itemize}
\item[(i)] the penalized  least-squares estimator 
$\widehat{\btau}$ with large tuning parameter $L$ (as in Proposition \ref{prp:rough_analysis}),
\item[(ii)] the post-processed least-squares estimator $\cL\cP(\widehat{\btau})$ with $L\geq 1$,
\item[(iii)] the post-processed complete vector $\cL\cP(\btau_f)$.
\end{itemize}
Then with probability higher than $1-\alpha$, $\tilde{I}^{(g)}_{\btau}$ satisfies:
\begin{enumerate}
 \item[(a)] Each segment  of $\tilde{I}^{(g)}_{\btau}$ contains at least one true change-point ;
 \item[(b)] Each $(\kappa,\zeta^2_{1-\alpha})$-high energy change-point $\tau^*_k$ belongs to $\tilde{I}^{(g)}_{\btau}$. Besides, the width of the  corresponding segment in $\tilde{I}^{(g)}_{\btau}$ is less or equal to 
 \[
 c \frac{\log\left(n\Delta_k^2\right)+ \zeta^2_{1-\alpha}}{\Delta_k^{2}}\enspace .   
 \]
\end{enumerate}
\end{cor}

The proof is straightforward and is therefore omitted. 
In the above corollary, all the change-points do not necessarily have a high energy but we allow such change-points not to belong to $\tilde{I}^{(g)}_{\btau}$. This contrasts with the results of Frick et al.~\cite{frick2014multiscale} which require that all the change-points have a high-energy: $\bE_{min}^2(\btheta)\geq \log[n/(\min_{k\in\{0,\ldots,K\}} (\tau^*_{k+1}-\tau_{k}^*))]$.

If we had assumed that all the energies of the change-points were high, we could improve the tuning confidence region $\tilde{I}^{(g)}_{\btau}$ to segments of size $\log K/\Delta_k^{2}$ by plugging an estimator of $\Delta_k$ in Hausdorff bounds such as ~\eqref{eq:Haussdorf}. Similarly, one could also obtain a confidence interval of width $c_{\alpha}/\Delta_k^{2}$ for a specific change-point $\tau^*_k$ by plugging an estimator of $\Delta_k$ in  Propositions~\ref{prp:local_analysis} or \ref{prp:post:localization}.

\section{Discussion}\label{sec:discussions}

\subsection{Adaptation to $K\leq 1$ and $K>1$}

In the later parts of the manuscript, we allowed the number $K$ to lie between $0$ and $n-1$. Still, when $K=1$, our multiple change-point procedures only detect $\tau_1^*$ when its energy $\bE^2_1(\btheta)$ is large compared to $\log({n}/{((\tau_1^*-1)\wedge (n+1-\tau_1^*))})$. In Section~\ref{sec:one_change_point}, we established that it is possible to detect and localize change-points $\tau_1^*$ for energies that are logarithmically smaller. It is possible to combine the results 
to achieve simultaneous optimality for both $K\leq 1$ and $K>1$ with the following scheme. First, one computes an estimator $\widehat{\btau}$ using either the penalized criterion~\eqref{eq:definition_criterion} or the two-step methods $\cL\cP(\btau_f)$. If $|\widehat{\btau}|>0$, then $\widehat{\btau}$ is left unchanged. If $|\widehat{\btau}|=0$ (no change-point detected), the null hypothesis of non-existence of a change-point is tested (Section~\ref{sec:CPDTest}). If the test accepts the null,  $\widehat{\btau}$ is left unchanged again (no-change point detected). If the test rejects the null, one uses for $\widehat{\btau}$ the one-change point least-squares  estimator of Section~\ref{sec:least_one}. One can then easily prove that the resulting $\widehat{\btau}$ is simultaneously optimal for both settings. 

\subsection{Extensions: unknown variance,  heavy tail distributions and dependences}

Considering the model \eqref{eq:univariate_model}, where the noise $\bepsilon$ has a Gaussian distribution, but with an unknown variance, the unknown variance has to be estimated: the resulting estimator can then be plugged-in to rescale the data $\bY$. Suppose that we are given an estimator $\widehat{\sigma}$ (possibly depending on $\bY$) such that the ratio $\widehat{\sigma}/\sigma$ belongs to $[1,2]$ on an event $\mathcal{E}$ of high probability. Then, all the properties of our multiscale least-squares and post-processing procedures remain valid  on the event $\cE$. Some practical recommendations for estimating $\sigma$ are given in~\cite{fryzlewicz2020detecting}. As an alternative to estimating $\sigma$, one could calibrate the multiplicative constant $L$ of the penalized least-squares estimator $\widehat{\btau}$ using the slope heuristic~\cite{arlot2019minimal}. Indeed, the slope heuristic is partly validated by Propositions~\ref{prp:negative_minimal_penalty} and \ref{prp:minimal_penalty_analysis}.

\medskip

The only stochastic ingredient in the first step analysis of the multiscale penalized estimator (Propositions~\ref{prp:rough_analysis} and \ref{prp:minimal_penalty_analysis}) and in the analysis of the pruning step of the post-processing procedure (Proposition~\ref{prp:pruned}) is an uniform control~\eqref{eq:definition_N} of the noise variables $\bN(\bt)$. In the presence of dependences between the noise variables (as e.g. in~\cite{cho2019localised}), these stochastic controls become more challenging and the constants need to be adapted. Upon establishing such controls, the remainder of the proofs and the conclusion of the propositions remain unchanged (up to other constants). If the noise tails are heavier than sub-Gaussians, then uniform controls of the form~\eqref{eq:definition_psi_alpha} do not hold anymore. Relying on chaining arguments as in the proof of Lemma~\ref{lem:concentration:N_t}, we would obtain larger uniform bounds for $|\bN(t)|$ (in the spirit of \cite{cho2019localised}). Hence, the form~\eqref{eq:definition_penalite} of the penalty $\pen_0(\btau,q)$ and thresholds for confidence intervals~\eqref{eq:definition_r_l} have to be accommodated in a similar fashion to~\cite{cho2019localised}. As for the local analyses (Propositions~\ref{prp:local_analysis}, \ref{prp:local_analysis_2}, and \ref{prp:post:localization}), it seems of reach to accommodate the presence of mild dependences or heavier tails and still derive a local error of order $\Delta_k^{-2}$, but the tail distribution of the error will not be exponential anymore. We leave this as an open problem.

\subsection{Exact constant for multiple change-point detection}

In the single change-point problem, we establish that the exact leading constant for detection equals $\sqrt{2}$, whereas all of our detection results in the multiple change-point setting are stated up to numerical constants. 

Let us shortly discuss the setting of segment detection. Segment detection is a specific instance of the change-point problem where $K$ is even, $\mu_1=\mu_3=\ldots =\mu_{K+1}=0$ and, for $\ell=1,\ldots, K/2$,  $\tau^*_{2\ell}-\tau^*_{2\ell-1}$ is much smaller than $\tau^*_{2\ell-1}-\tau^*_{2\ell-2}$. In other words,  the signal $\btheta$ is null  except at $K/2$ segments that are well-spaced. It has been established in \cite{Kou2017,Chan2013} that a segment (and the two corresponding change-points $\tau^*_{2\ell-1}$ and $\tau^*_{2\ell}$) can be confidently detected as long as $\tau^*_{2\ell-1}$ is a $(\kappa,q)$-high energy change point with $\kappa=1$, as defined in~\eqref{eq:definition_high_energy}. We conjecture that the constant $\kappa=1$ is not sufficient in the more general change-points detection problem so that the segment detection problem is intrinsically easier than change-point detection. Still, we leave this as an open question.

\subsection{Do other classical procedures satisfy the optimality requirements~({\bf NoSp}), ({\bf Detec}) and ({\bf Loc})?}

While we have introduced two change-point procedures meeting the specific requirements, this does not imply that they outperform other procedures. As discussed in Section~\ref{sec:least_squares}, the BIC-penalized, when tuned suitably, achieves ({\bf NoSp}) and variants ({\bf Detec}$_{\log(n)}$) and ({\bf Loc}$_{\log(n)}$) of ({\bf Detec}) and ({\bf Loc}) where only change-points whose energy is high compared to $\sqrt{\log(n)}$ are detected and localized at the parametric rate. 

At the very least, achieving ({\bf Detec}) (instead of ({\bf Detec}$_{\log(n)}$)) is possible only if the procedure uses scale-adaptive threshold as in Frick et al.~\cite{frick2014multiscale} or Cho and Kirch~\cite{cho2019localised}. Still, some work is required to check whether their methods accommodate well low-energy change-points. Investigating the proofs of Baranowski et al.~\cite{baranowski2019narrowest} and Kov{\'a}cs  et al.~\cite{kovacs2020seeded}, we believe that both NOT and Seeded Binary Segmentation achieve ({\bf NoSp}) and ({\bf Detec}$_{\log(n)}$), but it is not straightforward to check whether ({\bf Loc}) holds or not. 
\medskip

Our desired properties of the form ({\bf NoSp}), ({\bf Detec}) and ({\bf Loc})  are more restrictive than typical recent results in the literature by (i) tightening the logarithms (ii) requiring a parametric localization rate of the detection threshold and (iii) allowing for possibly many small jumps that are undetectable and possibly act as nuisance parameters. This  last requirement is more in line with the literature on related statistical fields such as multiple testing or high-dimensional variable selection. Such desiderata can be extended to many other change-points settings beside the univariate case including change-points analysis  in the multivariate~\cite{Wang2018} or kernel~\cite{Arlot2019} frameworks. We hope that it will stimulate a research agenda towards optimal procedures in these settings.


\section{Proofs for a single change-point}

For any two integers $t_1,t_2$ such that $1\le t_1<t_2\leq n+1$, we denote
\beq\label{eq:definition_Z}
Z_{t_1:t_2} :=  \sum_{i=t_1}^{t_2-1}\epsilon_i\enspace ,  
\eeq
the partial sum. Besides, we shall often write $\bE_k$ for $\bE_k(\btheta)$ to alleviate the notation.

\subsection{Proof of Proposition~\ref{prp:lower_test_one_change-point}}

Consider the dyadic collection $\cT=\{ 2^l; l= 1, \ldots , \lfloor \log_2(n)/2\rfloor \}$ where $\log_2(n)=\log(n)/\log(2)$. In the sequel, we denote $a= \kappa^2/[4(1-\kappa)]$, and for $\tau$ in $\cT$, we define $\btheta(\tau)$ in $\Theta_1$ by $\theta^{(\tau)}_i= 0$ for $i\geq \tau$ and $\theta^{(\tau)}_i= \sqrt{2(1-\kappa)\log\log(n)/(\tau-1)}$ for any $i<\tau$. 
 As a consequence, $\|\btheta(\tau) \| = \sqrt{2(1-\kappa)\log\log(n)}$  and the energy $\bE_1(\btheta(\tau))$ satisfies $\bE_1^2(\btheta(\tau))/[2(1-\kappa)\log\log(n)]\in [1-n^{-1/2}; 1]$. To alleviate the notation, we write $\mathbb{P}_{\tau}$ for the distribution $\mathbb{P}_{\btheta(\tau)}$, while $\P_{0}$ still stands for the distribution of $\bY$ when $\btheta=0$. 
 
 We claim that  testing whether $\Theta=0$ versus $\Theta\in \{\Theta^{(\tau)},\, \tau\in \cT\}$ with small type I and Type II error probabilities is impossible. 
 More precisely we shall prove that there exists a constant $c>0$ such that, for any test $\mathscr{T}$,
 \beq\label{eq:objectif_1}
  \P_0[\mathscr{T}=1]+ \max_{\tau \in \cT}\P_{\tau}[\mathscr{T}= 0]\geq 1 - c\log^{-a/2}(n)\enspace . 
 \eeq
Note that~\eqref{eq:objectif_1} implies the result of the proposition. 
Following Le Cam's approach~\cite{baraud02}, we define the mixture probability $\mathbf{P}= |\cT|^{-1}\sum_{\tau \in \cT}\P_{\tau}$. 
Since 
\begin{eqnarray*}
 \P_0[\mathscr{T}=1]+ \max_{\tau \in \cT}\P_{\tau}[\mathscr{T}= 0]&\geq &  \P_0[\mathscr{T}=1]+\mathbf{P}[\mathscr{T}= 0]\\
 &\geq &  1-\abs{\P_0[\mathscr{T}=1]-\mathbf{P}[\mathscr{T}= 1]}\enspace,
 \end{eqnarray*}
inequality~\eqref{eq:objectif_1} is satisfied for all test $\mathscr{T}$ if the total variation distance between $\|\P_0-\mathbf{P}\|_{TV}$ is less or equal to $c\log^{-a/2}(n)$. Writing $L_{\tau}$ the likelihood ratio of $\P_{\tau}$ over $\P_0$ and $L= |\cT|^{-1}\sum_{\tau \in \cT}L_{\tau}$ the likelihood ratio of $\mathbf{P}$ over $\P_0$, this is in turn equivalent to 
\beq\label{eq:objectif_2}
\|\P_0-\mathbf{P}\|_{TV} = \frac{1}{2}\E_0[|L-1|]\leq c\log^{-a/2}(n)\enspace . 
\eeq 
To upper bound the first moment of $|L-1|$ it is usual to apply Cauchy-Schwarz inequality and to control the second moment of the likelihood ratio. Unfortunately, the corresponding second moment is small only for $\kappa >1/2$. Here, we rely on a slight variation of this approach originally introduced by Ingster~\cite{ingster_suslina} and strengthened in~\cite{arias2014community}. We introduce a thresholded likelihood $\tilde{L}\leq L$ such that $\E_0[|L -\tilde{L}|]$ is small enough and $\E_0[\tilde{L}^2]$ is close to one.

 For any $\tau\in \cT$, define the event
 $\Gamma_{\tau}=\big\{\langle \bY , \btheta(\tau)\rangle \leq \|\btheta(\tau) \|^2 + \|\btheta(\tau) \|\sqrt{2a\log\log(n)}\big\}$. Under $\P_{\tau}$, $\langle \bY , \btheta(\tau)\rangle$ is distributed as a normal random variable with mean $\|\btheta(\tau)\|^2$ and variance $\|\btheta(\tau)\|^2$. As a consequence, we have $\P_{\tau}[\Gamma_{\tau}^c]= \overline{\Phi}(\sqrt{2a\log\log(n)})\leq \log^{-a}(n)$.
 
 Define the thresholded likelihood $\tilde{L}_{\tau}= L_{\tau}\1_{\Gamma_{\tau}}$ and $\tilde{L}=|\cT|^{-1}\sum_{\tau\in \cT} \tilde{L}_{\tau}$. We have $ \tilde{L}\leq L$ and
 \[
  \E_0[L-\tilde{L}]= \frac{1}{|\cT|}\sum_{\tau\in \cT} \E_0[L_{\tau}\1_{\Gamma_{\tau}^c}]= \frac{1}{|\cT|}\sum_{\tau\in \cT}\P_{\tau}[\Gamma_{\tau}^c]\leq \log^{-a}(n)\enspace . 
 \]

 Let us now upper bound the total variation distance between $\mathbb{P}_0$ and $\mathbf{P}$ by introducing the thresholded likelihood. 
 \begin{eqnarray}
  \E_0[|L-1|]&\leq &\E_{0}[|L-\tilde{L}|]+ \E_{0}[|\tilde{L}-1|]\leq  \E_0[L-\tilde{L}]+ \E^{1/2}_{0}[|\tilde{L}-1|^2] \nonumber \\
  &\leq & \E_{0}[L-\tilde{L}]+ \left[\E_0[\tilde{L}^2] -1 + 2 \E_{0}[L-\tilde{L}]\right]^{1/2} \nonumber \\
  &\leq & \frac{3}{\log^{a/2}(n)}+ \left[\E_0[\tilde{L}^2] -1 \right]^{1/2}\enspace ,  \label{eq:upper_second_threshold1}
 \end{eqnarray}
where we used that $L\geq \tilde{L}$  and Cauchy-Schwarz inequality. Let us work out the second thresholded moment. 
 \begin{eqnarray}\nonumber
 \E_0\left[\tilde{L}^2 \right]  &=&  \frac{1}{|\cT|^2}\sum_{\tau,\ \tau'\in \cT}\E_0\left[\tilde{L}_\tau \tilde{L}_{\tau'}\right]\\ \nonumber
 &\leq & \frac{1}{|\cT|^2}\sum_{\tau,\ \tau'\in \cT}\left[\1_{|(\tau-1)/(\tau'-1)|\in [\tfrac{1}{16}; 16]}\E_0\left[\frac{\tilde{L}^2_\tau +  \tilde{L}^2_{\tau'}}{2}\right]+ \1_{|(\tau-1)/(\tau'-1)|\notin [\tfrac{1}{16};16]}\E_0\left[L_\tau L_{\tau'}\right]\right] \nonumber\\
 &\leq & \frac{1}{|\cT|}\sum_{\tau\in \cT}\left[9\E\left[\tilde{L}_{\tau}^2\right] + \frac{1}{|\cT|}\sum_{|(\tau-1)/(\tau'-1)|\notin [\tfrac{1}{16};16]}\E_0\left[L_\tau L_{\tau'}\right]\right] \label{eq:upper_second_threshold2}\enspace , 
 \end{eqnarray}
 where we used Young's inequality and $L_{\tau} \leq L_{\tau'}$ in the second inequality. Let us first work out the right-hand side term in~\eqref{eq:upper_second_threshold2}. For $\tau'>\tau$, standard computations lead to 
 \[
  \E_0\left[L_\tau L_{\tau'}\right]= \exp\left[\langle \btheta(\tau),\btheta(\tau') \rangle\right]= e^{2(1-\kappa)\log\log(n)\sqrt{\frac{\tau-1}{\tau'-1}}} \leq e^{2\log\log(n)\sqrt{\frac{\tau-1}{\tau'-1}}}\ . 
\]
If $(\tau'-1)/(\tau-1) \geq \log^2(n)$, then we get
\[
\E_0\left[L_\tau L_{\tau'}\right]\leq e^{2\log\log(n)/\log(n)}\leq 1+ \frac{4\log\log(n)}{\log(n)}\enspace , 
\]
for $n$ large enough. 
When $(\tau'-1)/(\tau-1) >16$, then $\E_0\left[L_\tau L_{\tau'}\right]\leq \sqrt{\log(n)}$. As a consequence, 
\begin{eqnarray}\nonumber 
 \frac{1}{|\cT|}\sum_{|\tau/\tau'|\notin [\tfrac{1}{16};16]}\E_0\left[L_\tau L_{\tau'}\right]&\leq& 1+ \frac{4}{\log\log(n)}+ \frac{\sqrt{\log(n)}}{|\cT|} \big|\big\{\tau'\ , \,  \max(\frac{\tau'-1}{\tau-1}; \frac{\tau-1}{\tau'-1})\in [16;\log^2(n)]\big\}\big|\\
 &\leq& 1 + \frac{4\log\log(n)}{\log(n)}+ c \frac{\log\log(n)}{\sqrt{\log(n)}} \leq1 + c' \frac{\log\log(n)}{\sqrt{\log(n)}}\enspace . \label{eq:upper_1_second}
\end{eqnarray}
The second thresholded moment is upper bounded in the following lemma
\begin{lem}\label{lem:second_threshold}
For any $\tau\in \cT$, we have 
\[
 \E_0\left[L_\tau^2\1_{\Gamma_{\tau}}\right]\leq \log(n)^{1-a}\enspace ,
\]
\end{lem}
Since $|\cT|=\lfloor \log_2(n)\rfloor/2 $ and $a\leq 1/3$,  we conclude from~\eqref{eq:upper_second_threshold2}, \eqref{eq:upper_1_second}, and the last inequality that
$ \E_0[\tilde{L}^2]\leq 1+ c\log^{-a}(n)$ which, together with \eqref{eq:upper_second_threshold1} leads us to 
$ \E_0[|L-1|]\leq c\log^{-a/2}(n)$.
The result follows.

\begin{proof}[Proof of Lemma~\ref{lem:second_threshold}]
Let $Z$ denote  a standard random variable. Since under the null, $\langle \bY, \btheta(\tau)\rangle$ is distributed as $Z\|\btheta(\tau) \|$, we derive that
\beqn 
  \E_0\left[L_\tau^2\1_{\Gamma_{\tau}}\right]&=& \E\left[e^{2Z\|\btheta(\tau) \|^2 - \|\btheta(\tau) \|^2 }\1_{Z \leq \|\btheta(\tau) \|+ \sqrt{2a\log\log(n)}}\right]\\
  &= &\int_{-\infty}^{\sqrt{2\log\log(n)}[\sqrt{1-\kappa}+\sqrt{a}]}\phi\left(z- 2\sqrt{2(1-\kappa)\log\log(n)}\right)e^{2(1-\kappa)\log\log(n)}dz\\
  & =& \log^{2(1-\kappa)}(n)\overline{\Phi}\left[\sqrt{2\log\log(n)}[\sqrt{1-\kappa}- \sqrt{a}]\right]\\
  &\leq & \log(n)^{1- \kappa -a + 2\sqrt{(1-\kappa)a}}= \log(n)^{1-a}\enspace , 
\eeqn 
where we used in the last line that $a \leq (1-\kappa)$ for $\kappa < 2/3$, that $\overline{\Phi}(x)\leq e^{-x^2/2}$ for any $x\geq 0$ and the definition of $a$
\end{proof}

\subsection{Proof of Lemma~\ref{lem:LBCPE1}}
 For short, we write in this proof $\bP_{\tau}$ for $P_{\theta(\tau,\bmu)}$.
Fix $\tau=2$. Consider any $1/2\leq x\leq n/2-1 -4\Delta^{-2}$. 
Let $r$ denote the smallest integer larger than $4\Delta^{-2}+2x$.
Consider the test of assumptions 
\[
{\bf H0:}\ \tau^*=\tau\qquad \text{versus}\qquad {\bf H1:}\ \tau^*= \tau+r\enspace.
\] 
The Likelihood-ratio test rejects {\bf H0} when $T=r^{-1}\sum_{i=\tau}^{\tau+r-1}Y_i$ is larger than a threshold $t$. 
For symmetry reasons, the risk (sum of type I and type II error probabilities) of this test is minimal for  $t=(\mu_2+\mu_1)/2$.
If $T>t$, define $\widehat{\tau}=\tau+r$ and $\widehat{\tau}=\tau$ otherwise.
It holds that
\[
\mathbb{P}_{\tau}[\widehat{\tau}\neq \tau]+  \mathbb{P}_{\tau+r}[\widehat{\tau}\neq \tau+r]= 2\overline{\Phi}\left[\sqrt{r}\Delta/2\right]\enspace. 
\]
Let $\widehat{\tau}'$ denote any estimator of $\tau$.
Let $T'$ denote the test such that 
$T'=1$ iff $|\widehat{\tau}'-\tau|\geq r/2$.
Then
\begin{align*}
 \mathbb{P}_{\tau}[|\widehat{\tau}'- \tau|\geq r/2]+\mathbb{P}_{\tau+r}[|\widehat{\tau}'- \tau-r|\geq r/2]&\geq\mathbb{P}_{\tau}[T'=1]+\mathbb{P}_{\tau+r}[T'=0]\enspace.
\end{align*}
Let $t'$ denote the threshold such that $\mathbb{P}_{\tau}[T>t']=\mathbb{P}_{\tau}[T'=1]$.
By Neyman-Pearson's theorem, $\mathbb{P}_{\tau+r}[T'=0]\geq \mathbb{P}_{\tau+r}[T\leq t']$, hence,
\begin{align*}
 \mathbb{P}_{\tau}[|\widehat{\tau}'- \tau|\geq r/2]+\mathbb{P}_{\tau+r}[|\widehat{\tau}'- \tau-r|\geq r/2]&\geq \mathbb{P}_{\tau}[T>t']+\mathbb{P}_{\tau+r}[T\leq t']\\
&\geq \mathbb{P}_{\tau}[\widehat{\tau}\neq \tau]+  \mathbb{P}_{\tau+r}[\widehat{\tau}\neq \tau+r]\\
&= 2\overline{\Phi}\left[\sqrt{r}\Delta/2\right]\enspace.
\end{align*}
In particular, this leads us to
\[
 \inf_{\widehat{\tau}}\sup_{\tau'\in \{\tau, \tau+r\}}\mathbb{P}_{\tau'}[|\widehat{\tau}- \tau^*|\geq r/2]\geq  \overline{\Phi}\left[\sqrt{r}\Delta/2\right]\enspace. 
\]
Recall the following lower bound $\overline{\Phi}(x)\geq \tfrac{\phi(x)}{2x}$ if $x\geq 4$. Since $\sqrt{r} \Delta/2>2$ 
There exist absolute constants $c,c'>0$ such that
\beq\label{eq:localization}
\inf_{\widehat{\tau}}\sup_{\tau'\in \{\tau, \tau+r\}}\mathbb{P}_{\tau'}[|\widehat{\tau}- \tau^*|\geq r/2]\geq  c \frac{e^{-r\Delta^{2}/8}}{r^{1/2}\Delta}\geqslant c'e^{-c'r \Delta^{2}}\enspace .
\eeq
Combing back to the definition of $r$, we conclude the proof of  Lemma~\ref{lem:LBCPE1}.

\subsection{Deviations inequality and Laws of iterated logarithms}

The following results are non-asymptotic versions of the law of the iterated logarithm for sub-Gaussian random variables. 
We refer to~\cite{howard2018uniform} for proofs of (sharper versions) of Lemma~\ref{lem_lil_localiser}. 
For the sake of completeness, a proof of these lemmas is also provided in Section~\ref{sec:proof_lil}.

\begin{lem}
\label{lem:supsumGauss_log_iterated_0}
Let $\epsilon_1,\ldots,\epsilon_n$ be independent centered sub-Gaussian random variables such that $\E[e^{s\epsilon_i}]\leq e^{s^2/2}$, for any $i\geq 1$ and any $s>0$. 
Then, for any integer $d>0$, any $\alpha>0$ and any $x>0$, 
$$\P\left[\max_{k\in [d,(1+\alpha)d]} \frac{\sum_{i=1}^k \epsilon_i}{\sqrt{k}} \geq x\right]\leq \exp\left(- \frac{x^2}{2(1+\alpha)}\right).$$
\end{lem}

\begin{lem}\label{lem_lil_localiser}
Let $t_1\in\{1,\ldots,n\}$ and $ \nu >0$.
For any $t>0$, with probability larger than $1- e^{-x}$, for all $t_2\geq t_1+1/\nu $,
 \beq\label{eq:upper_Z_tau}
Z_{t_1:t_2}\leq 2\sqrt{(t_2-t_1)[\log\log(3\nu(t_2-t_1))+ x+1]}\enspace .
 \eeq
\end{lem}

\begin{lem}\label{lem:log_itere1}
For any $\alpha\in (0,1)$ and any $x>0$, with probability larger than $1-6e^{-x}$, for any $\tau \in\{2,\ldots, n\}$
\[
\bN(t_{\tau})\leq (1+\alpha)\sqrt{2\bigg(\log\log\bigg [(1+\alpha)\max\bigg\{\bigg(\tau\wedge\frac{n}\tau\bigg), \bigg(n+1-\tau\wedge\frac{n}{n+1-\tau}\bigg)\bigg\}\bigg]+3x+C_\alpha\bigg)} \enspace,
\]
where $t_\tau=(1,\tau,n+1)$ and $\bN$ is defined in \eqref{eq:definition_N}.
Here, the constant $C_\alpha$ can be chosen as follows
\[
C_\alpha=\frac{\log(1+\alpha^{-1})}{1+\alpha}-\log\log[1+\alpha]\enspace.
\]
\end{lem}

 \begin{lem}\label{lem:controle_uniforme_N_ratio}
 Fix $\tau,\tau^*\in \{2,\ldots,n\}$ and let 
 \[
 \gamma_\tau=\frac{n+1-\tau^*}{n+1-\tau}\frac{\tau-1}{\tau^*-1}\enspace.
 \]
 For any $x>0$, with probability higher than $1-2e^{-x}$, we have, uniformly over all $\tau \leq \tau^*$,
 \beq\label{eq:controle_uniforme_N_ratio}
 \bN(t_{\tau})\leq 4\sqrt{\log\log\left(e\gamma_\tau^{-1}\right)+x+1}\enspace.
 \eeq
 \end{lem}
\subsection{Proof of Proposition~\ref{prp:test_one_change-point}}\label{proof:PropTest0}

For any $t\in\mathcal{T}_3$, let 
\[
\bB(t) := \left(\frac{\sum_{i=t_2}^{t_3-1}\theta_{i}}{t_3-t_2} - \frac{\sum_{i=t_1}^{t_2-1}\theta_{i}}{t_2-t_1}\right)\sqrt{ \frac{(t_2-t_1)(t_3-t_2)}{t_3-t_1}}\enspace.
\]
Hence, $\bE(t)=|\bB(t)|$ (see the definition of $\bE(t)$ in~\eqref{eq:definition_energie}).
For any $\tau$, define the triad $t_{\tau}=(1,\tau,n+1)$. 
Basic algebra shows that (see~\eqref{eq:definition_N} for a definition of $\bN(t)$)
\begin{equation}\label{eq:ElemIneq}
 \|(\Pi_\tau-\Pi_0)\bY\|^2=\frac{(\tau-1)(n+1-\tau)}{n}\left(\frac1{\tau-1}\sum_{i=1}^{\tau-1}Y_i-\frac1{n+1-\tau}\sum_{i=\tau}^{n}Y_i\right)^2=\big(\bB(t_\tau)+\bN(t_\tau)\big)^2\enspace.
\end{equation}
A consequence of this elementary remark is that
\beq \label{eq:definition_criterion_1}
\varphi_1(\tau)= \Cr_1(\tau,\bY)- \|\bY\|^2 + \|\bPi_{0} \bY\|^2 = - \big(\bB(t_\tau)+\bN(t_\tau)\big)^2+ L^2\pen_1(\tau) \enspace.
\eeq

When $\btheta\in\Theta_0$, for any $\tau$, $\bB(t_{\tau})=0$, so $\varphi_1(\tau)= -\bN^2(t_{\tau})+ L^2\pen_1(\tau)$. 
It follows from Lemma~\ref{lem:log_itere1} that, for any $\alpha\in(0,1]$ and any $x>0$, with probability $1-12e^{-x}$, for any $\tau\in\{2,\ldots,n\}$,
\begin{multline*}
|\bN(t_{\tau})|\leq \\
(1+\alpha)\sqrt{2\bigg(\log\log\bigg [(1+\alpha)\max\bigg\{\bigg(\tau\wedge\frac{n}\tau\bigg), \bigg(n+1-\tau\wedge\frac{n}{n+1-\tau}\bigg)\bigg\}\bigg]+3x+C_\alpha\bigg)}\enspace. 
\end{multline*}
Choosing $\alpha=L-1$, this yields, for any $x>0$, with probability $1-12e^{-x}$, for any $\tau\in\{2,\ldots,n\}$,
\[
|\bN(t_{\tau})|\leq L\sqrt{\pen_1(\tau)+6x+C^*_L}\ , \quad\text{where}\quad C_L^*=\frac2{L}\log\bigg(\frac{L}{L-1}\bigg)-2\log\log(L)\enspace.
\]
In particular, for all $t>0$, with probability $1-12e^{-x}$,
\[
\boldsymbol{\varphi}_1 \geq  -L^2\big(6x+ C_L^*\big)\enspace .
\]
Choosing $x=\log(12/\alpha)$ proves the first statement of Proposition~\ref{prp:test_one_change-point}.

\bigskip

Assume now that $\btheta\in\Theta_1$. From \eqref{eq:definition_criterion_1} and the inequality $(a+b)^2\geq b^2/(L+1)-a^2/L$, valid for any $a,b\in \bbR$ and $L>0$, we derive that 
\begin{equation}\label{eq:LBTest0}
 \boldsymbol{\varphi}_1\leq \varphi_1(\tau^*)\leq - \frac{\bE^{2}_1}L+ \frac{\bN^2(t_{\tau^*})}{L+1} + L^2\pen_1(\tau^*)\enspace.
\end{equation}

Now for any $t=(t_1,t_2,t_3)\in \mathcal{T}_3$, there exist $\alpha_{t_1},\ldots,\alpha_{t_3}$ such that $\bN(t)=\sum_{i=t_1}^{t_3}\alpha_i\epsilon_i$ and
\[
\sum_{i=t_1}^{t_3-1}\alpha_i^2=\frac{(t_3-t_2)(t_2-t_1)}{t_3-t_1}\bigg(\frac1{t_3-t_2}+\frac1{t_2-t_1}\bigg)=1\enspace.
\]
Therefore, for any $s>0$,
\begin{align*}
 \E[e^{s\bN(t)}]&=\E[\prod_{i=t_1}^{t_3-1}e^{s\alpha_i\epsilon_i}]=\prod_{i=t_1}^{t_3-1}\E[e^{s\alpha_i\epsilon_i}]\leq e^{s^2\sum_{i=t_1}^{t_3-1}\alpha_i^2/2}= e^{s^2/2}\enspace.
\end{align*}
It follows that, for any $u>0$, taking $s=u/2$, $\P(\bN(t)>u)\leqslant e^{-su+s^2/2}=e^{-u^2/2}$, so 
\begin{equation}\label{eq:bN2tau_1}
\bbP\big(\bN(t_{\tau^*})\leq \sqrt{2x}\big)\geqslant 1-e^{-x},\qquad \bbP\big(\bN^2(t_{\tau^*})\leq 2x\big)\geqslant 1-2e^{-x}\enspace.
\end{equation}
Plugging this inequality in \eqref{eq:LBTest0} with $x=\log(2/\beta)$ concludes the proof of  Proposition~\ref{prp:test_one_change-point}.

\subsection{Proof of Proposition \ref{prp:tau_hat}}\label{Sec:Proof:PropPenLS1Jump}
The idea of the proof is that, with high probability, for all $\tau\in\{2,\ldots,n\}$ such that $\Delta^2|\tau-\tau^*|$ is large,
\[
\Cr_1(\tau,\bY)< \Cr_1(\tau^*,\bY)\enspace.
\]
On the same event, this implies that $\Delta^{2}|\widehat{\tau}-\tau^*|$ is small. 
Let
\[
\gamma=\frac{(\sqrt{L}-1)(L-1)}{2L^{3/2}}\in (0,1/2)\enspace.
\]
Assume that $n$ is sufficiently large to ensure that $2\gamma\sqrt{n}\geq e$. 

\bigskip 

\noindent 
{\bf Case 1: $\tau$ is far from $\tau^*$.} 
First, consider times $\tau$ satisfying
\begin{equation}\label{Ass:Case1}
\bE^2(t_{\tau})\leq \gamma \bE^{2}_1\enspace. 
\end{equation}
Define $\kappa_1= \sqrt{L}-1$. $\kappa_2= (1-L^{-1})/2$, so
\[
\gamma=\frac{\kappa_1\kappa_2}{1+\kappa_1}=\frac{\kappa_2}{1+\kappa_1^{-1}}.
\] 
From the inequality $2|xz|\leq \kappa x^2 + \kappa^{-1}z^2$ valid for any $x,z\in\bbR$ and $\kappa>0$, we derive 
\[
(x+z)^2\leq (1+\kappa)x^2+(1+\kappa^{-1}z^2),\qquad (x+z)^2\geq (1-\kappa)x^2+(1-\kappa^{-1}z^2)\enspace.
\]
Hence, from \eqref{eq:ElemIneq},
\begin{align} 
\notag\Cr_1(\tau,\bY)- \Cr_1(\tau^*,\bY)&=-\|(\Pi_\tau-\Pi_0)\bY\|^2+L\pen_1(\tau)+\|(\Pi_{\tau^*}-\Pi_0)\bY\|^2-L\pen_1(\tau^*)\\
\label{eq:Int0}&= -\big(\bB(t_\tau)+\bN(t_\tau)\big)^2+L\pen_1(\tau)+\big(\bB(t_{\tau^*})+\bN(t_{\tau^*})\big)^2-L\pen_1(\tau^*)\\
\notag&\geq- \bN^2(t_{\tau})(1+\kappa_1)+ L\pen_1(\tau) -\bE^2(t_{\tau})(1+\kappa_1^{-1})\\
\notag&\quad+ \bE^{2}(t_{\tau^*})(1-\kappa_2) + \bN^{2}(t_{\tau^*})\left(1-\kappa_2^{-1}\right)-L\pen_1(\tau^*)\enspace.
\end{align}
From Assumption \eqref{Ass:Case1}, it follows that 
\begin{align}
\label{eq:LBDeltaCr1} \Cr_1(\tau,\bY)- \Cr_1(\tau^*,\bY)\geq & - \bN^2(t_{\tau})(1+\kappa_1)+ L\pen_1(\tau) +  \bE^{2}(t_{\tau^*})\big[1-\kappa_2-\gamma(1+\kappa_1^{-1})\big] \\
\notag&+ \bN^{2}(t_{\tau^*})\left(1-\kappa_2^{-1}\right)-L\pen_1(\tau^*)\\
\notag= & - \sqrt{L}\bN^2(t_{\tau})+ L\pen_1(\tau)+ \frac1L\bE^{2}(t_{\tau^*}) -L\pen_1(\tau^*) - \bN^{2}(t_{\tau^*})\frac{L+1}{L-1}\enspace . 
\end{align}
Let us apply Lemma~\ref{lem:log_itere1} with $\alpha= L^{1/4}-1$. 
With probability larger than $1-12e^{-x}$, simultaneously for all $\tau\in\{2,\ldots,n\}$,
\beq\label{eq:upper_N2}
\bN^2(t_{\tau})L^{1/2}\leq L\pen_1(\tau)+ 6L x+C_L^*\enspace,
\eeq
where 
\[
C_L^*=L^{3/4}\log\bigg(\frac{L^{1/4}}{L^{1/4}-1}\bigg)-L\log\log(L^{1/4})\enspace.
\]
Moreover, by \eqref{eq:bN2tau_1}, $\bN^{2}(\tau^*)\leq 2x$ with probability larger than $1-2e^{-x}$. 
Plugging this bound and \eqref{eq:upper_N2} into \eqref{eq:LBDeltaCr1} shows that, with probability at least $1-14e^{-x}$, simultaneously for all $\tau\in\{2,\ldots,n\}$, such that \eqref{Ass:Case1} holds
\[
 \Cr_1(\tau,\bY)- \Cr_1(\tau^*, \bY)\geq  \frac1L\bE^{2}(t_{\tau^*}) -L\pen_1(\tau^*)  - 2x\left[3L+\frac{L+1}{L-1}\right]-C_L^*\enspace .
\]
This last expression is positive as long as 
\beq\label{condition_1_tau_far}
 \bE^{2}(t_{\tau^*}) > L^2\pen_1(\tau^*)+2Lx\left[3L+\frac{L+1}{L-1}\right]+C_L^*\enspace.
\eeq
We have proved, that with probability larger than $1-14e^{-x}$, we have $\bE^{2}(t_{\widehat{\tau}})>\gamma\bE^{2}_1$, as long as Condition~\eqref{condition_1_tau_far} is satisfied.

\bigskip

\noindent 
{\bf Case 2. $\tau$ is neither far nor close to $\tau^*$.} 
Now consider times $\tau$ such that 
\begin{equation}\label{Ass:NCloseNorFar}
\gamma<\frac{\bE^2(t_{\tau})}{\bE^{2}(t_{\tau^*})}\leq\frac12\enspace. 
\end{equation}
Starting from~\eqref{eq:Int0}, we derive that
$\Cr_1(\tau,\bY)> \Cr_1(\tau^*,\bY)$ if
\[
 \left[\bB(t_{\tau^*})+ \bN(t_{\tau^*})\right]^2 - \left[\bB(t_{\tau})+ \bN(t_{\tau})\right]^2 > L\pen_1(\tau^*)-L\pen_1(\tau)\enspace . 
\]
In particular, the above inequality is met if 
\[
 |\bB(t_{\tau^*})+ \bN(t_{\tau^*})| - |\bB(t_{\tau})+ \bN(t_{\tau})| > \sqrt{L(\pen_1(\tau^*)-\pen_1(\tau))_+}\enspace . 
\]
By the triangular inequality, it follows that $\Cr_1(\tau,\bY)> \Cr_1(\tau^*,\bY)$ if 
\beq\label{eq:t_tau_0}
 \bE(t_{\tau^*})- \bE(t_{\tau}) - |\bN(t_{\tau^*})|- |\bN(t_{\tau})|- \sqrt{L(\pen_1(\tau^*)-\pen_1(\tau))_+}>0\enspace . 
\eeq
Assume, without loss of generality, that $\tau<\tau^*$.
In this case, an important fact is that 
\[
\bE(t_{\tau})=\bE(t_{\tau^*})\sqrt{\frac{n+1-\tau^*}{n+1-\tau}\frac{\tau-1}{\tau^*-1}}\enspace.
\]
Thus, Condition \eqref{Ass:NCloseNorFar} is equivalent to 
\[
\gamma<\frac{n+1-\tau^*}{n+1-\tau}\frac{\tau-1}{\tau^*-1}\leq \frac12\enspace.
\]
In particular, as $\tau<\tau^*$, it implies that 
\beq\label{eq:CdtTau}
n+1-\tau^*>\gamma(n+1-\tau),\qquad \tau-1>\gamma(\tau^*-1)\enspace.
\eeq
We claim, that as long as $\tau$ satisfies~\eqref{eq:CdtTau} for some $\gamma\in (0,1/2)$ such that $\gamma\sqrt{n}\geq \sqrt{e}$, then 
\begin{equation}\label{eq:UBDeltaPen}
 \pen_1(\tau^*)-\pen_1(\tau)\leq c\log\log(e\gamma^{-1})\enspace\ ,
\end{equation}
for some constant $c$.

\begin{proof}[Proof of \eqref{eq:UBDeltaPen}]
Suppose that $\tau^*<n/2+1$ (the case $\tau^*\geq n/2+1$ follows by similar arguments).
In this case, as $\tau<\tau^*$, $\pen_1(\tau^*)=2\log\log(2(\tau^*\wedge n/\tau^*))$, $\pen_1(\tau)=2\log\log(2(\tau\wedge n/\tau))$. 
If $\tau^*>\sqrt{n}$, then $\pen_1(\tau^*)=2\log\log(2n/\tau^*)$.
For $\tau>\sqrt{n}$, we have  $\pen_1(\tau)=2\log\log(2n/\tau)$ and 
\[
\pen_1(\tau^*)-\pen_1(\tau)=2\log\log(2n/\tau^*)-2\log\log(2n/\tau)\leq 0\enspace.
\]
If $\tau<\sqrt{n}$, then $\tau>\gamma\sqrt{n}$, so
\[
\pen_1(\tau^*)-\pen_1(\tau)=2\log\log(2n/\tau^*)-2\log\log(2\tau)\leq 2\log\log(n) - 2\log\log(\gamma \sqrt{n})\leq c\log\log(e\gamma^{-1})\ .
\]
If $\tau^*\leq \sqrt{n}$, then $\tau\leq \sqrt{n}$ and 
\[
\pen_1(\tau^*)-\pen_1(\tau)=2\log\log(\tau^*)-2\log\log(\tau)\leq 2\log\log(1+\tau/\gamma)-2\log\log(\tau)\leq  c\log\log(e\gamma^{-1}) \enspace.
\]
This proves \eqref{eq:UBDeltaPen}. 
\end{proof}
Eq.~\eqref{eq:UBDeltaPen} implies that Condition~\eqref{eq:t_tau_0} is fulfilled if 
\beq\label{eq:t_tau_0bis}
 \bE(t_{\tau^*})- \bE(t_{\tau}) - |\bN(t_{\tau^*})|- |\bN(t_{\tau})|- \sqrt{LC_\gamma}>0\enspace . 
\eeq
Condition~\eqref{Ass:NCloseNorFar} implies that 
\[
\bE(t_{\tau^*})- \bE(t_{\tau})\geq (\sqrt{2}-1)\bE(t_{\tau^*})\enspace.
\]
By \eqref{eq:bN2tau_1}, $\bN^{2}(\tau^*)\leq 2x$ with probability higher than $1-2e^{-x}$.
Moreover, Lemma~\ref{lem:supsumGauss_log_iterated_0} applied with $d=\gamma(\tau^*-1)$ and $\alpha=1/\gamma-1$ shows that, with probability at least $1-2e^{-x}$,
\[
\forall \tau\in [\gamma\tau^*,\tau^*],\qquad Z_{1:\tau}\leq \sqrt{\frac{2(\tau-1)x}{\gamma}}\enspace.
\]
Likewise, Lemma~\ref{lem:supsumGauss_log_iterated_0} applied with $d=n+1-\tau^*$ and $\alpha=1/\gamma$ shows that, with probability at least $1-2e^{-x}$,
\[
\forall \tau\in [n+1-\tau^*,(n+1-\tau^*)/\gamma],\qquad Z_{\tau:(n+1)}\leq \sqrt{2(1+\gamma^{-1})(n+1-\tau)x}\enspace.
\]
Hence, with probability larger than $1-4e^{-x}$,
\[
|\bN(t_{\tau})|\leq  \bigg|\sqrt{\frac{n+1-\tau}{n(\tau-1)}}Z_{1:\tau}\bigg| + \bigg|Z_{\tau:n+1}\sqrt{\frac{\tau-1}{(n+1-\tau)n}}\bigg|\leq 2\sqrt{2(1+\gamma^{-1})x}\enspace.
\]

It follows that Condition~\eqref{eq:t_tau_0bis} holds with probability at least $1-6e^{-x}$ if 
\begin{align*}
 (\sqrt{2}-1)\bE(t_{\tau^*})-\sqrt{2x}(1+2\sqrt{1+\gamma^{-1}})+\sqrt{LC_\gamma}>0\enspace,
\end{align*}
that is if 
\beq\label{eq:CdtCase2}
\bE(t_{\tau^*})\geq \frac{\sqrt{2x}(1+2\sqrt{1+\gamma^{-1}})-\sqrt{LC_\gamma}}{\sqrt{2}-1}\enspace.
\eeq
The conclusion of Case 2 is that, if Condition \eqref{eq:CdtCase2} is fulfilled, with probability at least $1-6e^{-x}$, $\bE^2(t_{\widehat{\tau}})\geq\bE(t_{\tau^*})^2/2$ or $\bE^2(t_{\widehat{\tau}})< \gamma \bE(t_{\tau^*})^2$.

\bigskip 

\noindent 
{\bf Case 3}:  Let $\tau\in\{2,\ldots,n\}$ be such that   $\bE^2(t_{\tau})>  \bE^{2}(t_{\tau^*})/2$.  	
Assume moreover, without loss of generality, that $\tau < \tau^*$. 
Since 
$\bE(t_{\tau})=\bE(t_{\tau^*})\sqrt{\frac{\tau-1}{\tau^*-1}\cdot \frac{n+1-\tau^*}{n+1-\tau}}$, this implies that $\frac{\tau-1}{\tau^*-1}\cdot \frac{n+1-\tau^*}{n+1-\tau}\geq 1/2$.

We now use a slightly different approach than in the previous cases. 
Define the orthogonal projector $\bPi_{\tau,\tau^*}$ onto the space of vectors that are constant on $[1, \tau-1]$, $[\tau,\tau^*-1]$ and $[\tau^*,n]$. 
Note that, if ${\bf 1}_{a:b}$ denote the vector with coordinates $1$ for all $i\in\{a,\ldots,b-1\}$ and null otherwise, then, for any $\tau<\tau^*$ and any vector $\ba\in\bbR^n$, we have
\begin{align*}
 \bPi_{\tau} \ba&=\bigg(\frac1{\tau-1}\sum_{i=1}^{\tau-1} a_i\bigg){\bf 1}_{1:\tau}+\bigg(\frac1{n+1-\tau}\sum_{i=\tau}^n a_i\bigg){\bf 1}_{\tau:n+1}\enspace,\\
 \bPi_{\tau,\tau^*} \ba&=\bigg(\frac1{\tau-1}\sum_{i=1}^{\tau-1} a_i\bigg){\bf 1}_{1:\tau}+\bigg(\frac1{\tau^*-\tau}\sum_{i=\tau}^{\tau^*-1} a_i\bigg){\bf 1}_{\tau:\tau^*}+\bigg(\frac1{n+1-\tau^*}\sum_{i=\tau^*}^n a_i\bigg){\bf 1}_{\tau^*:n+1}\enspace.
\end{align*}
Hence,
\begin{align*}
(  \bPi_{\tau}-\bPi_{\tau,\tau^*} )\ba&=\bigg(\frac1{n+1-\tau}\sum_{i=\tau}^n a_i-\frac1{\tau^*-\tau}\sum_{i=\tau}^{\tau^*-1} a_i\bigg){\bf 1}_{\tau:\tau^*}\\
&\qquad +\bigg(\frac1{n+1-\tau}\sum_{i=\tau}^n a_i-\frac1{n+1-\tau^*}\sum_{i=\tau^*}^n a_i\bigg){\bf 1}_{\tau^*:n+1}\\
&=\bigg(\frac1{n+1-\tau^*}\sum_{i=\tau^*}^n a_i-\frac{1}{\tau^*-\tau}\sum_{i=\tau}^{\tau^*-1} a_i\bigg)\frac{n+1-\tau^*}{n+1-\tau}{\bf 1}_{\tau:\tau^*}\\
&\qquad+\bigg(\frac1{\tau^*-\tau}\sum_{i=\tau}^{\tau^*-1} a_i-\frac1{n+1-\tau^*}\sum_{i=\tau^*}^n a_i\bigg)\frac{\tau^*-\tau}{n+1-\tau}{\bf 1}_{\tau^*:n+1}\enspace.
\end{align*}
Therefore, by orthogonality of the vectors ${\bf 1}_{\tau^*:n+1}$ and ${\bf 1}_{\tau:\tau^*}$,
\begin{align}
\notag \|(  \bPi_{\tau}-\bPi_{\tau,\tau^*} )\ba\|^2&=\bigg(\frac1{\tau^*-\tau}\sum_{i=\tau}^{\tau^*-1} a_i-\frac1{n+1-\tau^*}\sum_{i=\tau^*}^n a_i\bigg)^2\bigg(\frac{(n+1-\tau^*)^2(\tau^*-\tau)}{(n+1-\tau)^2}+\frac{(\tau^*-\tau)^2(n+1-\tau^*)}{(n+1-\tau)^2}\bigg)\\
\label{eq:Pitau-Pitautau_1} &=\bigg(\frac1{\tau^*-\tau}\sum_{i=\tau}^{\tau^*-1} a_i-\frac1{n+1-\tau^*}\sum_{i=\tau^*}^n a_i\bigg)^2\frac{(n+1-\tau^*)(\tau^*-\tau)}{n+1-\tau}\enspace.
 \end{align}
 Likewise, one can show that
\begin{equation}\label{eq:Pitau-Pitautau_1bis}
  \|(  \bPi_{\tau^*}-\bPi_{\tau,\tau^*} )\ba\|^2=\bigg(\frac1{\tau^*-\tau}\sum_{i=\tau}^{\tau^*-1} a_i-\frac1{\tau-1}\sum_{i=1}^{\tau} a_i\bigg)^2\frac{(\tau-1)(\tau^*-\tau)}{\tau^*-1}\enspace.
\end{equation}
 
By Pythagoras relationship, we have  
\[ 
\Cr_1(\tau,\bY) -  \Cr_1(\tau^*,\bY)= \|( \bPi_{\tau,\tau^*}- \bPi_{\tau} ) \bY \|^2 - \|( \bPi_{\tau,\tau^*}-  \bPi_{\tau^*} )\bY \|^2 + L(\pen_1(\tau)-\pen_1(\tau^*))\enspace.
\]
From \eqref{eq:Pitau-Pitautau_1} and \eqref{eq:Pitau-Pitautau_1bis}, it follows that
\beqn 
\lefteqn{\Cr_1(\tau,\bY) -  \Cr_1(\tau^*,\bY)}&&\\
& =&  \frac{(\tau^*-\tau)(n+1-\tau^*)}{n+1-\tau}\left[\frac{\sum_{i=\tau^*}^{n} Y_i}{n+1-\tau^*} - \frac{\sum_{i=\tau}^{\tau^*- 1} Y_i}{\tau^*-\tau}\right]^2  +  \frac{(\tau^*-\tau)(\tau-1)}{\tau^*-1}\left[\frac{\sum_{i=\tau}^{\tau^*} Y_i}{\tau^*-\tau} - \frac{\sum_{i=1}^{\tau- 1} Y_i}{\tau-1}\right]^2
\\&&+ L(\pen_1(\tau)-\pen_1(\tau^*))\\
& = &  \frac{(\tau^*-\tau)(n+1-\tau^*)}{n+1-\tau}\left[\frac{Z_{\tau^*:n+1}}{n+1-\tau^*} - \frac{Z_{\tau:\tau^*}}{\tau^*-\tau} +  \Delta\right]^2-  \frac{(\tau^*-\tau)(\tau-1)}{\tau^*-1}\left[\frac{Z_{\tau:\tau^*}}{\tau^*-\tau} - \frac{Z_{1:\tau^*}-Z_{\tau:\tau^*}}{\tau-1}\right]^2
\\&&+ L(\pen_1(\tau)-\pen_1(\tau^*))\enspace.
\eeqn
From the inequalities $a^2/2-b^2\leq (a+b)^2\leq 2(a^2+b^2)$, it follows
\beqn 
\Cr_1(\tau,\bY) -  \Cr_1(\tau^*,\bY)
&\geq & \frac{(\tau^*-\tau)(n+1-\tau^*)}{n+1-\tau}\bigg(\frac{\Delta^2}2 -  \frac{4Z^2_{\tau:\tau^*}}{(\tau^*-\tau)^2}-\frac{4Z_{\tau^*:n+1}^2}{(n+1-\tau^*)^2}\bigg)\\
&&-\frac{2(\tau^*-\tau)(\tau-1)}{\tau^*-1}\bigg(\frac{Z_{1:\tau^*}^2}{(\tau-1)^2}+\frac{Z_{\tau:\tau^*}^2(\tau^*-1)^2}{[(\tau-1)(\tau^*-\tau)]^2}\bigg)\enspace.
\eeqn
Using repeatedly $\frac{\tau-1}{\tau^*-1}\cdot \frac{n+1-\tau^*}{n+1-\tau}\geq 1/2$ and $\tau < \tau^*$, this yields
\beqn 
\Cr_1(\tau,\bY) -  \Cr_1(\tau^*,\bY)&\geq&\frac{\Delta^2(\tau^*-\tau)}4-8\frac{Z^2_{\tau:\tau^*}}{\tau^*-\tau}-4(\tau^*-\tau)\frac{Z_{\tau^*:n+1}^2}{(n+1-\tau^*)^2}-2(\tau^*-\tau)\frac{Z_{1:\tau^*}^2}{(\tau-1)^2}\\
&&+L(\pen_1(\tau)-\pen_1(\tau^*))\enspace.
\eeqn 
By \eqref{eq:UBDeltaPen}, there exists an absolute constant $C$ such that $\pen_1(\tau)-\pen_1(\tau^*)\leq C$. It follows that
\begin{align*}
 \Cr_1(\tau,\bY) -  \Cr_1(\tau^*,\bY)&\geq \frac{\Delta^2(\tau^*-\tau)}4-8\frac{Z^2_{\tau:\tau^*}}{\tau^*-\tau}-4(\tau^*-\tau)\bigg(\frac{Z_{\tau^*:n+1}^2}{(n+1-\tau^*)^2}+\frac{Z_{1:\tau^*}^2}{(\tau-1)^2}\bigg)-2C\enspace .
\end{align*}

The proof of \eqref{eq:bN2tau_1} shows that, on an event of probability at least $1-4e^{-x}$, both $|Z_{\tau^*:n+1}|\leq \sqrt{2(n+1-\tau^*)x}$ and $|Z_{1:\tau^*}|\leq \sqrt{2(\tau^*-1) x}$ simultaneously. 
By Lemma~\ref{lem_lil_localiser}, with probability larger than $1-2e^{-x}$,
\[
|Z_{\tau:\tau^*}|{\bf 1}_{\{(\tau^*-\tau)\Delta^{2}\geq 1\}}\leq 2\sqrt{2(\tau^*-\tau)(\log\log[3\Delta^{2}(\tau^*-\tau)]+x+1)}
\]
Plugging these deviations inequalities in the above bound, we conclude that for, $(\tau^*-\tau)\Delta^2\geq 1$,
\beqn 
\Cr_1(\tau,\bY) -  \Cr_1(\tau^*,\bY)&\geq & (\tau^*-\tau)\left[\frac{\Delta^2}{4} - 16x\frac{n}{(n+1-\tau^*)(\tau^*-1)}\right]\\ && - 32(\log\log[3\Delta^{2}(\tau^*-\tau)]+x+1)-2C
\eeqn 
Restricting our attention to  $x \leq \bE_1^2/128 $ (which is possible by taking $c_{L}$ large enough in the statement of the proposition), this simplifies in 
\[
 \Cr_1(\tau,\bY) -  \Cr_1(\tau^*,\bY)\geq \Delta^2 \frac{(\tau^*-\tau)}{8}- 32(\log\log[3\Delta^{2}(\tau^*-\tau)]+x+1)-2C\enspace , 
\]
which is positive provided that $(\tau^*-\tau)\geq C\frac{1\vee x}{\Delta^2}$. 
Hence, there exists an absolute constant $C>0$ such that, for any $x\leq \bE_1^2/32$, with probability $1-12e^{-x}$, any $\tau$ such that $\bE^2(t_{\tau})>  \bE^{2}(t_{\tau^*})/2$ and $\Cr_1(\tau,\bY) \leq  \Cr_1(\tau^*,\bY)$ satisfies $|\tau^*-\tau|\leq C\frac{1\vee x}{\Delta^2}$.

Gathering the conclusions of Cases 1, 2 and 3, we conclude that there exist constants $C>0$ (absolute) and $C_L$ such that, with probability larger than $1-32e^{-x}$, if $\bE_1^2\geq L^2\pen_1(\tau^*)+C_L(1+x)$, then,  
\[
|\widehat{\tau}-\tau^*|\leq C\frac{1\vee x}{ \Delta^{2}}\enspace. 
\]

\subsection{Proof of Proposition~\ref{prp:IC_one_change-point}} \label{Sec:ProofICSingleJump}

All along the proof, the change-point energy is called small if 
\beq\label{eq:condition_small_energy_IC}
 \bE^{2}_1\leq L^2\pen_1(\tau^*)+ \underline{c}''_L\log\left(\frac{e}{\alpha}\right) \enspace.
 \eeq
In this expression, the constant $\underline{c}''_L$ is chosen to ensure that, when \eqref{eq:condition_small_energy_IC} does not hold, Proposition~\ref{prp:tau_hat} applies and $|\widehat{\tau}-\tau^*|\leq C\log(1/\alpha)/\Delta^{2}$ with probability higher than $1-\alpha/2$.

Proposition \ref{prp:IC_one_change-point} follows from the  three following claims.

\medskip 

 \noindent 
{\bf Claim 1}: If the change-point energy is small, with probability higher than $1-\alpha$, the test $\boldsymbol{\varphi}_{IC}$ does not reject the null. 
When $\boldsymbol{\varphi}_{IC}$ does not reject the null $I_{\widehat{\tau}}=\{2,\ldots,n\}$, so $\tau^*\in I_{\widehat{\tau}}$. 
Hence, if the change-point energy is small, with probability higher than $1-\alpha$, $\tau^*\in I_{\widehat{\tau}}$.

\medskip 

\noindent 
{\bf Claim 2}: If the change-point energy is not small, the confidence interval defined by \eqref{eq:definition_IC} contains $\tau^*$ with probability larger than $1-\alpha$.

\medskip

\noindent 
{\bf Claim 3}: If \eqref{eq:condition_large_energy_IC_2} is satisfied, with probability larger than $1-\alpha$,  $\boldsymbol{\varphi}_{IC}$ rejects the null  and $|\widehat{\Delta}|\geq |\Delta|/4$.

\medskip 

\noindent 
To conclude the proof, it remains to prove each claim. 

\medskip 

\noindent 
{\bf Proof of Claim 1}. We prove that,  with probability higher than $1-\alpha$, for all $\tau$, $\varphi_{IC}(\tau)\geq  - \underline{c}_{L,\kappa}\log(\frac{e}{\alpha})$. 
As in the proof of Proposition~\ref{prp:tau_hat}, we distinguish the cases where $\tau$ is far and close from $\tau^*$.

\bigskip 

\noindent
{\bf Case 1}: Fix $\gamma_0= (\pen_1(\tau^*))^{-1}\wedge 1$ and consider $\tau$ such that $\bE^2(t_{\tau})\leq \gamma_0 \bE^{2}_1$. 
From \eqref{eq:ElemIneq} and Lemma~\ref{lem:log_itere1} applied with $\alpha= L-1$, with probability larger than $1-\alpha/2$,
\beqn 
 \varphi_{IC}(\tau)  &\geq & - (1+\kappa^{-1})\bE^{2}(t_{\tau})-  (1+\kappa)\bN^{2}(t_{\tau})+  L^2(1+\kappa)\pen_1(\tau)\\
 &\geq & - (1+\kappa^{-1})\gamma_0 \bE^{2}_1 - (1+\kappa)L^2\left[6\log(1/\alpha)+C_{L}\right]\\
 &\geq& -(1+\kappa^{-1})\left[L^2+ \underline{c}''_{L}\log(\frac{e}{\alpha})\right]- (1+\kappa)L^2\left[6\log(1/\alpha)+C_{L}\right]\enspace.
\eeqn 
We used \eqref{eq:condition_small_energy_IC} in the third line.
There exists a constant $\underline{c}_{L,\kappa}$ such that this last lower bound is smaller than $-\underline{c}_{L,\kappa}\log(\frac{e}{\alpha})$. In conclusion, we have proved that, with probability higher than $1-\alpha/2$, one has $ \varphi_{IC}(\tau) \geq -\underline{c}_{L,\kappa}\log(\frac{e}{\alpha})$ simultaneously for all $\tau$ satisfying $\bE^2(t_{\tau})\leq \gamma_0 \bE^{2}_1$.

\bigskip 

\noindent 
{\bf Case 2}: Now consider $\tau$ satisfying $\bE^2(t_{\tau})\geq \gamma_0 \bE^{2}_1$. 
We assume that $\tau\leq \tau^*$, the case $\tau>\tau^*$ is handled similarly. 
Recall the basic inequality 
 \[
 \gamma_{\tau}= \frac{\bE^{2}(t_{\tau})}{\bE^{2}_1}= \frac{(\tau-1)(n+1-\tau^*)}{(\tau^*-1)(n+1-\tau)} \enspace .
 \]
 With this notation,
 \beqn 
 \varphi_{IC}(\tau)&\geq& -(1+\kappa) \bE^{2}(t_{\tau}) - (1+\kappa^{-1}) \bN^2(t_{\tau})+ (1+\kappa)L^2\pen_1(\tau)\\
 & \geq & -(1+\kappa)\gamma_{\tau}\bE^{2}_1-(1+\kappa^{-1})\bN^2(t_{\tau})+ (1+\kappa)L^2\pen_1(\tau)\enspace .
 \eeqn 
 We need uniform controls of both $\bN^2(t_{\tau})$ and $\pen_1(\tau)$.
By Lemma \ref{lem:controle_uniforme_N_ratio}, with probability higher than $1-\alpha/4$, we have 
 \beq\label{eq:upper_varphi_ic}
 \varphi_{IC}(\tau)
  \geq  -(1+\kappa)\gamma_{\tau}\bE^{2}_1- (1+\kappa^{-1})16\left[\log\log\left(e \gamma_{\tau}^{-1}\right)+\log\left(\frac{8}{\alpha}\right)+1\right]+ (1+\kappa)L^2\pen_1(\tau)\enspace .
 \eeq
Assume first that $\gamma_{\tau}\geq 1/2$.
In this case, by \eqref{eq:UBDeltaPen}, there exists an absolute constant $C$ such that
\[
\pen_1(\tau)\geq \pen_1(\tau^*)-C\enspace.
\]
Therefore, \eqref{eq:upper_varphi_ic} implies that
\beqn 
 \varphi_{IC}(\tau)& \geq&  (1+\kappa)\left[- \bE^{2}_1+ L^2(\pen_1(\tau^*)-C)\right] - (1+\kappa^{-1})16\left[0.6+\log\left(\frac{8}{\alpha}\right)+1\right]\enspace.
 \eeqn 
 As the change-point energy is small, this implies that
\beqn 
 \varphi_{IC}(\tau) &\geq & -(1+\kappa)\left[\underline{c}''_{L}\log(\frac{e}{\alpha})+ CL^2\right]- (1+\kappa^{-1})16\left[1.6+\log\left(\frac{8}{\alpha}\right)\right]\enspace.
\eeqn 
This last lower bound is larger than $- \underline{c}_{L,\kappa}\log(\frac{e}{\alpha})$ if we choose $\underline{c}_{L,\kappa}$ large enough.

Assume now that $\gamma_{\tau}< 1/2$. 
By \eqref{eq:UBDeltaPen}, there exists an absolute constant $C$ such that
\[
\pen_1(\tau)-\pen_1(\tau^*)\geq -C\log\log(e\gamma_0^{-1})\enspace.
\]
Therefore, \eqref{eq:upper_varphi_ic} implies that, with probability higher than $1-\alpha/4$
 \beqn 
 \varphi_{IC}(\tau)& \geq&  (1+\kappa)\left[- \frac{1}{2}\bE^{2}_1+ L^2(\pen_1(\tau^*)-C\log\log(e\gamma_0^{-1}))\right] \\
 &&\qquad - 16(1+\kappa^{-1})\left[\log\log(2e \gamma_0^{-1}) + 1+\log\left(\frac{8}{\alpha}\right)\right]\\
 &\geq &   (1+\kappa)\left[-\underline{c}''_{L}\log(\frac{e}{\alpha})+\frac{L^2}{2} \pen_1(\tau^*)\right] \\
 &&\qquad - C_{L,\kappa}\left[ \log\log\left[2e (\pen_1(\tau^*)\vee 1)\right]+ 1+\log\left(\frac{e}{\alpha}\right)\right]\enspace.
\eeqn
It remains to say that, for any positive $a$ and $b$, there exists $C_{a,b}$ such that $ax-b\log\log(2ex)\geq C_{a,b}$.
This implies that $\varphi_{IC}(\tau)\geq - \underline{c}_{L,\kappa}\log(\frac{e}{\alpha})$ if we choose $\underline{c}_{L,\kappa}$ large enough.

 Arguing simularly for $\tau\geq \tau^*$ and combining this result with Case 1, we have proved that the test $\boldsymbol{\varphi}_{IC}$ does not reject the null with probability higher than $1-\alpha$, when $\bE_1$ is small~\eqref{eq:condition_small_energy_IC}.

 \bigskip

\noindent 
{\bf Proof of Claim 2}:
Assume that $\bE_1$ is large so \eqref{eq:condition_small_energy_IC} does not hold. 
We have to prove that there exists an absolute constant $C$ such that, with probability higher than $1-\alpha$, $\tau^*$ belongs to the interval 
\[
[\widehat{\tau}-C \log(\tfrac{e}{\alpha})|\widehat{\Delta}|^{-2}; \widehat{\tau}+  C\log(\tfrac{e}{\alpha})|\widehat{\Delta}|^{-2}]\enspace.
\]

Since \eqref{eq:condition_small_energy_IC} does not hold, Proposition~\ref{prp:tau_hat} implies that there exists an absolute constant $C$ such that, for $n$ large enough, with probability larger than $1-\alpha/2$, 
\beq\label{eq:error_tau_hat_claim2}
|\widehat{\tau}-\tau^*|\leq C\frac{\log(e/\alpha)}{\Delta^{2}}\ \text{ and }\gamma_{\widehat{\tau}}\geq 1/2\enspace . 
\eeq
From Lemma \ref{lem:controle_uniforme_N_ratio}, with probability larger than $1-\alpha/4$, we have for all $\tau<\tau^*$ such that $\gamma_{\tau}\geq 1/2$
\[
 |\bN(t_{\tau})|\leq 4\sqrt{\log(8/\alpha)+1.6}=:u_{\alpha}
\]
By symmetry, the same bound holds for all $\tau> \tau^*$ such that $\gamma_{\tau}\geq 1/2$. In view of \eqref{eq:error_tau_hat_claim2}, it therefore holds for $\bN(t_{\widehat{\tau}})$ with probability larger than $1-\alpha/2$. Since
\[
\widehat{\Delta}:=   \frac{\sum_{i=\widehat{\tau}}^{n+1-\widehat{\tau}}Y_i}{n+1-\widehat{\tau}}- \frac{\sum_{i=1}^{\widehat{\tau}-1}Y_i}{\widehat{\tau}-1}= 
\Delta\left[\frac{\tau^*-1}{\widehat{\tau}-1}\wedge \frac{n+1-\tau^*}{n+1-\widehat{\tau}}\right]+ \sqrt{\frac{n}{(\widehat{\tau}-1)(n+1-\widehat{\tau})}}\bN(t_{\widehat{\tau}})\enspace ,
\]
we deduce that, with probability larger than $1-\alpha/2$,
\[
 \frac{|\Delta|}{2}-  \sqrt{\frac{n}{(\widehat{\tau}-1)(n+1-\widehat{\tau})}}u_{\alpha} \leq  |\widehat{\Delta}|\leq |\Delta|+ \sqrt{\frac{n}{(\widehat{\tau}-1)(n+1-\widehat{\tau})}}u_{\alpha}\enspace .
\]
Since $\gamma_{\widehat{\tau}}\geq 1/2$, we have 
$\sqrt{\frac{n}{(\widehat{\tau}-1)(n+1-\widehat{\tau})}}\leq \sqrt{2\frac{n}{(\tau^*-1)(n+1-\tau^*)}}$ and the previous bounds imply that 
\[
\left[\frac{1}{2}-\frac{\sqrt{2u_{\alpha}}}{\bE_1}\right] \leq  \frac{|\widehat{\Delta}|}{\Delta}\leq \left[1+\frac{\sqrt{2u_{\alpha}}}{\bE_1}\right] 
\]
If $\underline{c}''_{L}$ in \eqref{eq:condition_small_energy_IC} is sufficiently large, this implies that 
\beq\label{eq:controle_Delta_hat}
 \frac{|\widehat{\Delta}|}{\Delta}\in [1/4, 5/4]\enspace.
\eeq
Together with \eqref{eq:error_tau_hat_claim2}, this shows that $\tau^*$ belongs to the interval~\eqref{eq:definition_IC} with probability higher than $1-\alpha$, provided that the constant $\underline{c}$ is large enough.

\bigskip

\noindent 
{\bf Proof of Claim 3}. This last claim is quite straightforward. If $\bE_1$ is large enough, then $\varphi_{IC}(\tau^*)$ is small enough and the test is rejected with probability higher than $1-\alpha/2$. 
On this event,  $IC_{\widehat{\tau}}$ is defined as in~\eqref{eq:definition_IC}. 
Arguing as in Claim 2, 
we derive as in \eqref{eq:controle_Delta_hat} that $|\widehat{\Delta}|\geq |\Delta|/4$. The result follows.

\section{Proofs for multiple change-points} \label{Proofs_multiple}
\subsection{Proofs of the lower bounds}

\begin{proof}[Proof of Proposition~\ref{prp:lower_detection_multiple}]

Let $r$ in $\{1,\ldots,\lfloor n/4 \rfloor\}$, $\delta \leq \sqrt{2(1-\xi)\log(n/r)/r}$, and 
$$\cT=\left\{\lfloor n/4 \rfloor+kr+1,\ k\in\left\{0,\ldots, \left\lfloor \frac{\lfloor 3n/4 \rfloor - \lfloor n/4 \rfloor-1}{r}\right\rfloor -1\right\}\right\}\enspace.$$ For 
$\tau$ in $ \cT$, define $\btheta(\tau)$ in $\Theta_2$ by $\theta_i{(\tau)}=\delta \1_{i\in[\tau, \tau+r-1]}$. We aim at proving that, for any test $\mathscr{T}$,
 \beq\label{eq:objectif_detection_multiple}
  \P_0[\mathscr{T}=1]+ \max_{\tau \in \cT}\P_{\btheta(\tau)}[\mathscr{T}= 0]\geq 1 - c_n\enspace,
 \eeq
with $c_n=c(r/n)^{c'\xi}$ for some positive numerical constants $c$ and $c'$.  Note that~\eqref{eq:objectif_detection_multiple} implies the result of the proposition. 
As in the proof of Proposition~\ref{prp:lower_test_one_change-point} and with the same notation, we use Le Cam's approach and therefore define the mixture probability $\mathbf{P}= |\cT|^{-1}\sum_{\tau \in \cT}\P_{\tau}$. Then, one knows that Inequality~\eqref{eq:objectif_detection_multiple} holds for all test $\mathscr{T}$ if $\|\P_0-\mathbf{P}\|_{TV}$ is less or equal to $c_n$. 

Let us introducing for any $\tau$ in $\cT$, the event
 $\Gamma_{\tau}=\big\{\langle \bY , \btheta(\tau)\rangle \leq \|\btheta(\tau) \|^2 + \|\btheta(\tau) \|\sqrt{2a}\big\},$ ($a>0$) such that
  $\P_{\tau}[\Gamma_{\tau}^c]= \overline{\Phi}(\sqrt{2a})\leq e^{-a}$.\\
  Now, define the thresholded likelihood $\tilde{L}_{\tau}= L_{\tau}\1_{\Gamma_{\tau}}$ and $\tilde{L}=|\cT|^{-1}\sum_{\tau\in \cT} \tilde{L}_{\tau}$. We have $ \tilde{L}\leq L$ and
 \[
  \E_0[L-\tilde{L}]= \frac{1}{|\cT|}\sum_{\tau\in \cT} \E_0[L_{\tau}\1_{\Gamma_{\tau}^c}]= \frac{1}{|\cT|}\sum_{\tau\in \cT}\P_{\tau}[\Gamma_{\tau}^c]\leq e^{-a}\enspace . 
 \]
We first upper bound $\|\P_0-\mathbf{P}\|_{TV}$ by
\[
\|\P_0-\mathbf{P}\|_{TV}= \frac{1}{2} \E_0[|L-1|]\leq \frac{3}{2} \left[\E_{0}[L-\tilde{L}]\right]^{1/2} + \frac{1}{2}\left[\E_0[\tilde{L}^2] -1 \right]^{1/2}\leq \frac{3}{2} e^{-a/2}+ \frac{1}{2}\left[\E_0[\tilde{L}^2] -1 \right]^{1/2}\enspace.\]
Now, remark that:
\[
 \E_{0}[\tilde{L}^2 -1]= \frac{1}{|\cT|^2}\sum_{\tau, \tau'\in \cT}\left[\E_0[L_{\tau}L_{\tau'}\1_{\Gamma_{\tau}}\1_{\Gamma_{\tau'}}] - 1\right] \enspace.
\]
Recall that $L_{\tau}= \exp\left[\sum_{i=\tau}^{\tau+r-1}Y_i\tau - \frac{r\delta^2}{2}\right]$. Hence, $L_{\tau}$ is independent of $L_{\tau'}$ for $\tau\neq \tau'$. This implies that for $\tau\neq \tau'$,
\[
 \E_0[L_{\tau}L_{\tau'}\1_{\Gamma_{\tau}}\1_{\Gamma_{\tau'}}]\leq \E_{0}[L_{\tau}]\E_{0}[L_{\tau'}]= 1\enspace . 
\]
As a consequence, 
\begin{eqnarray*}
 \E_{0}[\tilde{L}^2 -1]\leq \frac{1}{|\cT|^2}\sum_{\tau\in \cT}\left[\E_0[\tilde{L}_{\tau}^2]- 1\right]\leq  \frac{1}{|\cT|^2}\sum_{\tau\in \cT}\E_{0}\left[\exp\left[\sum_{i=\tau}^{\tau+r-1}2Y_i\delta - r\delta^2\right]\1_{\sum_{i=\tau}^{\tau+r-1}Y_i\leq \sqrt{2ra}+ r\delta}\right] \enspace.
\end{eqnarray*}
Introducing $Z=r^{-1/2}\sum_{i=\tau}^{\tau+r-1}Y_i$ which follows a standard Gaussian distribution under $\P_0$, 
\beqn 
\E_{0}[\tilde{L}^2 -1]&\leq &\frac{1}{|\cT|} \E_0\left[e^{2Zr^{1/2}\delta - r\delta^2}\1_{Z\leq \sqrt{2a}+ r^{1/2}\delta}\right]= \frac{1}{|\cT|} e^{r\delta^2}\int^{\sqrt{2a}+ r^{1/2}\delta}_{-\infty}\phi(z- 2r^{1/2}\delta)dz\\
&\leq & \frac{1}{|\cT|} e^{r\delta^2}\overline{\Phi}(r^{1/2}\delta -\sqrt{2a})\\
&\leq & \frac{1}{|\cT|}e^{r\delta^2/2 -a +  \delta \sqrt{2ar}}\enspace,
\eeqn 
provided that  $a\leq r \delta^2/2$. 
Gathering everything, we have proved that 
\[\|\P_0-\mathbf{P}\|_{TV}\leq \frac{3}{2}e^{-a/2}+ \frac{1}{2|\cT|^{1/2}} e^{r\delta^2/4 -a/2 +  \sqrt{ar/2}\delta}\enspace .                                                                  \]
With $\delta=\sqrt{2(1-\xi) \log (n/r)/r}$, $n\geq 8$ and $r\leq n/4$, then $|\cT|\geq n/(8r)$. Taking $a= b\log(n/r)$ with $b\leq (1-\xi)$, this leads us to 
\begin{eqnarray*}
\|\P_0-\mathbf{P}\|_{TV}&\leq& \frac{3}{2}\left(\frac{n}{r}\right)^{-b/2}+ \sqrt{2}\left(\frac{n}{r}\right)^{- \xi/2-b/2+\sqrt{b(1-\xi)}}
\enspace.
\end{eqnarray*}
Take $b= [\xi^2/[2(1-\xi)]]\wedge (1-\xi)$ so that $\sqrt{b(1-\xi)}\leq \xi/2$ for any $\xi\in (0,1)$, which leads us to $\|\P_0-\mathbf{P}\|_{TV}\leq 3 (\tfrac{n}{r})^{-b/2}$. Finally, it suffices to observe that there exists $c'>0$ such that  $b\geq c'\xi^2$ for any $\xi$ in $ (0,1)$. 
\end{proof}

\begin{proof}[Proof of Proposition~\ref{prp:lower_loc_multiple}]
Let us first prove the lower bound~\eqref{eq:lower_risk_ke}.
As $\tau^*_k\in I_k$, both $\mu_{k-1}$ and $\mu_k$ are known to the statistician and as all other change-points do not belong to $I_k$, the statistic $(Y_i)$, $i\in I_k$ is sufficient for estimating of $\tau^*_k$. 
As a consequence, estimating $\tau^*$ boils downs to 6 a one-change-point estimation problem on $(Y_i)_{i\in I_k}$ and the results follow from Lemma~\ref{lem:LBCPE1}.

For any probability distribution $\pi$ on $\Theta[\cI,\mu]$ and any estimator $\widehat{\btau}$ of $\btau^*$,
\[
\sup_{\btheta\in \Theta[\cI,\mu]}\E_{\btheta}\left[\sum_{k=1}^K|\widehat{\tau}_k -\tau^*_k|\right]\geqslant \int \E_{\btheta}\left[\sum_{k=1}^K|\widehat{\tau}_k -\tau^*_k|\right]\pi({\rm d}\btheta)\enspace.
\]
To define a probability distribution $\pi$, fix first a sequence $(r_k)_{k\in\{1,\ldots,K\}}$ such that all $|r_k|\leq |I_k|-1$ and, for all $k\in\{1,\ldots K\}$, fix
\[
z_{k,0}=x_k+1,\qquad z_{k,1}= x_k+1+r_k\enspace.
\]
Let $\cU$ denote the uniform distribution over $\{0,1\}^K$ and, for any $\bu\in\{0,1\}^K$, let $\btheta_{\bu}\in \Theta[\cI,\mu]$ be the vector such that, for all $k\in\{1,\ldots,K\}$, $\tau^*_k=z_{k,\bu_k}$. 
Let $\pi$ denote the distribution of $\btheta_{\bU}$, where $\bU\sim\cU$.
As all $\btheta_{\bu}\in \Theta[\cI,\mu]$, the coordinates $U_1,\ldots,U_K$ of $\bU$ are independent conditionally on $\bY$.

First, consider the Wasserstein loss. The Bayes risk is achieved by the MAP estimator $\widehat{u}$. As a consequence, 
\beqn 
 \inf_{\widehat{\btau}\in \mathbb{N}^K}\sup_{\btheta\in \Theta[(I_l),\mu]}\mathbb{E}_{\btheta}\left[\sum_{k=1}^K|\widehat{\tau}_k -\tau^*_k|\right]&\geq &\sum_{k=1}^K r_k \mathbb{P}[\widehat{u}_k\neq u_k]\\
 &\geq & \sum_{k=1}^K r_k \overline{\Phi}[r^{1/2}_k|\Delta_k|/2]\enspace ,
\eeqn 
where we argued as in~\eqref{eq:lower_risk_ke}. Taking $r_k=1$ if $|\Delta_k|\geq 2$ and $r_k= \lfloor \frac{4}{\Delta_k^{2}}\rfloor$ for $|\Delta_k|\leq 2$ leads to \eqref{eq:lower_risk_l2}.

Turning to \eqref{eq:lower_risk_haussdorf}, we restrict ourselves to the case where all $r_k$'s are equal to some $0<r< \min_k |I_k|/2$. Again, one easily checks that 
\beqn 
\inf_{\widehat{\btau}\in \mathbb{N}^K}\sup_{\btheta\in \Theta[(I_l),\mu]}\mathbb{E}_{\btheta}\left[d_H(\widehat{\btau};\btau*)\right]&\geq& \frac{r}{2}\mathbb{P}[\widehat{u}\neq u]= \frac{r}{2}\left[1-(1- \overline{\Phi}[r^{1/2}|\Delta|/2])^{K}\right]\ . 
\eeqn
If $\Delta^2\geq 4$, we simply take $r=1$, which, together with $(1-x)^{K}\geq 1- Kx$, leads us to
\[
 \inf_{\widehat{\btau}\in \mathbb{N}^K}\sup_{\btheta\in \Theta[(I_l),\mu]}\mathbb{E}_{\btheta}\left[d_H(\widehat{\btau};\btau^*)\right]\geq \frac{K}{2}\overline{\Phi}(|\Delta|/2)\geq c\frac{Ke^{-\Delta^2/8}}{|\Delta|}\ , 
\]
since $\overline{\Phi}(x)\geq c x^{-1}e^{-x^2/2}$ for $x\geq 1/2$ by integration by part. If $\Delta^2\leq 4$, we take $r= \lfloor \frac{4}{\Delta^{2}} (\overline{\Phi}^{-1}[1/(4K)])^{2}\rfloor\geq 1$. We have $r < \min_{k}|I_k|/2$ since $|I_k|\geq c\log(K)/\Delta^2$ for a suitable constant $c>0$.
Hence, 
\beqn 
\inf_{\widehat{\btau}\in \mathbb{N}^K}\sup_{\btheta\in \Theta[(I_l),\mu]}\mathbb{E}_{\btheta}\left[d_H(\widehat{\btau};\btau^*)\right]&\geq& r\left(1-( 1- \frac{1}{4K})^{K}\right)\geq c'\frac{\log(K)}{\Delta^2}\ ,
\eeqn
  and the result follows.

\end{proof}

\subsection{Further Notation and preliminary Lemmas for Multiple Jumps}\label{sec:further_notation}
The purpose of this subsection is to introduce relevant quantities for evaluating the criteria differences $\Cr_{0}(\btau,\bY)- \Cr_{0}(\btau',\bY)$ for change-point vectors $\btau$ and $\btau'$ that differ at exactly one change-point. 
 For any $k=1,\ldots, K+1$, we write for short $\delta^*_{k}= \tau^*_{k}- \tau^*_{k-1}$. 
For any $q>0$, we define the function $\psi_q$, for any $\delta=(\delta_1,\delta_2)\in \{1,\ldots, n\}^2$, 
\beq\label{eq:definition_psi_q}
 \psi_q[\delta]:= \sqrt{2\log\left(\frac{n(\delta_1+\delta_2)}{\delta_1\delta_2}\right)+ q}\enspace . 
\eeq
Given $t\in \cT_3$,  we write define $\underline{\Delta}_{t}$  as the difference of means
\beq\label{eq:definition_Delta_t}
\underline{\Delta}_{t}:=\tfrac{1}{\tau_2 - \tau_1}\sum_{i=\tau_1}^{\tau_2-1}\theta_i - \tfrac{1}{\tau_3 - \tau_2}\sum_{i=\tau_2}^{\tau_3-1}\theta_i\enspace ,
\eeq
so that the energy satistfies $\bE[t]:=\sqrt{\frac{(\tau_{3}-\tau_2)(\tau_{2}-\tau_{1})}{\tau_{3}-\tau_{1}}}  |\underline{\Delta}_{t}|$. 
For $t\in \cT_3$, let $\delta(t)= (t_2-t_1,t_3-t_2)$ denote the differences. In the following lemma, we compute the difference of penalized criteria. It is a slight variation of the decomposition given in Lemma \ref{lem:decomposition_critere_tau_tau_}.

\begin{lem}[Comparison of the criteria for one change-point difference]\label{lem:comparison_one_change-point}
 Consider any change-point vector $\btau= (\tau_1, \ldots, \tau_{m})$ (with the convention $\tau_0=1$ and $\tau_{m+1}={n+1}$). For any $l\in [m]$,let $t= (\tau_{l-1},\tau_l,\tau_{l+1})\in \cT_3$. We have the following decomposition. 
\beqn
\Cr_{0}(\btau^{(-l)}, \bY) - \Cr_{0}(\btau, \bY) = \left((-1)^{sign(\underline{\Delta}_t)}\bE(t)-\bN(t)\right)^2-   L\psi^2_q(\tau_{l}-\tau_{l-1},\tau_{l+1}-\tau_l)\enspace .
\eeqn  
\end{lem}

The next lemma provides a probability bound on the noise random variables $\bN(t)$. Its proof is postponed to Section~\ref{sec:proof:concentration:Nt}.

\begin{lem}[Concentration of $\bN(\bt)$]\label{lem:concentration:N_t}
Fix any $x\in (0,1)$. With probability higher than $1-x$, one has
 \beq\label{eq:upper_bound_N_t}
 \bN^2(\bt)\leq   2\log\left(\frac{n(\delta_1(\bt)+\delta_2(\bt)}{\delta_1(\bt)\delta_2(\bt) }\right)+   c_1 \log\log\left(\frac{n(\delta_1(\bt)+\delta_2(\bt)}{\delta_1(\bt)\delta_2(\bt) x}\right) + c_2\log\left(\frac{1}{x}\right) + c_3 .
 \eeq
 simultaneously over all triads $\bt=(t_1,t_2,t_3)$ in $ \cT_3$.
\end{lem}

In the proof of Propositions \ref{prp:rough_analysis} and \ref{prp:minimal_penalty_analysis}, we fix $\ell= [(\sqrt{L}+ 9)/10]\wedge 2<\sqrt{L}$. When, $L>1$, we also have $\ell> 1$.  For any $L >1$ and $q>2$, define the event 
\beq\label{eq:definition_event_A}
\cA_{L ,q}:= \left\{|\bN(t)|\leq \ell \psi_q[\delta(t)]\enspace ,  \forall t\in \cT_3\right\}\enspace ,
\eeq
where $\psi_q$ is defined in \eqref{eq:definition_psi_q}.
Note that, for $L$ large enough, $\cA_{L,q}$ is simply the even $\cA_q$ defined in \eqref{eq:definition_A_q_large}. 
The next result is a consequence of the previous lemma.
\begin{lem}\label{lem:control_event_A}
There exist universal constants $q_0$, $c$ and $c'$ such that the following holds. 
For any $\ell>1$ and any  $q\geq q_0+ c \log[(\ell-1)^{-1}]_+$, we have 
\[
 \mathbb{P}\left[\cA_{L,q}\right]\geq 1 - e^{-c'q}\enspace .
\]

\end{lem}

Finally, we shall often use the following identity for the $\psi_q$ function~\eqref{eq:definition_psi_q}. For $\delta\neq \delta'$, one has 
\beq
 \psi_q[\delta]- \psi_q[\delta'] = \frac{2}{\psi_q(\delta)+\psi_q(\delta') }\log\left( \frac{(\delta_1+\delta_2)\delta'_1\delta'_2}{\delta_1\delta_2(\delta'_1+\delta'_2)}\right)\enspace . \label{eq:difference_penalty}
\eeq
Its proof is straightforward.

\subsection{Proof of Proposition~\ref{prp:negative_minimal_penalty} }

Let $\tau=\argmax_{i\in [n/4;3n/4]} \epsilon_i$ and define the change-point vector $\btau=(\tau,\tau+1)$. We shall prove that, with overwhelming probability, $
\Cr_0(\emptyset) > \Cr_0(\btau)$. As a consequence, the empty vector will not be the global minimal minimum and the penalized least-squares estimator $\widehat{\btau}$ selects at least one change-point.

Define $x_L= (\sqrt{L}+1)/2 \in (1/2,1)$. Note that 
\[
\P[Y_{\tau}-\theta_{\tau}\leq x_L\sqrt{2\log(n)}]= [1-\overline{\Phi}(x_L\sqrt{2\log(n)})]^{\lfloor n/2\rfloor}\leq e^{- \lfloor n/2\rfloor \log(\overline{\Phi}(x_L\sqrt{2\log(n)})) }\ . 
\]
Since $\overline{\Phi}(x)\geq ce^{-x^2/2}/x$ for any $x\geq 1$, it follows that, for a constant $\beta_L\in (0,1)$ and $n$ large enough, 
\[
 \P[Y_{\tau}-\theta_{\tau}\leq x_L\sqrt{2\log(n)}]\leq e^{-\beta_L n}\ . 
\]
Applying an union bound to all $Z_{1:i}$ with $i=n/4,\ldots, 3n/4$, we obtain that, with probability higher than 
$1-1/(2n)$, $|Z_{1:\tau}|\leq c\sqrt{n\log(n)}$. From the definition~\eqref{eq:definition_N} of $\bN(t)$, we deduce that 
\[
|\bN(1,\tau,\tau+1)|\geq \left[\frac{\sqrt{L}+1}{2}\sqrt{2\log(n)}- c'\sqrt{\frac{\log(n)}{n}}\right]\left[1-\frac{1}{\tau+1}\right]^{1/2}\geq \frac{\sqrt{L}+2}{3}\sqrt{2\log(n)}\enspace , 
\]
for $n$ large enough.  Applying two times the decomposition~\eqref{eq:decomposition_critere_tau_tau_-l}, we deduce that 
\beqn 
 \Cr_0(\emptyset) - \Cr_0(\btau)&\geq& - \bN^2(\tau,\tau+1,n+1)  -\bN^2(0,\tau,n+1)+   2Lq+ 4L\log(4)+2L\log(n)\\
 &\geq & -2\log(n)\left[\left(\frac{\sqrt{L}+2}{3}\right)^2-L\right]+2q+8\log(2) \enspace ,
\eeqn 
which is negative for $n$ large enough. The result follows.

\subsection{Proof of Propositions \ref{prp:rough_analysis} and \ref{prp:minimal_penalty_analysis} }\label{sec:proof_rough_analysis}

The proofs of these two propositions is decomposed in a few lemmas. 

\begin{lem}\label{lem:step1}
Fix any $L>1$ and any $q>2$.  Under the event $\cA_{L,q}$, the following holds
\begin{itemize}
 \item For any $k\in \{0,\ldots, K\}$,  $\widehat{\btau}$ contains at most two change-points in $[\tau_k^*;\tau^*_{k+1}]$. 
 \item Either  $\tau^*_k$ does not belong to $\widehat{\btau}$ or it is a $(\ell+ \sqrt{L},\widehat{\btau}^{(k)},q)$-high energy change-point. 
\end{itemize}

\end{lem}

\begin{lem}\label{lem:step2}
For any $L>1$ and $q>2$, there exists a constant $\kappa_L>1$ such that the following holds under the event $\cA_{L,q}$. If $\tau^*_k$ is a $(\kappa_L,q)$-high energy change-point, then 
\[
d_{H,1}(\wh{\btau},\tau^*_k)\leq\min\left[ \frac{\tau^*_{k+1}-\tau^*_k}{2}, \frac{\tau^*_{k}-\tau^*_{k-1}}{2} \right]\enspace .
\]
\end{lem}

Define $\delta_0 =  \kappa_L \Delta_k^{-2}[\log\left(n\Delta_k^2\right)+ q] $. If $2\delta_0 \leq \min(\tau^*_{k+1}-\tau^*_k, \tau^*_{k}-\tau^*_{k-1})$, then we may build a new vector $\tilde{\btau}^*= (\tau^*_1,\ldots, \tau^*_{k-1},\tau^*_k-\lceil 2\delta_0\rceil ,\tau^*_k, \tau^*_k+\lceil 2\delta_0\rceil , \tau^*_{k+1},\ldots )$. Obviously, $\btheta$ is piece-wise constant with respect to $\btau^*$. The energy of $\tau^*_k$ for this new change-point vector is lower bounded as follows 
\[
 \frac{\lceil 2\delta_0\rceil }{2}\Delta_k^{2} \geq \kappa_L \left[\log\left(n\Delta_k^2\right)+ q\right]>\kappa_L \left[\log\left(\frac{n}{\delta_0} \right)+ q\right]\geq \kappa_L \psi^2_q\left[\delta(\tau^*_k-\lceil 2\delta_0\rceil ,\tau^*_k, \tau^*_k+\lceil 2\delta_0\rceil)\right]\enspace .
\]
As a consequence,  $\tau^*_k$ is a $(\kappa_L,q)$-high energy change-point for $\tilde{\btau}^*$. From Lemma \ref{lem:step2}, we deduce that 
\[
d_{H,1}(\wh{\btau},\tau^*_k)\leq \delta_0
\]
Hence, we arrive the following proposition. It ensures that, under $\cA_{L,q}$, $\wh{\btau}$ satisfies Property {\bf (Detec[$\kappa_L,q,\kappa_L$])}. 
\begin{prp}\label{lem:step2-2}
For any $L>1$ and $q>2$, there exists a constant $\kappa_L>1$ such that the following holds under the event $\cA_{L,q}$. If $\tau^*_k$ is a $(\kappa_L,q)$-high energy change-point, then 
\[
d_{H,1}(\wh{\btau},\tau^*_k)\leq\min\left[ \frac{\tau^*_{k+1}-\tau^*_k}{2}, \frac{\tau^*_{k}-\tau^*_{k-1}}{2},  \kappa_L \frac{\log\left(n\Delta_k^2\right)+ q}{\Delta_k^{2}} \right]
\]
\end{prp}

To ensure that no spurious change-points is estimated, we need to strengthen Lemma~\ref{lem:step1}. The next two lemmas ensure, that in the vicinity of a true change-point $\tau^*_k$, $\widehat{\tau}$ does not contain two many change-points.

\begin{lem}\label{lem:step3}
For any $L>1$ and $q\geq 2$, there exists $\eta_L\in (0,1/2]$ such that the following holds under $\cA_{L,q}$.
For any $k=1,\ldots, K$, $\widehat{\btau}$ contains at most one change-point in $ [\tau^*_k;\eta_L(\tau^*_{k+1}+\tau^*_k)]$ and in $[\eta_L(\tau^*_{k}+\tau^*_{k-1}); \tau^*_k]$. Besides for $L$ large enough, we have $\eta_L=1/2$. 
\end{lem}

\begin{lem}\label{lem:step4}
There exists $L_0>1$ and $q_0$ such that, for all $L\geq L_0$ and $q\geq q_0$, $\widehat{\btau}$ satisfies {\bf (NoSp)} under the event $\cA_{L,q}$.
\end{lem}

Each of these four lemmas is proved using local improvements of $\widehat{\btau}$. More precisely, if $\widehat{\btau}$ does not satisfy any of the properties of Lemmas \ref{lem:step1}--\ref{lem:step4}, than a local modification of $\widehat{\btau}$ (insertion/deletion of a change-point) decreases the criterion $\Cr_{0}$ contradicting the optimality of $\widehat{\btau}$. This approach was already followed by Wang et al.~\cite{wang2020} but here the arguments are slightly more delicate.

\begin{proof}[Proof of Proposition \ref{prp:rough_analysis}]
 For $L$ large enough, we have $\ell=2$ and Lemma \ref{lem:control_event_A} ensures that $\mathbb{P}\left[\cA_{L,q}\right]\geq 1 - e^{-c'q}$. Since $\ell=2$, we simply have $\cA_{L,q}=\cA_q$ where 
 $\cA_q$ is defined in \eqref{eq:definition_A_q_large}. 
Since both $L$ and $q$ are large enough, we may apply Lemmas \ref{lem:step1} and \ref{lem:step4} and Proposition~\ref{lem:step2-2}. The result follows. 
\end{proof}

\begin{proof}[Proof of Proposition \ref{prp:minimal_penalty_analysis}]
It follows from Lemma \ref{lem:control_event_A} that  $\mathbb{P}\left[\cA_{L,q}\right]\geq 1 - e^{-c'q}$. We then apply  Lemmas \ref{lem:step1} and \ref{lem:step3} and Proposition~\ref{lem:step2-2} to conclude.  
\end{proof}

\subsubsection{Proof of Lemma \ref{lem:step1}} 

Consider a sequence $\btau$ such that, for some integers $l$ and $k$, we have $[\tau_{l-1},\tau_{l+1}]\subset[\tau_{k-1}^*,\tau_k^*]$. As the signal is constant over $[\tau_{l-1},\tau_{l+1})$, we have  $\underline{\Delta}_{t}=0$ (defined in \eqref{eq:definition_Delta_t})   for $t=(\tau_{l-1},\tau_{l},\tau_{l+1})$. We claim that, under $\cA_{L,q}$, $\Cr_{0}(\btau,\bY)>\Cr_{0}(\btau^{(-l)},\bY)$ which implies that $\widehat{\btau}\neq \btau$. Indeed,  Lemma \ref{lem:comparison_one_change-point} together with the definition of $\cA_{L,q}$ ensure 
\[
\Cr_{0}(\btau^{(-l)},\bY)- \Cr_{0}(\btau,\bY)=  \bN^2[t] - L \psi^2_q(\delta(t)) \leq \psi^2_q(\delta(t))[-L+ \ell]\enspace , 
\]
which is negative since $L>\ell$. 

\medskip

Turning to the second result, we consider a true change-point $\tau^*_k$ and a sequence $\btau$ such that $\tau^*_k$ is a $(\ell+ \sqrt{L},q,\btau)$-high energy change-point and $\tau^*_k$ does not belong to $\btau$. 
Let $l$ such that $\tau_l< \tau^*_k< \tau_{l+1}$. Write $t=(\tau_l,\tau^*_k,\tau_{l+1})$.  By Lemma \ref{lem:comparison_one_change-point} and the definition of $\cA_{L,q}$, 
 \beqn
 \Cr_{0}(\btau^{(k)}, \bY) - \Cr_{0}(\btau, \bY)&\leq&  - [\bE(t) - \ell\psi_q(\delta(t)) ]_+^2 + L\psi^2_q(\delta(t))\enspace , 
 \eeqn 
 which is negative as long as $\bE(t)> \psi_q(\delta(t))(\ell+\sqrt{L})$. The latter inequality holds since $\tau^*_k$ is a $(\ell+ \sqrt{L},q,\btau)$-high energy change-point. This implies that $\widehat{\btau}\neq \btau$.

 \subsubsection{Proof of Lemma \ref{lem:step2}} 

Consider any sequence $\btau$ that does not detect a high-energy change-point, say $\tau^*_k$. We shall prove that $\btau\neq \widehat{\btau}$. 
Write $r=[(\tau^*_{k+1}-\tau^*_k)\wedge (\tau^*_{k}-\tau^*_{k-1})]/2$. As a consequence, there exists an indice $l$ such that 
 \beq\label{eq:condition_lem_step2}
 \tau_l <  \tau^*_k -r <  \tau^*_k+r < \tau_{l+1}\ .
 \eeq
Define the vector $t=(\tau_{l},\tau^*_k,\tau_{l+1})$. 
\begin{eqnarray}
\bE(t)& =& \big|\overline{\theta}_{\tau^*_k:\tau_{l+1}}-\overline{\theta}_{\tau_{l}:\tau^*_k}\big|\sqrt{\frac{(\tau_k^*-\tau_l)(\tau_{l+1}- \tau_k^*)}{\tau_{l+1}-\tau_l}} \nonumber \\
& \geq &  \left[|\Delta_k| - \1_{\tau_l< \tau_{k-1}^*}|\mu_{k-1} -\overline{\theta}_{\tau_{l}:\tau^*_k}| - \1_{\tau_{l+1}> \tau_{k+1}^*}|\mu_{k} -\overline{\theta}_{\tau^*_{k}:\tau_{l+1}}| \right]\sqrt{\frac{(\tau_k^*-\tau_l)(\tau_{l+1}- \tau_k^*)}{\tau_{l+1}-\tau_l}} \label{eq:lower_ET}
\end{eqnarray}

Denote $A_{\btau}= \1_{\tau_l< \tau_{k-1}^*}|\mu_{k-1} -\overline{\theta}_{\tau_{l}:\tau^*_k}|$ and $B_{\btau}= \1_{\tau_{l+1}>\tau_{k+1}^*}|\mu_{k} -\overline{\theta}_{\tau_{k+1}^*:\tau_{l+1}}|$. We consider four subcases depending on the values of $A_{\btau}$ and $B_{\btau}$. Throughout this proof we shall often use that $\bE_k\leq |\Delta_k|\sqrt{2r}$, which holds by definition of $\bE_k$. 

\bigskip 

\noindent 
\underline{Case 1}: $A_{\btau}\vee B_{\btau}\leq |\Delta_k|/3$. It then follows from \eqref{eq:lower_ET} and that $(\tau_k^*- \tau_{l})\wedge (\tau_{l+1}-\tau_k^*)\geq r$ that 
\[
 \bE(t)\geq \frac{|\Delta_k|}{3}\sqrt{\frac{(\tau_k^*-\tau_l)(\tau_{l+1}- \tau_k^*)}{\tau_{l+1}-\tau_l}}\geq \frac{|\Delta_k|\sqrt{r}}{3\sqrt{2}}\geq \frac{\bE_k}{6}\enspace , 
\]
where we used~\eqref{eq:condition_lem_step2} and $\bE_k\leq |\Delta_k|\sqrt{2r}$. 
Then, we consider the sequence $\btau^{(k)}$ such that  $\tau^*_k$ has been inserted into $\btau$. We shall prove that $ \Cr_{0}(\btau^{(k)}, \bY) < \Cr_{0}(\btau, \bY)$. Indeed, we deduce from Lemma~\ref{lem:comparison_one_change-point} and the definition of $\cA_{L,q}$ that
\beqn 
 \Cr_{0}(\btau^{(k)}, \bY) - \Cr_{0}(\btau, \bY) & \leq   &  -\left[\bE(t)-\ell\psi_q(\delta(t))\right]_+^2+  L\left[\pen_0(\btau^{(k)}) -\pen_0(\btau)\right] \\
 &\leq & -\left[\bE_k/6 - \ell \psi_q(\delta(t))\right]^2+  L\psi^2_q(\delta(t))\enspace .
\eeqn
Since $(\tau_{l+1}-\tau^*_k)\wedge (\tau^*_{k}-\tau_l)\geq  [(\tau^*_{k+1}-\tau^*_k)\wedge (\tau^*_{k}-\tau^*_{k-1})]/2$, we derive from \eqref{eq:difference_penalty} that 
$ \psi_q(\delta(t))\leq \sqrt{\psi^2_q(\delta_k^*,\delta_{k+1}^*)+ \log(4)}< 1.5\psi_q(\delta_k^*,\delta_{k+1}^*)$ since $\psi_q(\delta_k^*,\delta_{k+1}^*)\geq q\geq 2$. 
As $\tau^*_k$ is a high-energy change-point, $\bE_k> \kappa_L \psi_q(\delta_k^*,\delta_{k+1}^*)$ is large compared to $L\psi_q(\delta_k^*,\delta_{k+1}^*)$. As a consequence, we have $\Cr_{0}(\btau^{(k)} \bY) < \Cr_{0}(\btau, \bY)$, providing that $\kappa_L \geq 9(\ell + \sqrt{L})$.

\bigskip 

\noindent 
\underline{Case 2}: $A_{\btau}\geq |\Delta_k|/3$ and $B_{\btau}\leq |\Delta_k|/3$. We then consider  the sequence $\btau^{(k-1,k)}$ such that both $\tau^*_{k-1}$ and  $\tau^*_k$ have been inserted into $\btau$. We consider two subcases
\begin{enumerate}
 \item $(\tau^*_{k-1}-\tau_{l})\geq (\tau^*_k -\tau^*_{k-1})\wedge (\tau_{l+1}-\tau^*_{k})$. Let  $t'= (\tau^*_{k-1},\tau^*_{k},\tau_{l})$. 
  Observing that we can move from $\btau$ to $\btau^{(k-1,k)}$ by first adding $\tau^*_{k-1}$ and then adding $\tau^*_k$, we get 
\begin{eqnarray}
 \nonumber
 \Cr_{0}(\btau^{(k-1,k)}, \bY) - \Cr_{0}(\btau, \bY) & \leq   &  -\left((-1)^{\sign(\underline{\Delta}_{t'})}\bE[t']+ \bN[t']\right)^2+  L\left[\pen_0(\btau^{(k-1,k)}) -\pen_0(\btau)\right] \\
 &\leq & -\left(\bE[t'] - \ell \psi_q(\delta(t'))\right)_+^2+  L\left[\pen_0(\btau^{(k-1,k)}) -\pen_0(\btau)\right]\label{eq:upper_t'}\enspace . 
\end{eqnarray}
As $\tau^*_{k-1}-\tau_{l}\geq (\tau^*_k -\tau^*_{k-1})\wedge (\tau_{l+1}-\tau^*_{k})$, the penalty difference is at most $2L\psi^2_q(\delta(t'))$. Arguing as in \eqref{eq:lower_ET}, we have 
\beqn 
 \bE[t']&\geq&  \left[|\Delta_k|  - \1_{\tau_{l+1}> \tau_{k+1}^*}|\mu_{k} -\overline{\theta}_{\tau^*_{k}:\tau_{l+1}}| \right]\sqrt{\frac{(\tau_k^*-\tau^*_{k-1})(\tau_{l+1}- \tau_k^*)}{\tau_{l+1}-\tau^*_{k-1}}} \\
 &\geq & \frac{2|\Delta_k|\sqrt{r}}{3\sqrt{2}}\geq  \frac{\bE_k}{3}\enspace .
\eeqn 
Finally, we observe  that, as for $t$ in Case 1,  $\psi_q(\delta(t'))\leq 1.5\psi_q(\delta_k^*,\delta_{k+1}^*)$. Since $\tau_k^*$ is a high-energy change-point, it then follows from \eqref{eq:upper_t'} that $\Cr_{0}(\btau^{(k-1,k)}, \bY) < \Cr_{0}(\btau, \bY)$, providing that $\kappa_L \geq 4.5(\ell + \sqrt{2L})$.
 \item $(\tau^*_{k-1}-\tau_{l})< (\tau^*_k -\tau^*_{k-1})\wedge (\tau_{l+1}-\tau^*_{k})$. Let $t_-= (\tau_{l},\tau^*_{k-1},\tau^*_{k})$. Observing that we may go from $\btau$ to $\btau^{(k-1,k)}$ by first adding $\tau^*_{k}$ and then adding $\tau^*_{k-1}$, we get 
\begin{eqnarray}
 \nonumber
 \Cr_{0}(\btau^{(k-1,k)}, \bY) - \Cr_{0}(\btau, \bY) & \leq   &  -\left((-1)^{\sign(\underline{\Delta}_{t_-})}\bE[t_-]+ \bN[t_-]\right)^2+  L\left[\pen_0(\btau^{(k-1,k)}) -\pen_0(\btau)\right] \\
 &\leq & -\left[\bE[t_-] - \ell \psi_q(\delta(t_-))\right]^2+  L\left[\pen_0(\btau^{(k-1,k)}) -\pen_0(\btau)\right]\enspace . \label{eq:upper_t'2}
\end{eqnarray}
Since $(\tau^*_{k-1}-\tau_{l})$ is small, the penalty difference is at most $2L\psi^2_q[\delta(t_-)]$. As a consequence, the expression \eqref{eq:upper_t'2} is negative as long as 
\beq  \label{eq:objective_lower_t-}
 \bE[t_-] > (\sqrt{2L}+\ell)\psi_q(\delta(t_-))\enspace .
\eeq
Working out $ \mu_{k-1}- \overline{\theta}_{\tau_{l}:\tau^*_{k}}$ and relying on the inequality $A_{\btau}\geq |\Delta_k|/3$, we get
\beqn 
\bE[t_-]&= & |\mu_{k-1}- \overline{\theta}_{\tau_{l}:\tau^*_{k-1}}|\sqrt{\frac{(\tau_{k-1}^*-\tau_{l})(\tau^*_{k}-\tau^{*}_{k-1})}{\tau^*_{k}-\tau_{l}}}=|\mu_{k-1}- \overline{\theta}_{\tau_{l}:\tau^*_{k}}|\sqrt{\frac{(\tau_{k}^*-\tau_{l})(\tau^*_{k}-\tau^{*}_{k-1})}{\tau^*_{k-1}-\tau_{l}}} 
\\
&\geq & \frac{\bE_k}{3}\sqrt{\frac{(\tau_{k}^*-\tau_{l})(\tau^*_{k+1}-\tau^{*}_{k-1})}{(\tau^*_{k-1}-\tau_{l})(\tau^*_{k+1} - \tau^*_{k})}}\\
&> & \frac{\bE_k}{3}\sqrt{\frac{\tau_k^*-\tau_{l}}{\tau_{k-1}^*-\tau_{l}}}\geq \frac{\kappa_L}{3}\psi_q[\delta_k^*,\delta_{k+1}^*] \sqrt{\frac{\tau_k^*-\tau_{l}}{\tau_{k-1}^*-\tau_{l}}}\  ,
\eeqn 
Since $\tau^*_k$ is a high-energy change-point.
By \eqref{eq:difference_penalty}, the penalty term in~\eqref{eq:objective_lower_t-} satisfies
\beqn 
\psi_q(\delta(t_-))&\leq& \sqrt{\psi^2_q[\delta_k^*,\delta_{k+1}^*]+ 2 \log\left(\frac{\tau_k^*-\tau_{l}}{\tau_{k-1}^*-\tau_{l}}\right)}\leq \psi_q(\delta_k^*,\delta_{k+1}^*)\left[1+  \frac{1}{q}\log\left(\frac{\tau_k^*-\tau_{l}}{\tau_{k-1}^*-\tau_{l}}\right)\right]\\ &\leq& \psi_q[\delta_k^*,\delta_{k+1}^*] \sqrt{\frac{\tau_k^*-\tau_{l}}{\tau_{k-1}^*-\tau_{l}}} \enspace ,
\eeqn 
where we used $q\geq 2$ and $x-1 - \log(x)>0$. Combining the two last bounds, we conclude that \eqref{eq:objective_lower_t-} holds providing that 
$\kappa_L\geq 3 (\sqrt{2L}+ \ell)$.

\end{enumerate}

\bigskip 

\noindent 
\underline{Case 3}: $A_{\btau}<|\Delta_k|/3$ and $B_{\btau}\geq |\Delta_k|/3$. We  consider  the sequence $\btau^{(k,k+1)}$ and argue as in Case 2.

\bigskip

\noindent 
\underline{Case 4}: $A_{\btau}\geq |\Delta_k|/3$ and $B_{\btau}\geq |\Delta_k|/3$. We consider the sequence $\btau^{(k-1,k,k+1)}$ by inserting all three true change-points into $\btau$. Define $t_+= (\tau_k^*,\tau_{k+1}^*,\tau_{l+1})$. Depending in which order we add the three change-points, we have the following bounds
\beq\label{eq:diff_crit_case4}
\Cr_{0}(\btau^{(k-1,k,k+1)}, \bY) - \Cr_{0}(\btau, \bY) \leq     -\max_{t=t_-,t^*_k,t_+}\left((-1)^{\sign(\underline{\Delta}_{t})}\bE(t)+ \bN(t)\right)^2+  L\left[\pen_0(\btau^{(k-1,k,k)}) -\pen_0(\btau)\right]
\eeq
Then, we focus on $t_-$, $t^*_k$, $t_+$, depending on the relative values of $(\tau^*_{k-1}-\tau_{l})$,  $(\tau^*_k -\tau^*_{k-1})\wedge (\tau_{k+1}^* -\tau^*_{k})$, and $(\tau_{l+1}-\tau_{k^*})$. If $(\tau^*_k -\tau^*_{k-1})\wedge (\tau_{k+1}^{*}-\tau^*_{k})$ is the smallest quantity of those three, it follows from \eqref{eq:diff_crit_case4} that 
\[
\Cr_{0}(\btau^{(k-1,k,k+1)}, \bY) - \Cr_{0}(\btau, \bY) \leq    -(\bE_{k}-\ell\psi_q(\delta_k^*,\delta_{k+1}^*))^2+ 3L\psi^2_q(\delta_k^*,\delta_{k+1}^*) \enspace , 
\]
which is negative since $\tau_k^*$ is a high-energy change-point. If $(\tau^*_{k-1}-\tau_{l})$ is smallest difference, then we focus on $t_{-}$. The corresponding penalty is the largest one and we have 
\[
 \Cr_{0}(\btau^{(k-1,k,k+1)}, \bY) - \Cr_{0}(\btau, \bY) \leq    (\bE[t_-]- \ell\psi_q(\delta(t_-)))^2+ 3L\psi^2_q(t_-) \enspace . 
\]
Arguing as in the second part of Case 2, we derive that $\bE[t_-]\geq \tfrac{1}{3}\bE_k\sqrt{\frac{\tau_k^*-\tau_{l}}{\tau_{k-1}^*-\tau_{l}}}$ which is large compared to $(\sqrt{3L}+\ell)\psi_q(\delta(t_-))$. The case where $(\tau_{l+1}-\tau_{k+1}^*)$ is the smallest difference is handled similarly by focusing on $t_+$.

 \subsubsection{Proof of Lemma \ref{lem:step3}}

 Consider some $\eta_L\in (0,1/2]$ whose value will be fixed later. 
Take  any $\btau$ such that, for some $k$ and some $l$, 
$ \tau^*_k\leq \tau_{l-1}< \tau_{l} \leq  \eta_L (\tau^*_{k+1}+\tau^*_k) < \tau^*_{k+1}$. We shall prove that $\btau\neq \widehat{\btau}$. If $\tau_{l+1}\leq \tau^*_{k+1}$, then Lemma \ref{lem:step1} already enforces that $\btau\neq \widehat{\btau}$. We assume henceforth that $\tau_{l+1}> \tau_k^*$. Consider the vector $\btau^{(-l)}$ and define $t= (\tau_{l-1},\tau_l,\tau_{l+1})$.
 By Lemma \ref{lem:comparison_one_change-point} and the definition of $\cA_{L,q}$, we have
\beqn
\Cr_{0}(\btau^{(-l)}, \bY) - \Cr_{0}(\btau, \bY)&\leq &  \left[(-1)^{\sign(\underline{\Delta}_{t})}\bE(t)+\bN(t) \right]^2- L\psi^2_q[\delta(t)] \\
&\leq &  \left[\bE(t)+  \ell \psi_q(\delta(t))   \right]^2- L\psi^2_q[\delta(t)]\enspace . 
\eeqn 
If $\bE(t)$ is small enough so that the right-hand side is negative, we have $\widehat{\btau}\neq \btau$. As a consequence, we only need to deal with the case where $\bE(t)$ is large. Henceforth, we assume that 
\[
 \bE(t)\geq  [\sqrt{L}-\ell]\psi_q[\delta(t)]\enspace .
\]
 In view of the definition of $\bE(t)$, this  is equivalent to
\[
(\overline{\theta}_{\tau_{k+1}^*:\tau_{l+1}}- \mu_{k+1})^2\geq  \frac{(\tau_{l+1}-\tau_l)(\tau_{l+1}-\tau_{l-1}) }{(\tau_{l+1}-\tau^*_{k+1})^2(\tau_{l}-\tau_{l-1})}[\sqrt{L}-\ell]^2\psi^2_q[\delta(t)]\enspace . 
\]
Introduce the vector $\btau^{(-l,k+1)}:= (\tau_1,\ldots, \tau_{l-1},\tau^*_{k+1},\tau_{l+1},\ldots)$. Define $u= (\tau_{l-1},\tau_l,\tau^{*}_{k+1})$,  $v= (\tau_{l},\tau^{*}_{k+1},\tau_{l+1})$, and
\[
z:= \frac{(\tau_{l+1}-\tau_{l-1})(\tau^*_{k+1}-\tau_{l}) }{(\tau_{l+1}-\tau_{k+1}^*)(\tau_{l}-\tau_{l-1})}\geq  \frac{(\tau^*_{k+1}-\tau_{l}) }{(\tau_{l}-\tau_{l-1})} \geq \frac{1}{\eta_L}-1 \enspace \ ,  
\]
since $\tau^*_k\leq \tau_{l-1}< \tau_{l} \leq  \eta_L (\tau^*_{k+1}+\tau^*_k)$. 
From the definition of $\bE(v)$ together with the previous bound, we derive that
\beq\label{eq:lower_Ev}
\bE^2(v) \geq z(\sqrt{L}-\ell)^2\psi^2_q[\delta(t)]\enspace .
\eeq
Since $\tau_{l+1}-\tau_l\geq \tau^*_{k+1}-\tau_l\geq \tau_l-\tau_{l-1}$, we deduce that $\psi^2_q(\delta(u))\leq \psi^2(\delta(t))+ \log(4)$.
 Applying Lemma \ref{lem:comparison_one_change-point} and relying on the event $\cA_{L,q}$, we obtain
\beqn
\lefteqn{\Cr_{0}(\btau^{(-l,k+1)}, \bY) - \Cr_{0}(\btau, \bY)}&&\\&=& \Cr_{0}(\btau^{(-l,k+1)}, \bY) - \Cr_{0}(\btau^{(k+1)}, \bY)+\Cr_{0}(\btau^{(k+1)}, \bY) - \Cr_{0}(\btau, \bY)\\
& \leq&   \bN^2(u) - \left( (-1)^{\sign(\underline{\Delta}_{v})}\bE(v)+ \bN(v)\right)^2+L\left[\pen_0(\btau^{(-l,k+1)}) - \pen_0(\btau)\right]\enspace , \\
& \leq&   \ell^2\psi^2_q[\delta(u)] - \left( \bE(v) -  \ell  \psi_q[\delta(v)]\right)_+^2+L\left[\pen_0(\btau^{(-l,k+1)}) - \pen_0(\btau)\right]\\
&< &   \ell^2 [\psi^2_q[\delta(t)]+\log(4)] - \left( \sqrt{z}(\sqrt{L}-\ell)\psi_q[\delta(t)] -  \ell  \psi_q[\delta(v)]\right)_+^2+L\left[\pen_0(\btau^{(-l,k+1)}) - \pen_0(\btau)\right]\enspace .
\eeqn 
 Let us compute the penalty difference.
\beqn 
\pen_0(\btau^{(-l,k+1)}) - \pen_0(\btau)&=&  \psi^2_q[\delta(v)]- \psi^2_q[\delta(u)] \\
&= & 2\log\left(\frac{(\tau_{l+1}-\tau_l)(\tau_l-\tau_{l-1})}{(\tau_{l+1}-\tau_{k+1}^*)(\tau_{k+1}^*- \tau_{l-1})}\right)\ .
\eeqn
Since $\tau^*_k\leq \tau_{l-1} <\tau_l\leq (\tau^*_k+\tau^*_{k+1})/2< \tau^*_{k+1} \leq \tau_{l+1}$, one easily checks that the above expression is less or equal to $2\log(z)$. Hence, we also have $\psi^2_q[\delta(v)]\leq \psi^2_q[\delta(t)]+ 2\log(z)$.
We derive from the previous inequalities that
\beqn 
\frac{ \Cr_{0}(\btau^{(-l,k+1)}, \bY) - \Cr_{0}(\btau, \bY)}{\psi^2_q[\delta(t)]}& \leq& - \left[z^{1/2}(\sqrt{L}-\ell)- \ell\sqrt{1+ \frac{2\log(2z)}{\psi^2_q(\delta(t))}}\right]_+^2 + \ell^2 +\frac{2L\log(z)+ 2\ell^2(\log(2))}{\psi^2_q[\delta(t)]}\\
&\leq & - \left[z^{1/2}(\sqrt{L}-\ell)- \ell\sqrt{1+ \log(2z)}\right]_+^2+  \ell^2+ L\log(2z)\enspace ,
\eeqn 
since $\psi^2_q[\delta(t)]\geq q\geq 2$ and $\ell^2\leq L$. 
Coming back to the definition of $\ell$ (Section \ref{sec:further_notation}), we observe that the last expression is  negative as soon as  
\begin{eqnarray}
\label{eq:objective_negative}
 0.9 \sqrt{z}(\sqrt{L}-1)&>&2\sqrt{L}\sqrt{1+\log(2z)}\enspace , \quad \text{ if }\sqrt{L}\leq 11\enspace ,\\
 \sqrt{z}(\sqrt{L}-2) &>&\sqrt{4+L\log(2z)}+ 2\sqrt{1+ \log(2z)}\enspace , \quad \text{ if }\sqrt{L}> 11 .\label{eq:objective_negative2}
\end{eqnarray}
Recall that $z\geq \eta_{L}^{-1}-1$. 
In summary we only have to prove, that for all $L>1$, there exists $\eta_L\in (0,1/2]$ such that \eqref{eq:objective_negative}  or \eqref{eq:objective_negative2} holds for all $z\geq \eta_L^{-1}-1$ and that, for $\sqrt{L}>11$ large enough, \eqref{eq:objective_negative2} holds for all $z\geq 1$ (which corresponds to taking $\eta_L=1/2$).

The function $h: z\mapsto \tfrac{z}{1+\log(2z)}$ is increasing for $z\geq 1$. If  $\eta_L\in (0,1/2)$ is chosen  in such a way $h(1/\eta_L-1)>4L[0.9(\sqrt{L}-1)]^{-2}$, then \eqref{eq:objective_negative} is satisfied.

\medskip 

Now consider the large $L$ regime $(\sqrt{L}>11)$. Since $1+\log(2z)\leq 2z$,  \eqref{eq:objective_negative2} is satisfied as soon as 
\beq \label{eq:objective_negative3} 
 z [\sqrt{L} -2-2\sqrt{2}]^2 - 4 -  L\log(2z) > 0\enspace 
\eeq
For a fixed $\sqrt{L}>11$, the above expression is positive for $z$ large enough. Hence, there exists $\eta_L\in (0,1/2)$ such that~\eqref{eq:objective_negative3} holds for all $z\geq \eta_L-1$. Beside, using a derivation argument, we derive that Inequality~\eqref{eq:objective_negative3} holds for all $z\geq 1$ as long as 
\[
L(1-\log(2))-4 + 2L\log\left(1- \frac{2+2\sqrt{2}}{\sqrt{L}}\right) >0\enspace ,
\]
which holds for $L$ large enough.  This concludes the proof.

 \subsubsection{Proof of Lemma \ref{lem:step4}}

 We consider the specific cases $k=1$ and $k=K$ at the end of this proof. 
 Consider any vector $\btau$ such that, for some $k$ and $l$, $\tau_{l-1}$ and $\tau_{l}$ both belong to $[(\tau^*_{k-1}+\tau^*_k)/2; (\tau^*_{k}+\tau^*_{k+1})/2]$. Since $L$ is chosen large enough, we may assume that $\eta_L$ in Lemma \ref{lem:step3} is equal to $1/2$. Applying this lemma, we have $\btau\neq \widehat{\btau}$ unless
 \[
  \frac{\tau^*_{k-1}+\tau^*_k}{2} \leq \tau_{l-1} < \tau^*_k < \tau_{l}\leq \frac{\tau^*_{k}+\tau^*_{k+1}}{2}
 \]
 By symmetry, we may assume that $\tau_{l}-\tau^*_k\geq \tau^*_{k}-\tau_{l-1}$. As a warm-up, we consider the simpler situation where $\tau_{l+1}\leq \tau^*_{k+1}$. We shall prove that either $\btau^{(-l)}$ or $\btau^{(-l,k)}$ achieves a smaller criterion value than $\btau$ which as usual implies that $\btau\neq \widehat{\btau}$. Write $t=(\tau_{l-1},\tau_l, \tau_{l+1})$, $u=(\tau^*_k,\tau_l,\tau_{l+1})$, and $v= (\tau_{l-1},\tau^*_k, \tau_{l})$. We have 
\[
\Cr_{0}[\btau^{(-l)}] - \Cr_{0}[\btau] \leq   (\bE(t) + \ell \psi_q[\delta(t)])^2   - L\psi^2_q[\delta(t)] \enspace . 
\]
If $\bE(t)< \psi_q[\delta(t)](\sqrt{L} - \ell)$, then $\Cr_{0}[\btau^{(-l)}]< \Cr_{0}[\btau]$ and $\btau\neq \widehat{\btau}$. Now assume that $\bE(t)$ is large. Considering the local modification $\btau^{(-l,k)}$, we obtain
\begin{equation}
\Cr_{0}[\btau^{(-l,k)}] - \Cr_{0}[\btau]\leq   - [\bE(v) -\ell  \psi_q[\delta(v)]]_+^2  +  \ell^2\psi^2_q[\delta(u)]  + L\left[\pen_0(\btau^{(-l,k) }) - \pen_0(\btau) \right] \ .  \label{eq:upper1}
\end{equation}
  We shall prove that the rhs in \eqref{eq:upper1} is negative. From the definition of energies, we derive 
\[ 
\bE^2(t) = \Delta_k^2 \frac{(\tau_{l+1}-\tau_l)(\tau_k^*-\tau_{l-1})^2 }{(\tau_{l+1} - \tau_{l-1})(\tau_l-\tau_{l-1})}=  \bE^2(v) \frac{(\tau_k^*-\tau_{l-1})(\tau_{l+1}-\tau_l)}{(\tau_l-\tau_k^*)(\tau_{l+1}-\tau_{l-1})}\enspace . 
\]
Define $z= \frac{(\tau_l-\tau_k^*)(\tau_{l+1}-\tau_{l-1})}{(\tau_k^*-\tau_{l-1})(\tau_{l+1}-\tau_l)}>1$. So that $\bE(v)= \sqrt{z}\bE(t)> \sqrt{z}\psi_q[\delta(t)](\sqrt{L} - \ell)$. Also 
\[
 \psi^2_q[\delta(u)] -\psi^2_q[\delta(t)]= 2 \log\left(\frac{(\tau_{l+1}-\tau^*_k)(\tau_{l}-\tau_{l-1})}{(\tau_{l+1}-\tau_{l-1})(\tau_{l}-\tau^*_k)}\right)\leq 2\log(2)\ ,
\]
since $\tau_{l}-\tau_{l-1}\leq 2(\tau_{l}-\tau^*_{k})$. 
By definition of the penalty function we have 
\beqn 
\pen_0(\btau^{(-l,k) }) - \pen_0(\btau) &= &  \psi^2_q[\delta(v)]- \psi^2_q[\delta(u)] =  2\log\left(\frac{(\tau_{l+1}-\tau_l)(\tau_l-\tau_{l-1})}{(\tau_{l+1}-\tau_{k}^*)(\tau_{k}^*- \tau_{l-1})}\right) \\
& =& 2 \log\left(\frac{(\tau_l-\tau_k^*)(\tau_{l+1}-\tau_{l-1})}{(\tau_{l+1}-\tau_l)(\tau_k^*-\tau_{l-1})}\right)+ 2 \log \left(\frac{(\tau_{l+1}-\tau_{l})^2(\tau_{l}-\tau_{l-1})}{(\tau_{l+1}-\tau_k^*)(\tau_{l+1}-\tau_{l-1})(\tau_l-\tau_k^*)}\right)\\
&\leq& 2 \log(z)+2\log(2)\enspace ,
\eeqn 
since $\tau_l-\tau_{l-1}\leq 2 (\tau_l-\tau_k)^*$. Hence, $\psi^2_q[\delta(v)]\leq \psi^2_q[\delta(t)]+ 2\log(4z)$.
Coming back to \eqref{eq:upper1}, this leads us to 
\begin{eqnarray}\nonumber
\frac{\Cr_{0}[\btau^{(-l,k)}] - \Cr_{0}[\btau]}{\psi^2_q[\delta(t)]}&\leq &  - \left[\sqrt{z}(\sqrt{L}-\ell)- \ell\left(1+ \frac{2\log(4z)}{\psi^2_q(\delta(t))}\right) ^{1/2}\right]_+^2  +    \ell^2 \left(1+ \frac{2\log(2)}{\psi^2_q[\delta(t)]}\right) + \frac{2L\log\left(2z\right)}{\psi^2_q[\delta(t)]}\\
&\leq & - \left[\sqrt{z}(\sqrt{L}-\ell)- \ell\left(1+ \frac{2\log(4z)}{q}\right) \right]_+^2  +   \ell^2(1+\frac{2}{q}\log(2))+  \frac{ 2L\log(2z)}{q} \enspace , \nonumber \\
&\leq & - \left[\sqrt{z}(\sqrt{L}-2)- 2\sqrt{1+ \log(4z)} \right]_+^2  +   4(1+\log(2))+  L\log(2z)\ ,
\end{eqnarray}
since $q\geq 2$ and $\ell= 2$ for $L\geq \sqrt{11}$. Hence, it only  remains to prove that, for all $L$ large enough, all $z\geq 1$,
\[
 \sqrt{z}(\sqrt{L}-2)> 2\sqrt{1+ \log(4z)} + 2\sqrt{1+\log(2)}+ \sqrt{L\log(2z)}\ . 
\]
Consider any $\zeta\in (0,1)$. There exists $L_{\zeta}$ such that,  for all $L>L_{\zeta}$ and all $z\geq 1$, 
\[
\zeta \sqrt{zL}>2\sqrt{z}+  2\sqrt{1+ \log(4z)} + 2\sqrt{1+\log(2)}\ . 
\]
Hence, it suffices to prove that there exists  $\zeta_0\in (0,1)$ such that for all $z\geq 1$, we have
$(1-\zeta_0)^2 z  \geq \log(2z)$. Deriving this expression with respect to $z$, we conclude that the latter inequality holds as long as $2\log(1-\zeta_0)^{-1})\leq \log(e/2)$. Taking $\zeta_0= 1-\sqrt{2/e}$ concludes this part.

\bigskip 

Now, we consider the more challenging (and painful) situation where $\tau_{l+1}> \tau_{k+1}^*$. We shall prove, by deleting $\tau_{l}$ and possibly inserting $\tau^*_{k}$ or $\tau^*_{k+1}$, the penalized criterion will decrease. In other words,  at least one the sequences  $\btau^{(-l)}$ $\btau^{(-l,k)}$ and $\btau^{(-l,k+1)}$ achieves a smaller criterion than $\btau$. It suffices to show that $\Cr_{0}[\btau] > \Cr_{0}[\btau^{(-l)}]$ as long as
$\Cr_{0}[\btau]\leq \Cr_{0}[\btau^{(-l,k)}]\wedge \Cr_{0}[\btau^{(-l,k+1)}]$. We assume in the remainder of this proof that the latter inequality holds. Under the event $\cA_{L,q}$,  Lemma \ref{lem:comparison_one_change-point} enforces that 
\beqn
\lefteqn{\Cr_{0}[\btau^{(-l,k)}] - \Cr_{0}[\btau]}&&\\ &\leq &  - \left(\bE[(\tau_{l-1},\tau_k^*,\tau_{l})] - \ell \psi_q[\delta(\tau_{l-1},\tau_k^*,\tau_l)]\right)^2_+ +    \left(\bE(\tau_{k}^*,\tau_{l},\tau_{l+1}) +  \ell\psi_q[\delta(\tau_{k}^*,\tau_{l},\tau_{l+1})]\right)^2 \\ & &  + L\left(\psi^2_q[\delta(\tau_{l-1},\tau_k^*,\tau_{l})] - \psi^2_q[\delta(\tau_{k}^*,\tau_{l},\tau_{l+1})]\right)\enspace ; \\
\lefteqn{\Cr_{0}[\btau^{(-l,k+1)}] - \Cr_{0}[\btau]}&&\\ &\leq &  - \left(\bE(\tau_{l},\tau_{k+1}^*,\tau_{l+1}) - \ell\psi_q[\delta(\tau_{l},\tau_{k+1}^*,\tau_{l+1})]\right)^2_+  +      \left(\bE(\tau_{l-1},\tau_l,\tau_{k+1}^*) +  \ell\psi_q[\delta(\tau_{l-1},\tau_{l},\tau^*_{k+1})]\right)^2 \\
&&  + L\left(\psi^2_q[\delta(\tau_{l},\tau_{k+1}^*,\tau_{l+1})] - \psi^2_q[\delta(\tau_{l-1},\tau_{l},\tau^*_{k+1})]\right) \ .
\eeqn 
From \eqref{eq:difference_penalty} and the respective positions of the change-points, we derive that 
\[
\psi^2_q[\delta(\tau_{l-1},\tau_k^*,\tau_{l})] - \psi^2_q[\delta(\tau_{k}^*,\tau_{l},\tau_{l+1})]= 2\log\left(\frac{(\tau_{l+1}-\tau_{l})(\tau_{l}-\tau_{l-1})}{(\tau_{l+1}-\tau_{k}^*) (\tau^*_{k}-\tau_{l-1})}\right) \leq 2\log\left(\frac{\tau_l-\tau_{l-1}}{\tau_k^*-\tau_{l-1}}\right)\enspace .
\]
Similarly, we get
\[
\psi^2_q[\delta(\tau_{l},\tau_{k+1}^*,\tau_{l+1})] - \psi^2_q[\delta(\tau_{l-1},\tau_{l},\tau^*_{k+1})]= 2\log\left(\frac{(\tau_{l}-\tau_{l-1})(\tau_{l+1}-\tau_{l})}{(\tau_{l+1}-\tau_{k+1}^*)(\tau^*_{k+1}-\tau_{l-1})}\right) \leq 2\log\left(\frac{\tau_{l+1}-\tau_{l}}{\tau_{l+1}-\tau_{k+1}^*}\right). 
\]
Since $\Cr_{0}[\btau]\leq \Cr_{0}[\btau^{(-l,k)}]\wedge \Cr_{0}[\btau^{(-l,k+1)}]$, since   $\ell\leq \sqrt{L}$, the above inequalities lead us to 
\begin{eqnarray} 
 \bE(\tau_{l-1},\tau_k^*,\tau_{l})&\leq& \bE(\tau_{k}^*,\tau_{l},\tau_{l+1})+ \omega_1\,  \label{eq:cond2} \enspace ;\\
 \bE(\tau_{l},\tau_{k+1}^*,\tau_{l+1})&\leq &\bE(\tau_{l-1},\tau_l,\tau_{k+1}^*)+\omega_2 \label{eq:cond3} \enspace ;\\ \nonumber
 \omega_1&:=& 2\ell \psi_q[\delta(\tau_{k}^*,\tau_{l},\tau_{l+1})]
  +2\sqrt{2L \log\left(\frac{\tau_l-\tau_{l-1}}{\tau_k^*-\tau_{l-1}}\right)}  \enspace ; \\
 \omega_2& :=&  2 \ell \psi_q[\delta(\tau_{l-1},\tau_l,\tau_{k+1}^*)]
  +2\sqrt{2L \log\left(\frac{\tau_{l+1}-\tau_{l}}{\tau_{l+1}-\tau_{k+1}^*}\right)} \enspace . \nonumber 
\end{eqnarray}
We aim at  proving that $\Cr_{0}(\btau^{(-l)}) < \Cr_{0}(\btau)$. Define $t= (\tau_{l-1},\tau_l,\tau_{l+1})$.  In view of Lemma \ref{lem:comparison_one_change-point}, it suffices to establish the following inequality
\beq
\label{eq:cond4}
 \bE(t) <  \psi_q[\delta(t)]\left(\sqrt{L} - \ell\right) \enspace , 
\eeq
 We start with the following lemma that relates the change-point energies. 
\begin{lem} \label{lem:energie}
It holds that 
\begin{eqnarray}\label{eq:upper_1_energie}
&&\frac{\bE(\tau^*_k,\tau_{l},\tau_{l+1})\bE(\tau_{l-1},\tau_{l},\tau^*_{k+1}) }{\bE(\tau_{l},\tau^*_{k+1},\tau_{l+1})\bE(\tau_{l-1},\tau^*_k,\tau_{l})}= \sqrt{\frac{(\tau_{l+1}-\tau_{k+1}^*)(\tau^*_k-\tau_{l-1})}{(\tau_{l+1}-\tau_k^*)(\tau^*_{k+1}-\tau_{l-1})}}\leq \sqrt{\frac{1}{3}}
\enspace ; \\
\label{eq:upper_2_energie}
&&\frac{\bE(\tau^*_k,\tau_{l},\tau_{l+1}) }{\bE(\tau_{l},\tau^*_{k+1},\tau_{l+1})}\leq  \sqrt{\frac{\tau_{l+1}-\tau^*_{k+1}}{\tau_{l+1}-\tau^*_k}} \enspace ; \quad \frac{\bE(\tau_{l-1},\tau_{l},\tau^*_{k+1}) }{\bE(\tau_{l-1},\tau^*_k,\tau_{l})} \leq  \sqrt{\frac{\tau^*_k-\tau_{l-1}}{\tau_{l}-\tau^*_k} }\enspace ; \\
\bE(t)& \leq&  \bE(\tau_{l-1},\tau_k^*,\tau_l)\sqrt{\frac{\tau^*_k-\tau_{l-1}}{\tau_l-\tau^*_k }}+ \bE(\tau_{l},\tau_{k+1}^*,\tau_{l+1})\sqrt{2\frac{\tau_{l+1}-\tau^*_{k+1}}{\tau_{l+1}-\tau_{l-1}}}   \enspace .  \label{eq:upper_3_energie}
\end{eqnarray}
\end{lem}
It follows from (\ref{eq:upper_1_energie}--\ref{eq:upper_2_energie})  together with \eqref{eq:cond2} and \eqref{eq:cond3} that $\bE(\tau_{l-1},\tau^*_k,\tau_{l})$ is small. Indeed, we have
\beqn 
\bE(\tau_{l-1},\tau^*_k,\tau_{l})&\leq&\bE(\tau^*_k,\tau_{l},\tau_{l+1})+ \omega_1
 \\
&\leq & \bE(\tau_{l},\tau^*_{k+1}, \tau_{l+1}) \frac{ \bE(\tau^*_{k},\tau_l, \tau_{l+1})}{\bE(\tau_{l},\tau^*_{k+1}, \tau_{l+1})} + \omega_1 \\
&\leq & \left(\bE(\tau_{l-1},\tau_{l}, \tau^*_{k+1}) + \omega_2\right) \frac{ \bE(\tau^*_{k},\tau_l, \tau_{l+1})}{\bE(\tau_{l},\tau^*_{k+1}, \tau_{l+1})}  + \omega_1 \\
&\leq & \frac{\bE(\tau_{l-1},\tau^*_k,\tau_{l})}{\sqrt{3}}+  \omega_1+\sqrt{\frac{\tau_{l+1}-\tau^*_{k+1}}{\tau_{l+1}-\tau^*_k}} \omega_2\enspace ,
\eeqn 
which implies 
\[
 \bE(\tau_{l-1},\tau^*_k,\tau_{l})\leq \frac{3+ \sqrt{3}}{2}\left(\omega_1+ \sqrt{\frac{\tau_{l+1}-\tau^*_{k+1}}{\tau_{l+1}-\tau^*_k}} \omega_2 \right)\enspace .
\]
Similarly, we obtain 
\[
 \bE(\tau_{l},\tau^*_{k+1}, \tau_{l+1})\leq \frac{3+ \sqrt{3}}{2}\left(\omega_2+ \sqrt{\frac{\tau^*_k-\tau_{l-1}}{\tau_{l}-\tau^*_k} } \omega_1 \right)\enspace .
\]
Thanks to \eqref{eq:upper_3_energie}, and since $\tau_l- \tau_k^*\leq \tau_k^*-\tau_{l-1}$, we arrive at 
\beqn 
\bE(t)&\leq& 6 \left[\omega_1 \sqrt{\frac{\tau^*_k-\tau_{l-1}}{\tau_l-\tau^*_k }}+ \omega_2\sqrt{\frac{\tau_{l+1}-\tau^*_{k+1}}{\tau_{l+1}-\tau^*_k}} \right] \enspace .
\eeqn 
Write $z_1= \frac{\tau_l-\tau^*_k }{\tau^*_k-\tau_{l-1}}\geq 1$ and $z_2= \frac{\tau_{l+1}-\tau_k^*}{\tau_{l+1}-\tau^*_{k+1}}\geq 1$. 
We then come back to the definition of $\omega_1$ and $\omega_2$
 \beqn 
 \bE(t)&\leq& 12 \ell \left[\psi_q[\delta(\tau^*_k,\tau_l,\tau_{l+1})]+ \psi_q[\delta(\tau_{l-1},\tau_l,\tau^*_{k+1})]\right]+ 4\sqrt{2L}\left[\sqrt{\frac{\log(1+z_1)}{z_1}}+ \sqrt{\frac{\log(1+ z_2)}{z_2}}\right]\\
 &\leq & 12\ell \left[\psi_q[\delta(\tau^*_k,\tau_l,\tau_{l+1})]+ \psi_q[\delta(\tau_{l-1},\tau_l,\tau^*_{k+1})] \right]+ 12\sqrt{L}\enspace .
 \eeqn 
 Then, we use the definition~\eqref{eq:definition_psi_q} of $\psi_q$.
 \[
  \psi^2_q[\delta(\tau^*_k,\tau_l,\tau_{l+1})]\leq 2\log\left(\frac{n(\tau_{l+1}-\tau^*_k)}{(\tau_{l+1}-\tau_l)(\tau_l-\tau^*_k)}\right)+  q\leq 2\log(2)+ \psi^2[\delta(t)]\enspace ,
 \]
since $\tau_l-\tau_k^*\geq (\tau_l-\tau_{l-1})/2$. Besides, 
 \[
  \psi^2_q[\delta(\tau_{l-1},\tau_l,\tau^*_{k+1})]\leq 2\log\left(\frac{3n}{\tau_l-\tau_{l-1}}\right)+  q\leq 2\log(2)+ \psi^2[\delta(t)]\enspace ,
 \]
since $\tau^*_{k+1}-\tau_{l}\geq (\tau_l-\tau_{l-1})/2$. We conclude that 
\[
 \bE(t)\leq 24\ell \psi_q[\delta(t)]+ 24\ell \sqrt{2\log(3)}+ 12\sqrt{L}\enspace ,
\]
which is smaller than $\psi_q[\delta(t)][\sqrt{L} - \ell]$ provided that we have chosen $q$ large enough. This shows \eqref{eq:cond4} and concludes the main part of the proof for general $k$.

\bigskip 

It remains to consider the cases $k=1$ and $k=K$. By symmetry, we focus on the case $k=1$. Suppose that $[2; (\tau^*_1+\tau^*_2)/2]$ contains two change-points $\tau_1$ and $\tau_2$. The case $\tau_2\leq \tau^*_1$ has already been handled in Lemma~\ref{lem:step1}. The case $\tau_1\geq (1+\tau^*_1)/2$ has already been considered in the general case case. Finally, the case $\tau_1\leq (\tau^*_1+\tau^*_2)/2$ has already been handled in Lemma~\ref{lem:step3}.

\begin{proof}[Proof of Lemma \ref{lem:energie}] 
In order to prove the first bound, we simply come back to the definition of the energy. To alleviate notation, we write $\mu'_{l+1}= \overline{\theta}_{\tau_{k+1}^*:\tau_{l+1}}$ and $\mu'_{l}= \overline{\theta}_{\tau_{l-1}:\tau^*_{k}}$. 
\begin{eqnarray} \nonumber
\bE^2[\tau^*_k,\tau_{l},\tau_{l+1}] &= &\left(\mu'_{l+1}\frac{\tau_{l+1}-\tau^*_{k+1}}{\tau_{l+1}-\tau_l}+ \mu_{k+1}\frac{\tau^*_{k+1}-\tau_l}{\tau_{l+1}-\tau_l}  -\mu_{k+1}  \right)^2 \frac{(\tau_{l+1}-\tau_l)(\tau_l-\tau^*_k)}{(\tau_{l+1}-\tau^*_k)}\\\nonumber
& = &\left(\mu'_{l+1}- \mu_{k+1}\right)^2 \frac{(\tau_{l+1}-\tau^*_{k+1})^2 (\tau_l-\tau_{k}^*)}{(\tau_{l+1}-\tau^*_k)(\tau_{l+1}-\tau_l)}\\  \label{eq:upper_1_ratio}
&= & \bE^2[\tau_{l},\tau^*_{k+1},\tau_{l+1}] \frac{(\tau_{l+1}-\tau^*_{k+1})(\tau_{l}-\tau^*_k)}{(\tau_{l+1}-\tau^*_k)(\tau^*_{k+1}-\tau_{l})}\\
&\leq & \bE^2[\tau_{l},\tau^*_{k+1},\tau_{l+1}]  \frac{(\tau_{l+1}-\tau^*_{k+1})}{(\tau_{l+1}-\tau^*_k)} \enspace , \nonumber
\end{eqnarray}
where we used in the last line that  $\tau_l$ is closer to $\tau_k^*$ than to $\tau_{k+1}^*$. This proves the first part of \eqref{eq:upper_2_energie}. 
Analogously, we obtain
\begin{eqnarray}
 \bE^2[\tau_{l-1},\tau_{l},\tau^*_{k+1}]&  = & \bE^2[\tau_{l-1},\tau^*_k,\tau_{l}] \frac{(\tau^*_k-\tau_{l-1})(\tau^*_{k+1}-\tau_{l})}{(\tau_{l}-\tau^*_k)(\tau^*_{k+1}-\tau_{l-1})} \leq 
 \frac{(\tau^*_k-\tau_{l-1})}{(\tau_{l}-\tau^*_k)}\enspace .
 \label{eq:upper_2_ratio}  
\end{eqnarray}
This proves the second of \eqref{eq:upper_2_energie}. 
Multiplying \eqref{eq:upper_1_ratio} and \eqref{eq:upper_2_ratio}, we obtain the identity in \eqref{eq:upper_1_energie}. Besides, $\tau_k^*\leq \tau_{k+1}^* \leq \tau_{l+1}$ and $\tau^*_{k+1}-\tau_{l-1}\geq 3(\tau^*_{k}-\tau_{l-1})$ since $\tau^*_{k}-\tau_{l-1}\leq \tau_{l}-\tau^*_{k}\leq \tau^*_{k+1}-\tau_{l}$. This allows us to recover the upper bound in~\eqref{eq:upper_1_energie}.

Turning to the last bound \eqref{eq:upper_3_energie}, we also come back to the definition of the energy
\beqn 
\lefteqn{\bE^2(t)\frac{(\tau_{l+1}-\tau_{l-1})}{(\tau_{l+1}-\tau_{l})(\tau_{l}-\tau_{l-1})} } \\ &=   &  \left[\left(\mu'_{l+1}\frac{\tau_{l+1}-\tau^*_{k+1}}{\tau_{l+1}-\tau_{l}} + \mu_{k+1}\frac{\tau^*_{k+1}-\tau_l}{\tau_{l+1}-\tau_{l}} \right)-\left(\mu'_{l}\frac{\tau^*_k-\tau_{l-1}}{\tau_{l}-\tau_{l-1}}+ \mu_{k+1}\frac{\tau_{l}-\tau^*_k}{\tau_{l}-\tau_{l-1}} \right) \right]^2\\
&= & \left[(\mu'_{l+1}-\mu_{k+1})\frac{\tau_{l+1}-\tau^*_{k+1}}{\tau_{l+1}-\tau_{l}}  - (\mu'_{l} -\mu_{k+1})\frac{\tau^*_k-\tau_{l-1}}{\tau_{l}-\tau_{l-1}} \right]^2\\
& \leq & \left[\bE(\tau_{l},\tau_{k+1}^*,\tau_{l+1})\sqrt{\frac{\tau_{l+1}-\tau^*_{k+1}}{(\tau_{l+1}-\tau_{l})(\tau_{k+1}^*-\tau_l)}}  + \bE(\tau_{l-1},\tau_k^*,\tau_l)\sqrt{\frac{\tau^*_k-\tau_{l-1}}{(\tau_{l}-\tau_{l-1})(\tau_l-\tau^*_k) }} \right]^2\enspace . 
\eeqn
This leads us to
\beqn 
\bE(t)
& \leq & \bE(\tau_{l},\tau_{k+1}^*,\tau_{l+1})\sqrt{\frac{(\tau_{l+1}-\tau^*_{k+1})(\tau_{l}-\tau_{l-1})}{(\tau_{l+1}-\tau_{l-1})(\tau_{k+1}^*-\tau_l)}}  + \bE(\tau_{l-1},\tau_k^*,\tau_l)\sqrt{\frac{(\tau^*_k-\tau_{l-1})(\tau_{l+1}-\tau_{l})}{(\tau_l-\tau^*_k)(\tau_{l+1}-\tau_{l-1}) }} \\
&\leq & \bE(\tau_{l},\tau_{k+1}^*,\tau_{l+1})\sqrt{2\frac{\tau_{l+1}-\tau^*_{k+1}}{\tau_{l+1}-\tau_{l-1}}}  + \bE(\tau_{l-1},\tau_k^*,\tau_l)\sqrt{\frac{\tau^*_k-\tau_{l-1}}{\tau_l-\tau^*_k }} \enspace , 
\eeqn 
where we used again  $\tau_k^*-\tau_{l-1}\leq \tau_l- \tau_k^*\leq (\tau_{k+1}^*-\tau_k^*)/2$.

\end{proof}

\subsection{Localized analysis (Proof of Propositions \ref{prp:local_analysis} and \ref{prp:local_analysis_2})}

We still work conditionally to  the event $\cA_{L,q}$ so that the result of Proposition \ref{prp:rough_analysis} is valid for large $L$ and the result of Proposition \ref{prp:minimal_penalty_analysis} is true for $L>1$. Consider any high-energy change-point $\tau^*_k$. The closest estimated change-point $\widehat{\tau}_l$ to $\tau^*_k$ belongs to the interval $[\frac{\tau^*_{k-1}+\tau^*_k}{2},\frac{\tau^*_{k}+\tau^*_{k+1}}{2}]$.  By symmetry, we may assume that $\widehat{\tau}_l> \tau^*_k$, 
also if the energy  $\bE_k$  is high enough, Propositions \ref{prp:rough_analysis} and \ref{prp:minimal_penalty_analysis} enforce that $\widehat{\tau}_l-\tau_k^*\leq (\tau^*_{k+1}-\tau_k^*)/4$. 
To alleviate the notation, we write $\btau$ for $\widehat{\btau}$ in this subsection.  We deduce from Lemma \ref{lem:comparison_one_change-point} that, for $t_1= (\tau_{l-1},\tau^*_k, \tau_l)$ and $t_2=(\tau^*_k, \tau_{l},\tau_{l+1})$, 
\beqn 
 \Cr_{0}(\btau^{(-l,k)},\bY) - \Cr_{0}(\btau,\bY)&= &- ((-1)^{\sign(\underline{\Delta}_{t_1})}\bE(t_1) -  \bN(t_1) )^2 +  ((-1)^{\sign(\underline{\Delta}_{t_2})}\bE(t_2) -  \bN(t_2) )^2 \\ & &+ L\psi^2_q(\delta(t_1)) - L\psi^2_q(\delta(t_2))
\eeqn 
This difference is non-negative  since $\btau$ minimizes the criterion $\Cr_{0}$. This implies that 
\beq \label{eq:condition_energie_noise}
 [\bE(t_1)-|\bN(t_1)|]_+^2\leq  [\bE(t_2)+  |\bN(t_2)|]^2 + L(\psi^2_q(\delta(t_1)) - \psi^2_q(\delta(t_2)))\enspace . 
\eeq
First, we control the energies $\bE(t_1)$ and $\bE(t_2)$. 

\begin{lem}\label{lem:e1}
The energy $\bE(t_1)$ satisfies
\beqn
\bE^2(t_1)&\geq&  [\tau_l -\tau_{k}^*]\left[\frac{\bE_k^{2} }{16 (\tau^*_{k+1}-\tau_l)}\bigvee \frac{\Delta^{2}_k}{8}\right]\enspace . 
\eeqn
\end{lem}

\noindent 
Turning  to $\bE(t_2)$, we have $\bE(t_2)=0$ if  $\tau_{l+1}\leq \tau^*_{k+1}$. If $\tau_{l+1} > \tau^*_{k+1}$, then, since   $\tau^*_{k+1}$ is not  a $(\sqrt{L}+\sqrt{\ell},q,\btau^{(k+1)})$-high  energy change-point, we get
\beqn 
\bE^2(t_2)&=&  [\tau_l-\tau_{k}^*] (\overline{\theta}_{\tau_{k+1}^*: \tau_{l+1}}-\mu_{k+1})^2  \frac{(\tau_{l+1}-\tau^*_{k+1})^2 }{(\tau_{l+1}-\tau_l)(\tau_{l+1}-\tau_{k}^*)}\\
&\leq &  \frac{(\tau_l-\tau_{k}^*)(\tau_{l+1}-\tau^*_{k+1})}{(\tau_{k+1}^* -\tau_l)(\tau_{l+1}-\tau_{k}^*)}(\sqrt{L}+\sqrt{\ell})^2 \psi^2_q(\delta(\tau_l,\tau^*_{k+1},\tau_{l+1}))\\
&\stackrel{(a)}{\leq} & \frac{\tau_l-\tau_{k}^*}{\tau_{k+1}^* -\tau_l}(\sqrt{L}+\sqrt{\ell})^2 \left[2\log\left(\frac{n(\tau_{l+1}-\tau_{l})}{(\tau_{l+1}-\tau^*_{k})(\tau^*_{k+1}-\tau_l)}\right)+q  \right]\\
&\stackrel{(b)}{\leq} & \frac{\tau_l-\tau_{k}^*}{\tau_{k+1}^* -\tau_l}(\sqrt{L}+\sqrt{\ell})^2 \left[2\log\left(\frac{n}{\tau_{l+1}-\tau_{k}^*}\right)+q  \right]\\
&\leq & \frac{\tau_l-\tau_{k}^*}{\tau_{k+1}^* -\tau_l}(\sqrt{L}+\sqrt{\ell})^2 \left[2\log\left(\frac{n}{\tau^*_{k+1}-\tau_{k}^*}\right)+q  \right]\\
&\stackrel{(c)}{\leq} & \frac{\tau_l-\tau_{k}^*}{\tau_{k+1}^* -\tau_l}(\sqrt{L}+\sqrt{\ell})^2 \psi^2_q\left[\delta(\tau_{k-1}^*,\tau_k^*, \tau_{k+1}^*)\right] \enspace ,
\eeqn
where we used in (a) that  $x(\log(t/x)+1)\leq \log(t)+1$ for $t\geq 1 $ and $x\leq 1$, in (b) that $\tau_{k+1}^*\leq \tau_l$  and in (c) that $\tau_l$ is closer to $\tau^*_k$ than to $\tau^*_{k+1}$. Together with Lemma \ref{lem:e1}, we obtain that 
$\bE^2(t_1)\geq 2\bE^2(t_2)$ provided $\tau_k^*$ is a ($\sqrt{32}(\sqrt{L}+\sqrt{\ell}),q$) high-energy change-point. Since $(x-y)_+^2 \geq 3x^2/4 -3y^2$ and $(x+y)^2\leq 5/4x^2 + 5 y^2$, we deduce from~\eqref{eq:condition_energie_noise} that 
\[
 \frac{\bE^2(t_1)}{8}\leq 3 \bN^2(t_1)+ 5\bN^2(t_2) + L(\psi^2_q(\delta(t_1)) - \psi^2_q(\delta(t_2)))\enspace . 
\]
Then, we use Lemma~\ref{lem:e1} to deduce that 
\beq\label{eq:condition_energie_noise2}
 [\tau_l-\tau_k^*]\Delta_k^2 \leq c\left[\bN^2(t_1)+ 5\bN^2(t_2) + L[\psi^2_q(\delta(t_1)) - \psi^2_q(\delta(t_2))]\right] \enspace .  
\eeq

\medskip 

Next, we control the random variables $|\bN(t_1)|$ and $|\bN(t_2)|$ relying on a non-asymptotic law of iterated logarithms. 

\noindent 
{\bf Case 1}: $\tau_{l+1}-\tau_{l}\geq \tau_{l}-\tau_{k}^*$. In that situation, the penalty difference satisfies
\beq\label{eq:condition_difference_penalite}
 \psi^2_q(\delta(t_1)) - \psi^2_q(\delta(t_2))= 2\log\left[\frac{(\tau_l-\tau_{l-1})(\tau_{l+1}-\tau_l)}{(\tau_k^*-\tau_{l-1})(\tau_{l+1}-\tau_{k}^*)}\right]\leq 2\log\left(\frac{\tau_l-\tau_{l-1}}{\tau_k^*-\tau_{l-1}}\right)\leq 2\log\left(1+\frac{\tau_l-\tau_k^*}{\tau_k^*-\tau_{l-1}}\right)\enspace ,
\eeq
since $\tau_l$ is closer to $\tau^*_k$ than $\tau_{l-1}$.

To control $|\bN(t_1)|$ and $|\bN(t_2)|$, we apply Lemma \ref{lem_lil_localiser} to the variables $(Z_{\tau_k^*: \tau'},$  $\tau'> \tau_{k}^*$) and  to the variables 
 $(Z_{\tau':\tau_k^* },$  $\tau'< \tau_{k}^*$). Since $|\tau_l-\tau_{k}^*|\leq |\tau_k^* - \tau_{l-1}|$, we have, with probability higher than $1-c e^{-s}$, 
\begin{eqnarray}
 |\bN(t_1)|1_{\Delta^{2}_k|\tau_l-\tau_k^*|\geq 1}&\leq&  c\left[\sqrt{\frac{\tau_k^* - \tau_{l-1}}{\tau_l -\tau_{l-1}}}+ \sqrt{\frac{\tau_l - \tau^*_{k}}{\tau_l -\tau_{l-1}}} \right]\left[\sqrt{s} + \sqrt{\log\log(\Delta^{2}_k |\tau_l-\tau_k^*|)}\right] \nonumber\\ &\leq&  c \left[\sqrt{s} + \sqrt{\log\log(\Delta^{2}_k |\tau_l-\tau_k^*|)}\right] \enspace .\label{eq:upper_nt_1}
\end{eqnarray}

Note that $Z_{\tau_l: \tau_{l+1}}= -Z_{\tau_k^*:\tau_l}+ Z_{\tau_{k}^*: \tau_{l+1}}$. Applying again Lemma~\ref{lem_lil_localiser}, we derive that, with probability higher than $1-ce^{-x}$, 
\beqn 
 |Z_{\tau_l: \tau_{l+1}}|1_{\Delta^{2}_k|\tau_l-\tau_{k}^*|\geq 1}&\leq& c \sqrt{\tau_{l+1}-\tau_k^*}\left[ \sqrt{\log\log( \Delta^{2}_k ( \tau_{l+1}-\tau_k^*))}+ \sqrt{x} \right]\\ 
 \\
 &\leq& 
 c' \sqrt{\tau_{l+1}-\tau_l}\left[ \sqrt{\log\log( \Delta^{2}_k ( \tau_{l+1}-\tau_l))}+ \sqrt{x} \right]
 \enspace , 
\eeqn 
since $\tau_{l+1}-\tau_{l}\geq \tau_{l}-\tau_{k}^*$. We have bounded $Z_{\tau^*_{k}:\tau_l}$ above. Combining these bounds, we get 
\beqn 
|\bN(t_2)|1_{\Delta^{2}_k|\tau_l-\tau_k^*|\geq 1}&\leq& c \sqrt{\frac{\tau_{l+1} - \tau_{l}}{\tau_{l+1} -\tau^*_{k}}}\left[\sqrt{x} + \sqrt{\log\log(\Delta^{2}_k|\tau_l-\tau_k^*|)}\right]\\ & &+ c' \sqrt{\frac{\tau_{l} - \tau_{k^*}}{\tau_{l+1} -\tau^*_{k}}}\left[\sqrt{x} + \sqrt{\log\log(\Delta^{2}_k(\tau_{l+1}-\tau_{l})|)}\right]
\\ &\leq &c \left[\sqrt{x} + \sqrt{\log\log(\Delta^{2}_k|\tau_l-\tau_k^*|)}\right] \enspace .
\eeqn 

Gathering \eqref{eq:condition_energie_noise} and \eqref{eq:condition_difference_penalite}, we deduce from~\eqref{eq:condition_energie_noise2} that
\beq\label{eq:conclusion_1_0}
 |\tau_l -\tau_k^*|\Delta^{2}_k 1_{\Delta^{2}_k|\tau_l-\tau_k^*|\geq 1}\leq c \left[x + \log\log(\Delta^{2}_k|\tau_l-\tau_k^*|) +L\log\left(1+ \frac{\tau_l-\tau_k^*}{\tau_k^*-\tau_{l-1}}\right) \right] \enspace .
\eeq
The expression inside the last logarithm is smaller or equal to $2$, which implies that 
$|\tau_l -\tau_k^*|\Delta^{2}_k \leq c' (x\vee L)$
with probability higher than $1-c''e^{-x}$. If we restrict ourselves to $L\leq L_0$ (where $L_0$ is the absolute constant introduced in Proposition~\ref{prp:rough_analysis}), the above inequality yields 
\[
|\tau_l -\tau_k^*|\Delta^{2}_k \leq c (x\vee 1)\enspace .
\]
When $L$ is larger than $L_0$ (as in Proposition~\ref{prp:rough_analysis}), we know from that proposition that ${\tau_{k}^*-\tau_{l-1}}\geq (\tau_{k}^*-\tau^*_{k-1})/2$.  Since $(\tau_{k+1}^*-\tau_k^*) \Delta^{2}_k\geq \bE_k^{2}\geq \kappa_L \psi^2_{q}[\delta(\tau_k^*,\tau_k^*,\tau^*_{k+1})]\geq 2\kappa_L$ since $q\geq 2$. Then, we deduce from \eqref{eq:conclusion_1_0} that
\[
|\tau_l -\tau_k^*|\Delta^{2}_k 1_{\Delta^{2}_k|\tau_l-\tau_k^*|\geq 1}\leq c \left[x  +L\log\left(1+ \frac{\Delta^2(\tau_l-\tau_k^*)}{2\kappa_L}\right) \right]\leq c\left[x  +\frac{L}{2\kappa_L}\Delta^2(\tau_l-\tau_k^*)\right]
\]
Provided that we fix $\kappa_L$ in such a way that $\kappa_L\geq cL$ (where $c$ is the same constant as in the above inequality), we conclude that 
\[
 |\tau_l -\tau_k^*|\Delta^{2}_k 1_{\Delta^{2}_k|\tau_l-\tau_k^*|\geq 1}\leq c'\left[x\vee 1\right]\enspace \ ,
\]
with probability higher than $1-c''e^{-x}$.

\bigskip

\noindent

\noindent 
{\bf Case 2}: $\tau_{l+1}-\tau_{l}\leq \tau_{l}-\tau_{k}^*$. This situation does not arise for large $L\geq L_0$ setting as justified in Proposition~\ref{prp:rough_analysis}. We still use the same deviation bound~\eqref{eq:upper_nt_1} for $|\bN(t_1)|$, but we need to rely on a different approach for $\bN(t_2)$.

Fix any $q\geq 1$ and $\tau'_q= \tau_k^* 2^q \lceil 1/\Delta_k^{2}\rceil$.  Applying Lemma \ref{lem:concentration:N_t} with $n$ replaced by $\tau'_q-\tau^*_k$, we derive that, with probability higher than $1-x$, we have simultaneously over all $\tau\in [ \tau'_q/2;\tau'_q]$
\[
 \bN^2(\tau_k^*,\tau,\tau'_q)\leq   2\log\left(\frac{(\tau-\tau^*_k)(\tau'_q-\tau^*_k)}{(\tau'_q-\tau)(\tau-\tau^*_k) }\right)+   c_1 \log\log\left(\frac{(\tau-\tau^*_k)(\tau'_q-\tau^*_k)}{(\tau'_q-\tau)(\tau-\tau^*_k) x}\right) + c_2\log\left(\frac{1}{x}\right) + c_3\enspace .  \ \]
Applying an union bound over all $q$, we deduce that, with probability higher than $1-x$, 
\begin{eqnarray}\nonumber
 \bN^2(t_2)\1_{(\tau_l-\tau^*_k)\Delta_k^{2*}\geq 1}&\leq &2\log\left(\frac{\tau_{l+1}-\tau^*_{k}}{\tau_{l+1}-\tau_l}\right)+ c_1 \log \log\left(\frac{\tau_{l+1}-\tau^*_{k}}{\tau_{l+1}-\tau_l}\right)\\ & & +c_2\log\log\left((\tau_l-\tau^*_k)\Delta_k^{2} \right)+ c_3 \log(\frac{1}{x}) + c_4 \label{eq:control_n_t2}
\end{eqnarray}

As for the penalty difference, we have 
\[
 \psi^2_q(\delta(t_1)) - \psi^2_q(\delta(t_2))= 2\log\left[\frac{(\tau_l-\tau_{l-1})(\tau_{l+1}-\tau_l)}{(\tau_k^*-\tau_{l-1})(\tau_{l+1}-\tau_{k}^*)}\right]\leq 2\log(2)-2\log\left( \frac{\tau_{l+1}-\tau_k^*}{\tau_{l+1}-\tau_{l}}\right)\enspace .
\]

Since $\bE(t_2)=0$ in that case, it follows from \eqref{eq:condition_energie_noise} and Lemma \ref{lem:e1}, we conclude that 
\beqn
 |\tau_l -\tau_k^*|\Delta^{2}_k 1_{\Delta^{2}_k|\tau_l-\tau_k^*|\geq 1}&\leq& c \left[1+ x+ L + \log\log(\Delta^{2}_k|\tau_l-\tau_k^*|)\right]\\
 &&+ c_1 \log \log\left(\frac{\tau_{l+1}-\tau^*_{k}}{\tau_{l+1}-\tau_l}\right) - 2(L-1) \log\left( \frac{\tau_{l+1}-\tau_k^*}{\tau_{l+1}-\tau_{l}}\right)\ . 
 \eeqn 
 Since $a\log(x) - bx \leq a\log(a/(eb))$ for any positive numbers $a$, $b$, $x$, the expression in the second line is a most $c[1+ \log((L-1)^{-1})_+$. Hence, we conclude that 
 \[
  |\tau_l -\tau_k^*|\Delta^{2}_k \leq c (x\vee 1+ \log((L-1)^{-1})_+ )\enspace , 
\]
with probability higher than $1-ce^{-x}$. 

\begin{proof}[Proof of Lemma \ref{lem:e1}]
If $\tau_{l-1} \geq \tau^*_{k-1}$, then the energy expression is simply 
\[
 \bE^2(t_1)= [\tau_l -\tau_{k}^*]\Delta_k^2 \frac{(\tau_{k}^* - \tau_{l-1})  }{(\tau_l -\tau_{l-1})}\geq  [\tau_l -\tau_{k}^*]\frac{\Delta_k^2}{2} \enspace , 
\]
since $\tau^*_{k}$ is closer to $\tau_l$ than $\tau_{l+1}$.
We claim that, even in the more involved situation where $\tau_{l-1}< \tau^*_{k-1}$, one has 
\beq\label{eq:first_expression_et1}
 \bE^2(t_1)\geq  [\tau_l -\tau_{k}^*] \frac{\Delta_k^2}{8}\enspace . 
\eeq
 Coming back to the definition of $\bE_{k}$, we have 
\beqn
\frac{\bE^2(t_1)}{[\tau_l -\tau_{k}^*]\bE_k^{2}}&\geq& \frac{(\tau^*_{k+1}-\tau^*_{k-1})  }{8(\tau^*_{k+1}-\tau^*_k)(\tau^*_k -\tau^*_{k-1})}  
\geq  \frac{1}{8(\tau^*_{k+1} -\tau^*_{k})}
\geq \frac{1}{16(\tau_{k+1}^*-\tau_l)}\enspace , 
\eeqn
where we used in the last inequality that $2(\tau_l-\tau_{k}^*)\leq (\tau_{k+1}^*-\tau_{k}^*)$.
\bigskip 

\noindent 
To finish this proof, it remains to prove \eqref{eq:first_expression_et1} when $\tau_{l-1}< \tau^*_{k-1}$. In that case, we have
\beqn 
 \frac{(\tau_l -\tau_{l-1}) \bE^2(t_1)}{[\tau_l -\tau_{k}^*]( \tau_{k}^* - \tau_{l-1}) }&=&  \left(\mu_{k+1} -\mu_k \frac{\tau^*_k-\tau^*_{k-1}}{\tau^*_k-\tau_{l-1}} -\overline{\theta}_{\tau_{l-1}: \tau^*_{k-1}}\frac{\tau^*_{k-1}- \tau_{l-1}}{\tau^*_k-\tau_{l-1}}\right)^2 \\
 &=  &  \left( \mu_{k+1} -\mu_k + (\mu_k -\overline{\theta}_{\tau_{l-1}: \tau^*_{k-1}})\frac{\tau^*_{k-1}-\tau_{l-1}}{\tau^*_k-\tau_{l-1}}\right)^2\ . 
\eeqn 
Since $\tau_l -\tau_{l-1}\leq 2(\tau_l-\tau^*_k)$ ($\tau_l$ is closer to $\tau^*_k$ than any other point), it suffices to prove that
\beq\label{eq:objective}
|\mu_k -\overline{\theta}_{\tau_{l-1}: \tau^*_{k-1}}|\frac{\tau^*_{k-1}-\tau_{l-1}}{\tau^*_k-\tau_{l-1}}\leq  \frac{1}{2}|\mu_{k+1} -\mu_k |\enspace . 
\eeq
Since $\tau^*_{k-1}$ is not a $(\sqrt{L}+\sqrt{\ell},q, \btau^{(k-1)})$-high  energy change-point, we have 
\[
 \big|\overline{\theta}_{\tau_{l-1}: \tau^*_{k-1}} -\mu_k \frac{\tau^*_{k} - \tau^*_{k-1}}{\tau_{l} - \tau^*_{k-1}} - \mu_{k+1}\frac{\tau_l -\tau^*_k}{\tau_{l} - \tau^*_{k-1}}  \big|\leq \psi_q[\delta(\tau_{l-1},\tau^*_{k-1},\tau_l)][\sqrt{L} + \sqrt{\ell}]\frac{\tau_l-\tau_{l-1}}{(\tau_l-\tau^*_{k-1})(\tau^*_{k-1}-\tau_{l-1})}\enspace ,
\]
which in turn implies 
\beqn 
 \big|\overline{\theta}_{\tau_{l-1}: \tau^*_{k-1}} -\mu_k\big|\frac{\tau^*_{k-1}-\tau_{l-1}}{\tau^*_k-\tau_{l-1}}&\leq& |\mu_{k+1}-\mu_{k}|\frac{(\tau_l -\tau^*_k)(\tau^*_{k-1}-\tau_{l-1})}{(\tau_{l} - \tau^*_{k-1})(\tau^*_k-\tau_{l-1})}\\ &&+ \psi_q[\delta(\tau_{l-1},\tau^*_{k-1},\tau_l)][\sqrt{L} + \sqrt{\ell}]\sqrt{\frac{(\tau_l-\tau_{l-1})(\tau^*_{k-1}-\tau_{l-1})}{(\tau_l-\tau^*_{k-1})(\tau^*_{k}-\tau_{l-1})^2}}\ . 
 \eeqn 
 The first expression in the rhs is smaller than $|\mu_{k+1}-\mu_k|/3$ because  $(\tau_l -\tau^*_k)\leq (\tau_k^* -\tau^*_{k-1})/2$ and $\tau^*_{k-1}\leq \tau^*_k$. Since $x\log(t)\leq \log(tx)$ for $t\geq1$ and $x\leq 1$, it follows that 
 \beqn 
 \psi_q[\delta(\tau_{l-1},\tau^*_{k-1},\tau_l)]\sqrt{\frac{\tau^*_{k-1}-\tau_{l-1}}{\tau_l-\tau_{l-1}} }&\leq &\sqrt{2\log\left(\frac{n}{\tau_l-\tau^*_{k-1}}\right) + q}\leq \sqrt{2\log\left(\frac{n}{\tau^*_k-\tau^*_{k-1}}\right) + q}\\&\leq& \psi_q[\delta(\tau_{k-1}^*,\tau_k^*,\tau_{k+1}^*)]\\ &\leq&  \frac{\bE_k}{\kappa_L}\leq \frac{|\mu_{k+1}-\mu_k|}{\kappa_L}(\tau^*_{k}-\tau^*_{k-1})^{1/2}\enspace , 
\eeqn 
since we assume that $\bE_k$ is a $\kappa_L$-high energy change-point. Coming back to $\overline{\theta}_{\tau_{l-1}: \tau^*_{k-1}} -\mu_k|$, this yields
 \beqn 
 \big|\overline{\theta}_{\tau_{l-1}: \tau^*_{k-1}} -\mu_k\big|\frac{\tau^*_{k-1}-\tau_{l-1}}{\tau^*_k-\tau_{l-1}}&\leq& \frac{|\mu_{k+1}-\mu_{k}|}{3}+\frac{\sqrt{L} + \sqrt{\ell}}{\kappa_L}|\mu_{k+1}-\mu_{k}|  \frac{(\tau_l-\tau_{l-1})\sqrt{\tau^*_k- \tau^*_{k-1}}}{(\tau^*_{k}-\tau_{l-1})\sqrt{\tau_l-\tau^*_{k-1}}}  \\
 &\leq & |\mu_{k+1}-\mu_{k}|\left[\frac{1}{3}+ 2\frac{\sqrt{L} + \sqrt{\ell}}{\kappa_L}
\right]\\
&\leq &  \frac{|\mu_{k+1}-\mu_{k}|}{2}\enspace , 
\eeqn 
where we used that $\tau_l$ is closer to $\tau^*_k$ than $\tau_{l-1}$ and that $\kappa_L$ is large enough.

\end{proof}

\subsection{Proofs for the post-processing steps}

 \begin{proof}[Proof of Proposition \ref{prp:pruned}]
 We recall the definition of that the event $\cB_{1-\alpha}$ of probability higher than $1-\alpha$ is such that 
\beq\label{eq:definition_event_B}
 |\bN(t)|\leq \sqrt{2 \log\left(\frac{n(t_3-t_1)}{(t_2-t_1)(t_3-t_2)}\right) }+ \zeta_{1-\alpha}\enspace , \quad \quad \forall t\in \cT_3\enspace .
\eeq
Under this event, for any $\tau$ such that $\widehat{r}_{\tau}<\infty$, the interval $[t^{(\tau, \widehat{r}_{\tau})}_1+1; t^{(\tau, \widehat{r}_{\tau})}_3-1]$ contains a least one true change-point, say $\tau^*_k$. By definition of the pruning step, the 
confidence intervals associated to the pruned change-points $\cP(\btau)$ do not intersect. Suppose that two change-points, say $\cP(\btau)_l$ and $\cP(\btau)_{l+1}$, in the segment  $\big(\frac{\tau^*_{k-1}+\tau^*_k}{2},\frac{\tau^*_{k}+\tau^*_{k+1}}{2}\big]$. Since both confidence intervals of $\cP(\btau)_l$ and $\cP(\btau)_{l+1}$ contain at least one change-point, this implies that $\tau^*_k$ belongs to these two intervals which contradicts the non intersection property.For $k=1$ and $k=K$, the same argument applies with $[2; \frac{\tau^*_{1}+\tau^*_2}{2}]$ and $[\frac{\tau^*_{K-1}+\tau^*_K}{2};n]$ respectively. 

\medskip

Let us turn to the second result. Consider a $(\kappa,\zeta^2_{1-\alpha})$-high-energy change-point $\tau^*_k$ and assume that there exists $\tau_l$ such that 
\[
|\tau_l-\tau^*_k|
<  \frac{(\tau_{k}^* - \tau_{k-1}^*)\wedge (\tau_{k+1}^* - \tau_k^*)}{8}\enspace . 
\]   
Consider any $r\geq |\tau_l-\tau_k^*|\wedge \tau_l \wedge (n+1-\tau_l)$. 
 Writing down the energy of $t^{(\tau_l,r) }$ and relying on \eqref{eq:definition_event_B}, we derive that 
\beqn 
\bE(t^{(\tau_l,r) })& =& |\Delta_k|\left(1- \frac{|\tau_l-\tau_k^*|}{r}\right)_+\sqrt{\frac{r}{2}}\enspace ,\quad \quad  \quad 
|\bN(t^{(\tau_l,r) })|\leq \sqrt{2\log\left(2\frac{n}{r} \right)}+ \zeta_{1-\alpha}\enspace . 
\eeqn 
Let $\overline{r}_l$ be the smallest $r>0$ such that  $\bE(t^{(\tau_l,r) })\geq 2\sqrt{2\log\left(2\frac{n}{r} \right)}+ 2\zeta_{1-\alpha} $. Since $|\bC(t^{(\tau_l,r) })|\geq \bE(t^{(\tau_l,r) })-|\bN(t^{(\tau_l,r) })|$, we have  $\overline{r}_l\geq \widehat{r}_{\tau_l}$. This implies that 
\beq\label{eq:definition_r_tau_good}
\widehat{r}_{\tau_l} \leq 2 |\tau_l -\tau^*_k|\vee \left\lceil \frac{64\left[\sqrt{\log\left(n\Delta_k^{2}\right)}+ \zeta_{1-\alpha}/\sqrt{2}\right]^2}{\Delta_k^{2}}\right\rceil < 1+ \frac{(\tau^*_k-\tau^*_{k-1})\wedge (\tau^*_{k+1}-\tau^*_{k})}{4}\enspace , 
\eeq
since $\tau^*_k$ is a $(\kappa,\zeta^2_{1-\alpha})$-high-energy change-point and since the constant $\kappa$ is large enough. In turn, this also implies that 
\beq\label{eq:subset_I_l}
\underline{I}_{\tau_l}\subset \left [\tau^*_k - \frac{(\tau^*_k-\tau^*_{k-1})\wedge (\tau^*_{k+1}-\tau^*_{k})}{4},\tau^*_k + \frac{(\tau^*_k-\tau^*_{k-1})\wedge (\tau^*_{k+1}-\tau^*_{k})}{4}\right]\enspace . 
\eeq

We prove below  that, in $\cP(\btau)$, there exists a least one change-point, say $\cP(\btau)_j$ satisfying $|\cP(\btau)_j-\tau^*_k|\leq \widehat{r}_{\tau_l}$. If $\tau_l$ belongs to $\cP(\btau)$, this is obviously true. Now assume that $\tau_l$ is pruned. Consequently, there exists $\tau_{m}$ such that $\widehat{r}_{\tau_m}\leq \widehat{r}_{\tau_l}$ and $\underline{I}_{\tau_m}\cap \underline{I}_{\tau_l}\neq \emptyset$. Since $\widehat{r}_{\tau_m}\leq \widehat{r}_{\tau_l}$  an by \eqref{eq:subset_I_l}, this implies that 
\[
|\tau_m-\tau^*_k|\leq     \frac{(\tau^*_k-\tau^*_{k-1})\wedge (\tau^*_{k+1}-\tau^*_{k})}{2}\enspace .
\]
Recall that $\underline{I}_{\tau_m}$ contains at least one true change-point. Hence,  $\tau^*_k\in \underline{I}_{\tau_m}$. Therefore, 
\[
|\tau_{m}-\tau^*_k|\leq \widehat{r}_{\tau_m}-1\leq \widehat{r}_{\tau_l}-1
\]
If $\tau_{m}$ itself is also pruned, then we show similarly that there exists a change-point $\tau_{m'}$ satisfying $|\tau_{m'}-\tau^*_k|\leq \widehat{r}_{\tau_m} -1 \leq \widehat{r}_{\tau_l}-1$. By recursion, this implies that, there exists one change-point $\cP(\btau)_j$ in $\cP(\btau)$ such that  $|\cP(\btau)_j-\tau^*_k|\leq \widehat{r}_{\tau_l}-1 $. The result follows.

\end{proof}

\begin{proof}[Proof of Proposition \ref{prp:post:localization}]
 On the event $\cB_{1-\alpha}$, we have 
 \beq\label{eq:encadrement_erreur}
 |\tau-\tau_k^*|\leq  \widehat{r}_{\tau} -1  < \frac{(\tau_{k}^* - \tau_{k-1}^*)\wedge (\tau_{k+1}^* - \tau_k^*)}{4}\enspace .
 \eeq
 The first inequality holds since at least one true change-point belongs to $\underline{I}_{\tau}$. The second inequality is proved as~\eqref{eq:definition_r_tau_good} in the proof of  Proposition~\ref{prp:pruned}. As a consequence, 
 \[
 |\tau-\tau^k_*|+ 2\widehat{r}_{\tau}-1 < 1+ \frac{3}{4} (\tau_{k}^* - \tau_{k-1}^*)\wedge (\tau_{k+1}^* - \tau_k^*)\enspace .
 \]
Since the lhs is an integer, this implies that 
\[
 |\tau-\tau^k_*|+ 2\widehat{r}_{\tau}-1 \leq \left\lceil \frac{3}{4} (\tau_{k}^* - \tau_{k-1}^*)\wedge (\tau_{k+1}^* - \tau_k^*)\right\rceil\leq (\tau_{k}^* - \tau_{k-1}^*)\wedge (\tau_{k+1}^* - \tau_k^*)\ . 
\]
Hence, the interval $[\tau-2\widehat{r}_{\tau}+1, \tau+2\widehat{r}_{\tau}-1)$ is included in $[\tau^*_{k-1}; \tau^*_{k+1})$ and only contain the true change-point $\tau^*_k$.

 Then, we argue as in the proof of Proposition~\ref{prp:local_analysis_2}. Fix any $x>0$. From Lemma \ref{lem_lil_localiser} (with $\nu=\Delta_k^{-2}$), we deduce that, with probability higher than $1-2e^{-x}$, we have 
 \begin{eqnarray}\label{eq:upper_Z_tau:post:localization}
| Z_{\tau_k^*-s:\tau_k^*}|\leq 2 \sqrt{2s}\sqrt{2\log(\log(3\Delta_k^{-2}s))+ x+2}\enspace ;\\
| Z_{\tau_k^*:\tau_k^*+s}|\leq 2 \sqrt{2s}\sqrt{2\log(\log(3\Delta_k^{-2}s))+ x+2}\enspace , \nonumber 
 \end{eqnarray}
 simultaneously over all integers $s\geq  \Delta_k^{-2}$.

 Consider any $\tau'\in \underline{I}_{\tau}$ such that $|\tau'-\tau|\geq \Delta_k^{-2}$. By symmetry we can assume that $\tau'\leq \tau^*_k$.  Define $t_1= (\tau', \tau^*_k, \tau+2\widehat{r}_{\tau}-1)$ and $t_2=(\tau-2\widehat{r}_{\tau}+1, \tau',\tau_k^*)$.
 \begin{eqnarray} \nonumber
 \|\bPi_{\tau'}\bY^{(\tau;2\widehat{r}_{\tau}-1)} \|^2- \|\bPi_{\tau^*_k}\bY^{(\tau;2\widehat{r}_{\tau}-1)} \|^2&=& -\bC^2[t_1] + \bC^2[t_2]\\ \nonumber
 &= &  -((-1)^{\sign(\Delta_k)}\bE[t_1]- \bN(t_1))^2 + \bN^2(t_2)\\&\geq& \frac{1}{2}\bE^2(t_1)-\bN^2(t_1)- \bN^2(t_2)\enspace . \label{eq:decomposition_locale1}
 \end{eqnarray}
 First, we work out the energy. 
\beqn 
 \bE^2(t_1)= |\tau_k^*-\tau'|\frac{|\tau+2\widehat{r}_{\tau}-1- \tau_k^*|}{|\tau+2\widehat{r}_{\tau}-1- \tau'|}\Delta_k^{2}\geq |\tau_k^*-\tau'|\frac{\Delta_k^{2}}{3}\enspace .
 \eeqn 
 Since $\tau^*_k\leq \tau+\widehat{r}_{\tau}-1$ and $\tau'\geq \tau-\widehat{r}_{\tau}+1$. Then, we rely on \eqref{eq:upper_Z_tau:post:localization} to control the stochastic terms $\bN(t_1)$  and $\bN(t_2)$. 
 \beqn 
 \frac{|\bN(t_1)|}{2\sqrt{2}}&\leq& \sqrt{\frac{\tau+2\widehat{r}_{\tau}-1- \tau_k^*}{\tau+2\widehat{r}_{\tau}-1- \tau'}} \sqrt{\log(\log(e\Delta_k^{-2}|\tau^*_k-\tau'|))+ x+2}  \\&+ & \sqrt{\frac{|\tau^*_k-\tau'|}{\tau+2\widehat{r}_{\tau}-1- \tau'}} \sqrt{\log(\log(e\Delta_k^{-2}|\tau+2\widehat{r}_{\tau}-1-\tau^*_k|))+ x+2}\\
 &\leq & \sqrt{\log(\log(e\Delta_k^{-2}|\tau^*_k-\tau'|))+ x+2}\\ &&   +  \sqrt{\frac{|\tau^*_k-\tau'|}{\tau+2\widehat{r}_{\tau}-1- \tau'}} \sqrt{\log(\log(e\Delta_k^{-2}|\tau+2\widehat{r}_{\tau}-1- \tau'|))+ x+2}\\
 &\leq &2\sqrt{\log(\log(e\Delta_k^{-2}|\tau^*_k-\tau'|))+ x+2}\enspace ,
 \eeqn  
  where we used in the last line that the function $x\mapsto x[\log\log(a/x)+ b]$ is increasing on $(0,1]$ for any $a\geq e$. As for $\bN(t_2)$, we decompose $Z_{\tau-2\widehat{r}_{\tau}+1:\tau'}$ into $Z_{\tau-2\widehat{r}_{\tau}+1:\tau_k^*}-Z_{\tau':\tau_k^*}$ to rely on \eqref{eq:upper_Z_tau:post:localization}
\beqn 
   |\bN(t_2)|&\leq& \frac{|Z_{\tau':\tau_k^*}|}{\sqrt{|\tau_k^*-\tau'|}}+ \frac{(|Z_{\tau-2\widehat{r}_{\tau}+1:\tau_k^*}|-|Z_{\tau':\tau_k^*}|) }{\sqrt{\tau'- \tau+2\widehat{r}_{\tau}-1}}\sqrt{\frac{|\tau_k^*-\tau|}{\tau^*_k- \tau+2\widehat{r}_{\tau}-1}} \\
   &\leq & 2\sqrt{2}\sqrt{\log(\log(e\Delta_k^{-2}|\tau^*_k-\tau'|))+ x+2}\\ & &
   + 4\sqrt{2}\sqrt{\log(\log(e\Delta_k^{-2}|\tau^*_k-\tau+2\widehat{r}_{\tau}-1|))+ x+2}\sqrt{\frac{|\tau^*_k-\tau'|}{\tau_k^*-\tau+2\widehat{r}_{\tau}-1}} \\
   &\leq & 6\sqrt{2}\sqrt{\log(\log(e\Delta_k^{-2}|\tau^*_k-\tau'|))+ x+2}\enspace .
 \eeqn 
 Coming back to \eqref{eq:decomposition_locale1} we conclude that $\cL(\tau)\neq \tau'$ if $|\tau'-\tau^*_k|\geq \Delta_k^{-2}$ and 
\[
 |\tau^*_k-\tau'|\geq c\frac{\log\log(e\Delta_k^{-2}|\tau^*_k-\tau'|)+ x}{\Delta_k^{2}}\enspace .
\]
This concludes the proof.

\end{proof}

\appendix

\section{Proof of the deviation inequalities}

\subsection{Proof of the law of iterated logarithms lemmas}~\label{sec:proof_lil}

\begin{proof}[Proof of Lemma~\ref{lem:supsumGauss_log_iterated_0}]

Define, for any integer $n\geqslant 1$, the random walk $S_n=\sum_{i=1}^n \epsilon_i$.
Let $k< n$ denote two positive integers. For any $s >0$, 
\[
\E[e^{s (S_n-S_k)}]=\prod_{i=k+1}^n\E[e^{s\epsilon_i}]\leq \prod_{i=k+1}^n\E[e^{s^2/2}]=e^{(n-k)s^2/2}\enspace.
\]
Let $\cF_k$ denote the sigma-algebra induced by $(\epsilon_1,\ldots, \epsilon_k)$. 
As $S_k$ is independent of $S_n-S_k$, this entails that  
\[
\E[e^{s S_n-\frac{s^2n}{2}}|\cF_k]=e^{s S_k-\frac{s^2k}{2}}\E[e^{s (S_n-S_k)-\frac{s^2(n-k)}{2}}]\leqslant e^{s S_k-\frac{s^2k}{2}}\enspace.
\]
Hence, $e^{s S_n-\frac{s^2n}{2}}$ is a super-martingale.
Fix $x>0$ and let $A$ be the stopping time defined by
$$A=\inf \{n\geq d, S_n\geq\sqrt{n}x\}\enspace .$$
By definition, $S_A\geq \sqrt{A} x \geq \sqrt{d} x$. 
Hence, 
\[
\E\big[e^{s \sqrt{d} x -\frac{s^2A}{2}}\big]\leq \E\big[e^{s S_A -\frac{s^2A}{2}}\big] \leq \E\big[e^{s S_1 -\frac{s^2}{2}}\big]\leq 1\enspace.
\]
Now, by definition,
\[
\P\left[\max_{k\in [d,(1+\alpha)d]} \frac{\sum_{i=1}^k Y_i}{\sqrt{k}} \geq x\right] = \P[A \leq (1+\alpha)d]\enspace.
\]
As the function $u\mapsto e^{\sqrt{d} x -\frac{s^2u}{2}}$ is non-increasing, it follows that
\[
\P\left[\max_{k\in [d,(1+\alpha)d]} \frac{\sum_{i=1}^k Y_i}{\sqrt{k}} \geq x\right] = \P\left[e^{s \sqrt{d} x -\frac{s^2A}{2}} \geq e^{s \sqrt{d} x -s^2(1+\alpha)d/2}\right]\enspace .
 \]
By Markov inequality, we obtain that
\[
\P\left[\max_{k\in [d,(1+\alpha)d]} \frac{\sum_{i=1}^k Y_i}{\sqrt{k}} \geq x\right] \leq e^{-s \sqrt{d} x +s^2(1+\alpha)d/2}\enspace.
\]
Choosing $s=x/[(1+\alpha)\sqrt{d}]$ concludes the proof of  Lemma~\ref{lem:supsumGauss_log_iterated_0}. 
\end{proof}

\begin{proof}[Proof of Lemma \ref{lem_lil_localiser}]
Let $\nu>0$, and $t_1$, $t_2$, two integers such that $t_1-1/\nu<t_2$. 
Without loss of generality, assume that $1/\nu\geq 1$ and that $t_1+1/\nu\leq t_2$.
Up to a renumbering of the indices, the sum $Z_{t_1:t_2}$ can be written 
\[
Z_{t_1:t_2}=\sum_{i=1}^{t_2-t_1}\epsilon_i\enspace.
\]
Consider any non-negative integer $s$. By Lemma~\ref{lem:supsumGauss_log_iterated_0}, we have 
\[
\forall x>0,\qquad \P\bigg(\sup_{t_2-t_1\in[2^s/\nu,2^{s+1}/\nu]}\frac{Z_{t_1:t_2}}{\sqrt{t_2-t_1}}>2x\bigg)\leq e^{-\frac{x^2}{4}}\enspace.
\]
Let $T_s=[t_1+2^s/\nu,t_1+2^{s+1}/\nu]$. By a union bound, this yields, for any $x>0$,
\begin{equation}\label{eq:UBLLN}
 \P\bigg(\exists s\geq 0 :\sup_{t_2\in T_s}\frac{Z_{t_1:t_2}}{\sqrt{t_2-t_1}}>2\sqrt{\log[(s+1)(s+2)]+x}\bigg)\leq \sum_{s=0}^{+\infty} \frac{e^{-x}}{(s+1)(s+2)}= e^{-x}\enspace.
\end{equation}
For any $t_1$, $t_2$ such that $t_2-t_1\geq 2^s/\nu$, $s\leq \log[(t_2-t_1)\nu]/\log 2$, hence
\[
(s+1)(s+2)\leq \frac{\log[2(t_2-t_1)\nu]\log[4(t_2-t_1)\nu]}{(\log 2)^2}\leq\bigg(\frac{\log[\sqrt{8}(t_2-t_1)\nu]}{\log 2}\bigg)^2\enspace.
\]
Thus,
\[
\log[(s+1)(s+2)]\leq 2\log\bigg(\frac{\log[3(t_2-t_1)\nu]}{\log 2}\bigg)\leq2\log\log[3(t_2-t_1)\nu]+1\enspace.
\]
Plugging this inequality into \eqref{eq:UBLLN} shows that, with probability larger than $1-e^{-x}$, for any $t_2\geq t_1+1/\nu$, 
\begin{align*}
Z_{t_1:t_2}&\leq  2\sqrt{t_2-t_1}(\sqrt{2\log\log[3(t_2-t_1)\nu]+1+x})\enspace .
\end{align*}
\end{proof}

\begin{proof}[Proof of Lemma \ref{lem:log_itere1}]
By definition, 
\[
\bN(t_{\tau})= \sqrt{\frac{n+1-\tau}{n(\tau-1)}}Z_{1:\tau} - Z_{\tau:n+1}\sqrt{\frac{\tau-1}{(n+1-\tau)n}}\enspace.
\] 
We control the deviations of $\bN(t_{\tau})$ simultaneously for all $\tau\leq n/2+1$ with probability higher than $1-6e^{-x}$. 
By symmetry between $\tau$ and $n+1-\tau$, the desired result follows. 
Applying Lemma~\ref{lem:supsumGauss_log_iterated_0} with $\alpha=1$ and $d=n/2$, we  show that, with probability at least $1-e^{-x}$, for all $\tau\leq n/2+1$, 
\beq\label{eq:control_1}
- Z_{\tau:n+1}\sqrt{\frac{\tau-1}{(n+1-\tau)n}}\leq 2\sqrt{\frac{\tau-1}{n}x} \enspace.
\eeq
Fix now $\alpha>0$ and $s\geq 0$.
By Lemma~\ref{lem:supsumGauss_log_iterated_0} applied with $d=(1+\alpha)^s$, we derive that, for any $x\geq 0$, 
\[
\P\left[\forall \tau\in [(1+\alpha)^{s},(1+\alpha)^{s+1}] ,\quad \frac{Z_{1:\tau}}{\sqrt{\tau-1}}\leq \sqrt{2(1+\alpha) x } \right]\geq 1 -e^{-x}\enspace.
\]
A union bound shows then that, for any $x>0$, 
\[
\forall s\geq 0,\ \forall \tau\in [(1+\alpha)^{s},(1+\alpha)^{s+1}] ,\qquad \frac{Z_{1:\tau}}{\sqrt{\tau-1}}\leq \sqrt{2(1+\alpha) (x+\log[(1+\alpha^{-1})(s+1)^{1+\alpha}]) } \enspace,
\]
with probability at least
\[
1 -\sum_{s\geq 0}e^{-(x+\log[(1+\alpha^{-1})(s+1)^{1+\alpha}])}=1-\frac{\alpha e^{-x}}{1+\alpha}\sum_{s\geq 0}\frac1{s^{1+\alpha}}\geq 1-e^{-x}\enspace.
\]
Now, for any $\tau\in[(1+\alpha)^{s},(1+\alpha)^{s+1}]$, we have $s\leq \log(\tau)/\log(1+\alpha)$. As a consequence, we have 
\[
s+1\leq \frac{\log[(1+\alpha)\tau]}{\log(1+\alpha)}\enspace.
\]
It follows that, with probability at least $1-e^{-x}$, for any $\tau\geq 1$,
\begin{align*}
 \frac{Z_{1:\tau}}{\sqrt{\tau-1}}&\leq \sqrt{2(1+\alpha) \bigg(x+\log\bigg[(1+\alpha^{-1})\bigg(\frac{\log[(1+\alpha)\tau]}{\log(1+\alpha)}\bigg)^{1+\alpha}\bigg]\bigg)}\\
&\leq \sqrt{2(1+\alpha) \bigg(x+\log(1+\alpha^{-1})+(1+\alpha)\log\bigg[\frac{\log[(1+\alpha)\tau]}{\log(1+\alpha)}\bigg]\bigg)}\\
&\leq (1+\alpha)\sqrt{2(\log\log [(1+\alpha)\tau]+x+C_\alpha)}\enspace.
\end{align*}
Here
\[
C_\alpha=\frac{\log(1+\alpha^{-1})}{1+\alpha}-\log\log[1+\alpha]\enspace.
\]
Combined with \eqref{eq:control_1}, with probability $1-2e^{-x}$, we proved that, for any $\tau\leq n/2+1$,
\begin{align*}
 \bN(t_\tau)&\leq (1+\alpha)\sqrt{2(\log\log [(1+\alpha)\tau]+x+C_\alpha)}\sqrt{\frac{n+1-\tau}{n}}+2\sqrt{\frac{\tau-1}nx}\\
 &\leq (1+\alpha)\sqrt{2(\log\log [(1+\alpha)\tau]+3x+C_\alpha)}\enspace.
\end{align*}
Here, we used the inequality
\[
(\sqrt{1-u}a+\sqrt{u}b)^2=(1-u)a^2+2\sqrt{u(1-u)}ab+u b^2\leq a^2+b^2,
\]
which holds for any $u\in [0,1]$, $a, b>0$ with 
\[
a=(1+\alpha)\sqrt{2(\log\log [(1+\alpha)\tau]+x+C_\alpha)},\quad b=2\sqrt{x},\quad u=\frac{\tau-1}n\enspace.
\]
The term $\log\log(en/\tau)$ is obtained similarly by applying Lemma~\ref{lem:supsumGauss_log_iterated_0} to bound the deviations of $Z_{1:\tau}$ on the intervals $[n/(1+\alpha)^{s+1}; n/(1+\alpha)^{s}]$. 
Combining both results shows that, with probability larger than $1-3e^{-x}$, for all $\tau\leq n/2+1$
\[
\bN(t_{\tau})\leq (1+\alpha)\sqrt{2\bigg(\log\log\bigg [(1+\alpha)\bigg(\tau\wedge\frac{n}\tau\bigg)\bigg]+3x+C_\alpha\bigg)} \enspace.
\]
The result follows. 

\end{proof}

 \begin{proof}[Proof of Lemma \ref{lem:controle_uniforme_N_ratio}]
Recall that 
\beq\label{eq:Ntau} 
\bN(t_{\tau})= \sqrt{\frac{(\tau-1)(n+1-\tau)}{n}}\left(\frac{Z_{1:\tau}}{\tau-1}+\frac{Z_{\tau:n+1}}{n+1-\tau}\right)\enspace .
\eeq 
Let us first apply Lemma~\ref{lem:supsumGauss_log_iterated_0}. For any $s\geq 0$, with probability $1-e^{-x}/(s+1)(s+2)$, we have 
\beq\label{eq:Z1}
\frac{Z_{1:\tau}}{(\tau-1)^{1/2}}\leq 2\sqrt{\log[(s+1)(s+2)]+x}\enspace , 
\eeq
uniformly over all $\tau \in[\tau^*/2^{s+1}, \tau^*/2^s]$.
For any $\tau\leq \tau^*/2^s$, $s\leq \log(\tau^*/\tau)/\log 2$, hence,
\[
(s+1)(s+2)\leq \frac{\log(e\tau^*/\tau)\log(e^2\tau^*/\tau)}{(\log 2)^2}\leq \bigg(\frac{\log(e^{3/2}\tau^*/\tau)}{\log 2}\bigg)^2\enspace.
\]
Hence, by the inequality $\log(a+x)\leq \log a+\log (1+x)$ valid for any $x\geq 0$ and $a\geq 1$,
\begin{align}
\notag\log[(s+1)(s+2)]&\leq 2\log\big(3/2+\log[\tau^*/\tau]\big)-2\log\log 2\\
\notag&\leq 2\log[3/(2\log 2)]+2\log\log(e\tau^*/\tau)\\
\label{eq:Z2}&\leq 1.6+2\log\log\bigg(e\frac{\tau^*-1}{\tau-1}\bigg)\enspace. 
\end{align}
Moreover, for any $\tau\leq \tau^*$, we have $n+1-\tau^*\leq n+1-\tau$ and 
\[
\frac{\tau^*-1}{\tau-1}=\gamma_\tau^{-1}\frac{n+1-\tau^*}{n+1-\tau}\leq\gamma_\tau^{-1}\enspace.
\]
Applying a union bound over all non-negative integers $s$, we derive that,  with probability higher than $1-e^{-x}$, simultaneously over all $\tau\leq \tau^*$,
\begin{eqnarray}
\frac{Z_{1:\tau}}{(\tau-1)^{1/2}}&\leq& 2 \sqrt{2\log\log\left(e\gamma_\tau^{-1}\right)+x+1.6}  \enspace.\label{eq:2}
\end{eqnarray}

\medskip

We proceed similarly for $Z_{\tau:n+1}/(n+1-\tau)$.
It follows from Lemma~\ref{lem:supsumGauss_log_iterated_0} that, for any $s\geq 0$, with probability $1-e^{-x}/(s+1)(s+2)$, uniformly over all $\tau \in[2^s(n+1-\tau^*),2^{s+1}(n+1- \tau^*)]$,
\beq\label{eq:Z3}
\frac{Z_{\tau:n+1}}{(n+1-\tau)^{1/2}}\leq 2\sqrt{\log[(s+1)(s+2)]+x}\enspace.
\eeq
Arguing exactly as when we deduced \eqref{eq:2} from \eqref{eq:Z1}, we obtain that, simultaneously over all $\tau\leq \tau^*$,
\begin{eqnarray}
\frac{Z_{\tau:n+1}}{(n+1-\tau)^{1/2}}&\leq& 2 \sqrt{2\log\log\left(e\gamma_\tau^{-1}\right)+x+1.6}  \label{eq:3}\enspace , 
\end{eqnarray}
with probability higher than $1-e^{-x}$. 
Plugging \eqref{eq:2} and \eqref{eq:3} into \eqref{eq:Ntau}, we get 
\[
\bN(t_\tau)\leq \frac{2}{\sqrt{n}} \sqrt{2\log\log\left(e\gamma_\tau^{-1}\right)+x+1.6}(\sqrt{\tau-1}+\sqrt{n+1-\tau})\leq 4\sqrt{\log\log\left(e\gamma_\tau^{-1}\right)+x+1}\enspace.
\]
Here, we used for the last inequality that $\sqrt{a}+\sqrt{n-a}\leq \sqrt{2n}$ for any $a\in [0,n]$.

 \end{proof}

\subsection{Proof of Lemmas \ref{lem:concentration:N_t} and \ref{lem:control_event_A}} \label{sec:proof:concentration:Nt}
We first establish Lemma~\ref{lem:concentration:N_t}. For any $t\in \cT_3$, define 
\[
 X_t= \frac{(t_3-t_2)(t_2-t_1)}{(t_3-t_1)}\left[\frac{Z_{t_1:t_2}}{t_2-t_1}  - \frac{Z_{t_2:t_3}}{t_3-t_2}\right]\enspace ,
\]
so that $Z_t= \sqrt{\frac{(t_3-t_2)(t_2-t_1)}{t_3-t_1}}X_t$. In the sequel, we write $\delta_1$ and $\delta_2$ for $t_2-t_1$ and $t_3-t_2$ respectively.

Given a centered sugGaussian random variable $Y$, it sub-Gaussian norm $\|Y\|_{\psi_2}$ is defined the smallest $\sigma$ such that $\E[e^{sY}]\leq e^{s^2\sigma^2/2}$ for all $s\in \mathbb{R}$. 
Defining $\sigma^2(t)=\frac{\delta_1\delta_2}{\delta_1+\delta_2} \in [(\delta_1\wedge \delta_2)/2,\delta_1\wedge \delta_2)$, we observe that 
$Z_t=X_t/\sigma(t)$ and that the sub-Gaussian norm of $X_t$ is less or equal to $\sigma(t)$.

We prove Lemma~\ref{lem:concentration:N_t} using an adaptive  peeling argument as in D\"umbgen and Spokoiny~\cite{dumbgen2001multiscale}.
Denote $\cT_{+}$ (resp. $\cT_{-}$)  the subset of vectors $t\in  \cT_3$ satisfying $\delta_1\leq \delta_2$ (resp. $\delta_2\leq \delta_1$).  We focus on triads $t\in \cT_{+}$, triads in $\cT_{-}$ being handled analogously.
 Given any $t\neq t'$ in $\cT_{+}$, define 
 \beq\label{eq:definition_rho}
 \rho^2(t,t')= |t_1-t'_1|+ |t_2-t'_2|+ |t_3-t'_3|\left	(\frac{\delta_1}{\delta_2}+ \frac{\delta'_1}{\delta'_2}\right)^2 \enspace .
 \eeq
 The next lemma bounds the deviations of $X_{t}-X_{t'}$ in terms of $\rho(t,t')$.
  \begin{lem}\label{lem:rho_var}
 For any $t$, $t'\in \cT_{+}$, we have 	
 \beq\label{eq:var_Xt}
 \|X_t-X_{t'}\|_{\psi_2}^2 \leq \frac{7}{2} \rho^2(t,t')\enspace . 
 \eeq
 \end{lem}
In other words, $X_t-X_{t'}$ is small if $(t_1,t_2)$ is close to $(t'_1,t'_2)$ and either if $|t_3-t'_3|$ is small or if $\delta_1/\delta_2$ and $\delta'_1/\delta'_2$ are small.

To formalize our adaptive peeling argument, we partition $\cT_+$ into the collections
\[
\cT_+^{(k,q)}:=\{t\in \cT_{+}, \delta_1\in [n2^{-k},n2^{-k+1}),\quad \delta_2\in [n2^{q-k}n,2^{q-k+1}\} \enspace ,
\]
where  $k=1, \ldots, \lceil \log_2(n)\rceil$ and $q=0,\ldots, k-1$. Given any $\kappa\in (0, 1)$, we denote $\cS_{k,q,\kappa}$ a minimal covering of $\cT_+^{(k,q)}$ of radius $ r_{k,\kappa}:= \kappa \sqrt{n 2^{-k+1}}$ with respect to the semi-metric $\rho$. 
\begin{lem}\label{lem:covering_number_gamma}
 For any $\kappa \in (0,1)$ and positive integer $k\leq \lceil \log_2(n)\rceil$, we have 
 \beq
 |\cS_{k,q,\kappa}|\leq c \cdot\frac{2^{k-q}}{\kappa^6}\enspace , 
 \eeq
 for a numerical constant $c>0$. 
\end{lem}
For $t\in \cT_+^{(k,q)}$, we write $\pi_{k,q,\kappa}(t)$ for any closest element of $t$ in $\cS_{k,q,\kappa}$. 
Given any  integer $k\leq \lceil \log_2(n)\rceil$, define $l(k) := \lceil \tfrac{\log(\log(n))}{\log(k)}\rceil$.  Using a chaining  argument with the  collections $(\cS_{k,q,k^{-i}})$, $i=1,\ldots, l(k)$, we arrive at the decomposition 
\beq \label{eq:chaining}
X_t = X_{\pi_{k,q,k^{-1}}(t)}+ \sum_{i=1}^{l(k)-1} \left(X_{\pi_{k,q,k^{-i-1}}(t)}- X_{\pi_{k,q,k^{-i}}(t)}\right) + X_{t}- X_{\pi_{k,q,k^{-l(k))}}(t)}\enspace .
\eeq
Let us first control the deviation of $X_{t}- X_{\pi_{k,q,k^{-l(k))}}(t)}$. Since $|\cT_{+}|\leq n^{3}/6$, it then follows from an union bound over all $t,t'\in \cT_{+}$ and from \eqref{eq:var_Xt} that, with probability higher than $1-x$,
\beq\label{eq:upper_difference_X}
|X_t-X_{t'}|\leq \rho(t,t')\sqrt{42\log(n)+ 7\log(1/x)}\enspace , 
\eeq 
simultaneously over all $t$ and $t'$ in $\cT_+$.  Next, the random variables $X_{s}$, with $s\in \cS_{k,q,k^{-1}}$ are simply controlled using an union bound. With probability higher than $1-x$, we have
\beq\label{eq:upper_difference_X_1}
 |X_{\pi_{k,q,k^{-1}}(t)} |\leq \sigma(\pi_{k,q,k^{-1}}(t))\sqrt{2\log\Big(\frac{2|S_{k,q,k^{-1}}|}{x}\Big)}\enspace ,
\eeq
simultaneously over all $t\in \cT_+^{(k)}$.  Finally, we consider the differences  $X_{\pi_{k,q,k^{-i-1}}(t)}- X_{\pi_{k,q,k^{-i}}(t)}$ for  $1\leq i\leq l(k)$. For a fixed $i$, there are at most $|\cS_{k,q,k^{-i-1}}| | \cS_{k,q,k^{-i}}|\leq |\cS_{k,q,k^{-i-1}}|^2$ such differences.
Taking an upper bound over all possible $X_{\pi_{k,q,k^{-i-1}}(t)}- X_{\pi_{k,q,k^{-i}}(t)}$ with $t\in \cT_+^{(k)}$, we conclude that, with probability higher than $1-x$
\beq
|X_{\pi_{k,q,k^{-i-1}}(t)}- X_{\pi_{k,q,k^{-i}}(t)}|\leq \rho(\pi_{k,q,k^{-i-1}}(t),\pi_{k,q,k^{-i}}(t)) \sqrt{7 \log\left(2\frac{|\cS_{k,q,k^{-i-1}}|^2}{x}\right)} \enspace ,
\eeq
simultaneously for all $t\in \cT_+^{(k)}$. Then taking an union bound over all $i\leq l(k)-1$ with weight $18x/(\pi^2i^2)$ and gathering it \eqref{eq:upper_difference_X} and \eqref{eq:upper_difference_X_1}, we conclude that with probability higher than $1-x$, we have 
\beqn 
|X_t|&\leq&  \sigma(\pi_{k,q,k^{-1}}(t))\sqrt{2\log\Big(\frac{6|S_{k,q,k^{-1}}|}{x}\Big)} +   \sum_{i=1}^{l(k)-1} \rho(\pi_{k,q,k^{-i-1}}(t),\pi_{k,q,k^{-i}}(t))\sqrt{7\log\Big(\frac{36|S_{k,q,k^{-i-1}}| i^2}{\pi^2x}\Big)} \\ && +  \rho(t,\pi_{k,q,k^{-l}}(t)) \sqrt{42\log(n)+7\log\left(\frac{3}{x}\right   )}\enspace ,
\eeqn 
simultaneously over all $t\in \cT_{+}^{(k,q)}$. For any such $t, t'\in \cT_{+}^{(k,q)}$, we have
\[
|t_1-t'_1|+ |t_2-t'_2|+ 2^{2-2q}|t_3-t'_3| \leq \rho^2(t,t')\leq  |t_1-t'_1|+ |t_2-t'_2|+ 2^{4-2q}|t_3-t'_3|
\]
Hence, we get 
\[
\rho^2(\pi_{k,q,k^{-i-1}}(t),\pi_{k,q,k^{-i}}(t))\leq 
4 \big[\rho^2(\pi_{k,q,k^{-i-1}}(t),t)+ \rho^2(\pi_{k,q,k^{-i}}(t),t)\big]\leq 8 \rho^2(\pi_{k,q,k^{-i}}(t),t)\leq 8k^{-2i}n2^{-k+1} \enspace , 
\]
  and $\sigma(t)\geq \sqrt{\delta_1/2}\geq \sqrt{n2^{-k-1}}$. Then, relying on Lemma~\ref{lem:covering_number_gamma}, we conclude that 
 \beqn 
 \frac{|X_t|}{\sigma(t)}&\leq & \frac{\sigma(\pi_{k,q,k^{-1}}(t))}{\sigma(t)}\sqrt{2\log(2^{k-q})+ 12\log(k) + c+ 2\log(1/x)}  \\ && + c'\sum_{i=1}^{\infty} k^{-i}\sqrt{k-q+ \log(i) + i\log(k)+ 1+\log(1/x)} + c'' \frac{\sqrt{\log(1/x)+1}}{\sqrt{\log(n)}}\\
 &\leq & \sqrt{2\log(2^{k-q})+ c_1\log(k)+ c_2 + c_3\log(\frac{1}{x})}\\ &&+ \frac{|\sigma(\pi_{k,k^{-1}}(t))-\sigma(t)|}{\sigma(t)}\sqrt{2k+ 12\log(k) + c+ 2\log(1/x)}\enspace .
 \eeqn 
\begin{lem}\label{lem:comparaison_variance}
For $\kappa\in (0,1)$ and $t\in \cT_+^{(k)}$, we have
\[
\frac{|\sigma(t) - \sigma(\pi_{k,q,\kappa}(t))|}{\sigma(t)} \leq \frac{2 r_{k,\kappa}^2}{\sigma^2(t)} \leq 8 \kappa^2\enspace .
\] 
\end{lem}
With this last lemma, we obtain
\[
 \frac{|X_t|}{\sigma(t)}
 \leq  \sqrt{2\log(2^{k-q})+ c_1\log(k)+ c_2 + c_3\log(\frac{1}{x})}\enspace , 
 \]
simultaneously over all $t\in \cT_+^{(k,q)}$ with probability higher than $1-x$. Taking an union bound over all $k$ and $q$ with weights $36x/[\pi^2 k^2q^2]$, we conclude that,  with probability higher than $1-x$,
\[
 \frac{|X_t|}{\sigma(t)}
 \leq  \sqrt{2\log(2^k)+ c'_1\log(k)+ c'_2 + c'_3\log(\frac{1}{x})}\enspace , 
 \]
simultaneously over all $t\in \cT_+^{(k,q)}$ with $k=1,\ldots, \lceil \log_2(n)\rceil$ and $q=1,\ldots, k$. 
Since $t\in \cT_+^{(k,q)}$, we have  $2^{k}\leq 2n/\delta_1$, which implies 
\[
 \frac{|X_t|}{\sigma(t)}
 \leq  \sqrt{2\log\left(\frac{n(\delta_1+\delta_2)}{\delta_1\delta_2 }\right)+ c'_1\log\log\left(\frac{n(\delta_1+\delta_2)}{\delta_1\delta_2 }\right)+ c'_2 + c'_3\log(\frac{1}{x})}\enspace .  
 \]

\begin{proof}[Proof of Lemma \ref{lem:rho_var}]
 By symmetry, we assume that $\delta_2\geq \delta'_2$.  Let us upper bound $\|X_t-X_{t'}\|_{\psi_2}^2$. Since $X_t-X_t'$ decomposes as $\sum_{i=t_1\wedge t'_1}^{(t_3\vee t'_3)-1}\alpha_i \epsilon_i$ for some $\alpha_i$'s, its squared sub-Gaussian norm is at most $\sum_{i}\alpha_i^2$. The value of the $\alpha_i$'s depends on  whether $i$ belongs to $[t_1,t_2)$, $[t_2,t_3)$, $[t'_1,t'_2]$ or $[t'_2,t'_3)$. More precisely, we have 
 \begin{itemize}
  \item[(a)]  $|\alpha_i|\leq 1$ if $i\in ([t_1,t_2)\cap [t'_1,t'_3)^c)\cup ([t'_1,t'_2)\cap [t_1,t_3)^c)$
  \item[(b)] $|\alpha_i|\leq (\delta_1/\delta_2+\delta'_1/\delta'_2)$ if $i\in ([t_2,t_3)\cap [t'_1,t'_3)^c)\cup ([t'_2,t'_3)\cap [t_1,t_3)^c)$
  \item[(c)] $|\alpha_i|\leq 3/2$ if $i\in ([t_1,t_2)\cap [t'_2,t'_3))\cup ([t'_1,t'_2)\cap [t_2,t_3))$
  \item[(d)] For $i\in ([t_1,t_2)\cap[t'_1,t'_2))\cup ([t_2,t_3)\cap[t'_2,t'_3))$, we have 
 \[
 \alpha_i^2 = \left[ \frac{\delta'_1}{\delta'_1+\delta'_2}- \frac{\delta_1}{\delta_1+\delta_2}\right]^2\leq \left[\frac{(\delta_2-\delta'_2)\delta'_1+\delta'_2|\delta'_1-\delta_1|}{(\delta'_1+\delta'_2)(\delta_1+\delta_2)}\right]^2\enspace . 
 \]
If $\delta_2=\delta'_2$, the above expression is less or equal to $|\delta_1-\delta'_1|$. Now assume that $\delta_2>\delta'_2$. 
Using $(a+b)^2\leq (1+x)a^2+(1+x^{-1})b^2$ with $x= \frac{\delta_1+\delta'_2}{\delta_2-\delta'_2}$ yields 
 \beqn 
 \alpha_i^2  &\leq& \frac{|\delta_2-\delta'_2|}{\delta_1+\delta_2}\left(\frac{\delta'_1}{\delta'_1+\delta'_2}\right)^2 + 
 \frac{|\delta_1-\delta'_1|^2 \delta^{'2}_2}{(\delta'_1+\delta'_2)^2(\delta_1+\delta_2)(\delta_1+\delta'_2)}\\
  &\leq & \frac{|\delta_2-\delta'_2|}{\delta_1+\delta_2}\left(\frac{\delta'_1}{\delta'_1+\delta'_2}\right)^2 + \frac{|\delta_1-\delta'_1|}{\delta_1+\delta_2} \enspace .
 \eeqn 
 \end{itemize}

There are at least $\delta_1+\delta_2$ indices of type (d). If $[t_1,t_3)\cap[t_1, t'_3)\neq \emptyset$, then there are at most $|t_1-t'_1|$ indices of type (a), $|t_3-t'_3|$ indices of type (b) and $|t_2-t'_2|$ indices of type (c). This leads us to 6
\beqn 
 \|X_t-X_{t'}\|_{\psi_2}^2&\leq&   |t_1-t'_1|+ \frac{9}{4}|t_2-t'_2|+ |t_3-t'_3|\left(\frac{\delta_1}{\delta_2}+ \frac{\delta'_1}{\delta'_2}\right)^2+ |\delta_2-\delta'_2|\left(\frac{\delta'_1}{\delta'_1+\delta'_2}\right)^2 + |\delta_1-\delta'_1|\\
 &\leq & \frac{7}{2}\rho^2(t,t')\enspace . 
\eeqn 
Now assume that $[t_1,t_3)\cap[t_1, t'_3)= \emptyset$. For instance, assume that $t_3 \leq t'_1$. Then, 
\beqn 
\|X_t-X_{t'}\|_{\psi_2}^2&\leq& \delta_1+\delta'_1+ (\delta_2+\delta'_2)\left(\frac{\delta_1}{\delta_2}+ \frac{\delta'_1}{\delta'_2}\right)^2\\
&\leq& \left[ \delta_1+ 2\delta_2\right] +\left[\delta'_1+ 2\delta_2\right] + (t_3-t'_3)\left(\frac{\delta_1}{\delta_2}+ \frac{\delta'_1}{\delta'_2}\right)^2\\
&\leq & 2\rho^2(t,t')\ , 
\eeqn 
where we used that $|t'_1-t_1|\geq \delta_1+\delta_2$ and $|t'_2-t_2|\geq \delta'_1+\delta_2$. 
\end{proof}

\begin{proof}[Proof of Lemma \ref{lem:covering_number_gamma}]
Let us upper bound $|\cS_{k,q,\kappa}|$ by building a covering subset $\cS'_{k,q,\kappa}$ of $\cT_+^{(k,q)}$. First, we consider the case where $q\geq 3$. 
Take a regular subgrid (containing $n$) of $\{1,\ldots, n\}$ with radius $\lceil  r_{k,\kappa}^2/6\rceil$. Then, given any $t_1$ on this grid, we build a regular grid (containing $t_1+\lfloor n2^{-k+1}\rfloor$) of $[t_1+ n2^{-k}, t_1+ n2^{-k+1}]$ with  radius $\lceil  r_{k,\kappa}^2/6\rceil$. Finally, given any such $t_1$ and $t_2$, we construct a regular grid (containing $t_2+\lfloor n 2^{q+1-k}\rfloor$ and $t_2+\lfloor n 2^{q-k}\rfloor$ )  of $[t_2+2^{q-k} n, t_2+ n 2^{q+1-k}]$ with radius $ \lceil r_{k,\kappa}^2 2^{2q-5}\rceil$. 
\[
|\cS'_{k,q,\kappa}|\leq \frac{n}{\lceil r^2_{k,\kappa}/6\rceil}\cdot \frac{n 2^{-k+1}}{\lceil r^2_{k,\kappa}/6\rceil } \cdot  \frac{n2^{q-k}}{\lceil r^2_{k,\kappa}  2^{2q-5}\rceil} \leq  c\frac{ 2^{k-q}}{\kappa^6}\enspace .
\]
It remains to check that $\cS'_{k,q,\kappa}$ is a covering subset. Consider any $t\in \cT_+^{(k,q)}$ and define $\overline{t}\in \cS'_{k,q,\kappa}$ in such a way that $\overline{t}_1- t_1>0$ is the smallest possible, then  $\overline{t}_2-t_2>0$ is the smallest possible and finally $\overline{t}_3\geq t_3$ is the smallest possible.  Obviously, we have  $|\overline{t}_1-t_1|\leq r^2_{k,\kappa}/6$, $|\overline{t}_2-t_2|\leq  r_{k,\kappa}^2/3$. Since $|t_3-t_2|\geq 2^{q-k} n\geq 2^{q}r^{2}_{k,\kappa}/2$, we have $t_3\geq \overline{t}_2$. 
 $t_3\in [\overline{t}_2+ 2^{q-k}n, \overline{t}_2+ 2^{q-k+1}n]$ so that $|\overline{t}_3-t_3|\leq  (r_{k,\kappa}^2 2^{2q-5}) \vee  (r_{k,\kappa}^2 /3)\leq r_{k,\kappa}^2 2^{2q-5}$ since $q\geq 3$. We have
\beqn 
\rho(t,\overline{t})&\leq& \frac{ r_{k,\kappa}^2}{2}+ r_{k,\kappa}^2 2^{2q-5}  \left[\frac{\delta_1}{\delta_2}+ \frac{\ol{\delta}_1}{\ol{\delta}_2}\right]^2\leq  \frac{ r_{k,\kappa}^2}{2}+ \frac{r_{k,\kappa}^2}{2} \leq r_{k,\kappa}^2\enspace ,
\eeqn 
since both $\ol{\delta}_1/\ol{\delta}_2$ and $\delta_1/\delta_2$ are at most $2^{-q+1}$. 
This proves that $|\cS_{k,q,\kappa}|\leq c\frac{2^{k-q}}{\kappa^6}$ for any $q\geq 3$. It remains to consider the case $q\leq 2$. We build the same covering subsets as $\cS'_{k,q,\kappa}$ above except the radius of each of the subgrid is at most $\lceil  r_{k,\kappa}^2/16\rceil$. One easily checks that $|\cS'_{k,q,\kappa}|\leq c' 2^{k}\kappa^{-6}\leq 4c'2^{k-q}\kappa^{-6}$. Then, given $t\in \cT_+^{(k,q)}$, we build $\ol{t}$ as above. One easily checks that $|t_1-\ol{t_1}|\leq r_{k,\kappa}^2/16$, $|t_2-\ol{t_2}|\leq r_{k,\kappa}^2/8$ and that $|t_2-\ol{t_2}|\leq r_{k,\kappa}^2/8$
Since $\rho(t,\overline{t})\leq |t_1-\ol{t_1}|+|t_2-\ol{t_2}|+4|t_3-\ol{t_3}|\leq r_{k,\kappa}^2$, we obtain the desired result. 

\end{proof}

\begin{proof}[Proof of Lemma \ref{lem:comparaison_variance}]
Recall that $ \sigma^2(t)= \delta_1\delta_2/(\delta_1+\delta_2)$ and write for short $\overline{t}= \pi_{k,\kappa}(t))$ and $(\ol{\delta}_1,\ol{\delta}_2)$ the corresponding segment length.
\beqn 
|\sigma^2(t) - \sigma^2(\overline{t})|&= &\Big|\delta_1\frac{\delta_2}{\delta_1+\delta_2} - \ol{\delta}_1\frac{\ol{\delta}_2}{\ol{\delta}_1+\ol{\delta}_2}\Big|\leq |\delta_1-\ol{\delta}_1|+ \delta_1 \frac{\ol{\delta} _1 |\delta_2 -\ol{\delta}_2 | + \ol{\delta}_2 |\delta_1 - \ol{\delta}_1|}{(\delta_1+\delta_2)(\ol{\delta}_1+\ol{\delta}_2)}\\
&\leq & 2\left[ |t_1- \overline{t}_1| + |t_2- \overline{t}_2|\right]+ \frac{\delta_1\ol{\delta}_1}{\delta_2\ol{\delta}_2} |t_3- \overline{t}_3|\\ 
&\leq &2\rho^2(t,\overline{t})\leq 2 r_{k,\kappa}^2\enspace ,
\eeqn 
where we used in the second line that $|\delta_j-\delta'_j|\leq |t_{j}-t'_{j}|+|t_{j}-t'_{j+1}|$ for $j=1,2$.
Since $\sigma^2(t)\geq (\delta_1\wedge \delta_2)/2\geq n2^{-k-1}$, we obtain
\[
\frac{|\sigma(t) - \sigma(\pi_{k,\kappa}(t))|}{\sigma(t)} \leq \frac{2 r_{k,\kappa}^2}{\sigma^2(t)} \leq 8 \kappa^2\enspace .
\]
\end{proof}

\begin{proof}[Proof of Lemma \ref{lem:control_event_A}]
Fix $\ell>1$.  For any $a>0$, $b>0$ and any $z>0$, we have $b\log(z)\leq  az + b\log[(b/ae)\vee e]$. Applying Lemma \ref{lem:concentration:N_t}, we deduce that, uniformly over all $t\in \cT$, one has
\[
 \bN^2(t)\leq   2\ell^2 \log\left(\frac{n(\delta_1(t)+\delta_2(t)}{\delta_1(t)\delta_2(t) }\right)+ (c_1+  c_2)\log\left(\frac{1}{x}\right) + c_3+ c_1\log\left(\frac{c_1}{\ell-1} \vee e\right)\enspace ,
\]
 with probability higher than $1-x$. Then, taking $x= \exp[-q/(2(c_1+c_2))]$ and assuming that $q\geq 2[c_3+ c_1\log(\tfrac{c_1}{\ell-1} \vee e]$, we conclude that $\mathbb{P}\left[\cA_{L,q}\right]\geq 1 - e^{-c'q}$. 
\end{proof}

\paragraph{Acknowledgements}
We are grateful to Guillem Rigaill for many stimulating discussions.

\bibliography{biblio}
\bibliographystyle{plain}

\end{document}